\documentclass[aap]{imsart}

\RequirePackage{amsthm,amsmath,amsfonts,amssymb}
\RequirePackage[numbers]{natbib}
\RequirePackage[colorlinks,citecolor=blue,urlcolor=blue]{hyperref}

\RequirePackage{enumerate}

\startlocaldefs
\numberwithin{equation}{section}
\theoremstyle{plain}
\newtheorem{lemma}{Lemma}
\newtheorem{proposition}{Proposition}
\newtheorem{theorem}{Theorem}

\theoremstyle{remark}
\newtheorem{remark}{Remark}
\newtheorem{definition}{Definition}
\numberwithin{lemma}{section}
\numberwithin{proposition}{section}
\numberwithin{theorem}{section}
\numberwithin{corollary}{section}
\numberwithin{definition}{section}
\numberwithin{remark}{section}
\newcommand{\del}{\partial}
\renewcommand{\div}{\on{div}}
\newcommand{\eps}{\varepsilon}
\newcommand{\Id}{\on{Id}}
\newcommand{\ind}{\mathbf{1}}
\newcommand{\loc}{\mathrm{loc}}
\newcommand{\nor}[2]{\left\| #1 \right\|_{#2} }
\newcommand{\norm}[1]{ \left\| #1 \right\| }
\newcommand{\oline}[1]{\overline{#1}}
\newcommand{\oo}{\infty}
\newcommand{\sgn}{\on{sgn}}
\newcommand{\supp}{\on{supp }}
\newcommand{\tr}{\on{tr}}

\newcommand{\EE}{{\mathbb E}}
\newcommand{\E}{{\mathbb E}}
\newcommand{\NN}{{\mathbb N}}

\newcommand{\PP}{{\mathbb P}}
\newcommand{\FF}{{\mathbb F}}

\newcommand{\RR}{{\mathbb R}}
\newcommand{\R}{{\mathbb R}}
\newcommand{\ZZ}{{\mathbb Z}}

\DeclareMathOperator*{\esssup}{ess\,sup}
\DeclareMathOperator*{\essinf}{ess\,inf}

\newcommand{\mcl}{\mathcal}
\newcommand{\mbb}{\mathbb}
\newcommand{\mbf}{\mathbf}
\newcommand{\on}{\operatorname}

\endlocaldefs

\begin{document}

\begin{frontmatter}
\title{Mean field games with common noise and degenerate idiosyncratic noise}
\runtitle{Mean field games with common noise and degenerate idiosyncratic noise}

\begin{aug}
\author[A]{\fnms{Pierre}~\snm{Cardaliaguet}\ead[label=e1]{cardaliaguet@ceremade.dauphine.fr}},
\author[B]{\fnms{Benjamin}~\snm{Seeger}\ead[label=e2]{bseeger@unc.edu}\orcid{0000-0003-4472-605X}}
\and
\author[C]{\fnms{Panagiotis}~\snm{Souganidis}\ead[label=e3]{souganidis@math.uchicago.edu}\orcid{0000-0003-2399-0759}}
\address[A]{CEREMADE, Universit\'e Paris--Dauphine\printead[presep={,\ }]{e1}}

\address[B]{Department of Statistics and Operations Research, University of North Carolina at Chapel Hill\printead[presep={,\ }]{e2}}

\address[C]{Department of Mathematics, University of Chicago\printead[presep={,\ }]{e3}}
\end{aug}

\begin{abstract}
We study the forward-backward system of stochastic partial differential equations describing a mean field game for a large population of small players subject to both idiosyncratic and common noise. The unique feature of the problem is that the idiosyncratic noise coefficient may be degenerate, so that the system does not admit smooth solutions in general. We develop a new notion of weak solutions for backward stochastic Hamilton-Jacobi-Bellman equations, and use this to build probabilistically weak solutions of the mean field game system. Under an additional monotonicity assumption, we prove the uniqueness of a strong solution.
\end{abstract}



\end{frontmatter}


\section{Introduction}
We investigate a notion of weak solution for the following stochastic Mean Field Games (MFG) system
\begin{equation}\label{e.MFGstoch}
\left\{ 
	\begin{split}
		&d  u_{t} = \bigl[ - \beta \Delta  u_{t} -\tr( a_t(x, m_t) D^2  u_t) + H_t(x,D u_{t}, m_{t}) -  {2 \beta} \div( v_{t}) \bigr] dt\\
		&\qquad \qquad \qquad +  v_{t} \cdot \sqrt{2\beta}dW_{t} \ \ \text{in} \ \ [0,T)\times \R^d , \\
		& d  m_{t} = \div \bigl[  \beta D  m_{t} + \div( a_t(x, m_t)  m_{t})+ m_{t} D_{p}  H_t(x,D u_{t}, m_{t}) 
 \bigr] dt \\
		&\qquad \qquad \qquad  - \div (  m_{t} \sqrt{2\beta} dW_{t} \bigr) \ \  \text{in} \ \   (0,T]\times \R^d, \\
		&u_T(x)= G(x,m_T) \ \ \text{and} \ \   m_{0}=\bar m_0 \ \  {\rm in } \ \  \R^d,
	\end{split}
\right.
\end{equation}
which describes a differential game with infinitely many small players subject to a common noise, more precisely a Wiener process $W$ with constant volatility $\beta > 0$, and an inhomogeneous idiosyncratic noise, the latter leading  to the terms involving the possibly degenerate diffusion matrix $a$. The system consists of a backward stochastic Hamilton-Jacobi-Bellman (HJB) equation satisfied by the value function $u$ of a typical player, and  a forward Fokker-Plank (FP) describing  the evolution of the population density $m$. The third unknown, the $\RR^d$-valued map $v$ in the backward HJB equation, ensures that $u$ is adapted to the filtration generated by the common noise $W$. The Hamiltonian $H: [0,T] \times \R^d\times \R^d\times \mathcal P(\R^d)\to \R$ (where $\mathcal P(\R^d)$ is the set of Borel probability measures on $\R^d$) is continuous and convex in the gradient variable, the map $ a:[0,T] \times \R^d \times \mcl P(\RR^d) \to \R^{d\times d}$ is continuous  with values in the set of symmetric nonnegative matrices, and $ G:\R^d\times \mathcal P(\R^d)$ is the terminal cost. More precise assumptions are listed later.

In \cite{CDLL}, Cardaliaguet, Delarue, Lasry and Lions proved the existence and uniqueness of classical solutions of \eqref{e.MFGstoch} set on the torus for a nondegenerate matrix $a$, $H$ with a separated dependence on $Du$ and $m$, and strictly monotone coupling functions. Cardaliaguet  and Souganidis  considered in \cite{CSmfg} the first-order version of \eqref{e.MFGstoch} set in the whole space when there is no idiosyncratic noise, that is, $a \equiv 0$ and $\beta > 0$.

The goal of the present note is to extend the results of  \cite{CDLL} and \cite{CSmfg} to the whole space with  nonzero but possibly degenerate $a$. As in \cite{CSmfg}, the main feature of the problem that distinguishes it from the existing literature on MFG systems with common noise is the nonexistence of classical solutions in general, a fact that necessitates the introduction of a suitable notion  of weak solution. Due to the second-order nature of the system, the methods of \cite{CSmfg} are no longer applicable, and must be substantially revisited.

The starting point in the analysis is the observation that, because all players are affected by the same noise $W$, the stochasticity can be cancelled out by shifting the density $m$ of the  players along this noise. Indeed performing the change of variables
\begin{equation}\label{changevariable}
\widetilde  u_t(x)= u_t(x+\sqrt{2\beta}W_{t}) \quad and \quad \widetilde  m_t= (\Id- \sqrt{2\beta}W_{t})_\sharp  m_t
\end{equation}
and setting, for $(t,x,p,m) \in [0,T] \times \RR^d \times \RR^d \times \mcl P(\RR^d)$,
\begin{equation}\label{pierre0}
\widetilde  a_t(x,m)=  a_t(x+\sqrt{2\beta} W_t,(\Id+\sqrt{2\beta}W_t)_\sharp m), 
\end{equation}
\begin{equation}\label{takis0} 
\widetilde  H_t(x,p,m)=  H_t( x+\sqrt{2\beta} W_t,p,(\Id+\sqrt{2\beta}W_t)_\sharp m) ,  
 \end{equation}
\begin{equation}\label{pierre02} \widetilde  G(x,m)= G(x+\sqrt{2\beta} W_T, (\Id+\sqrt{2\beta}W_T)_\sharp m),
\end{equation}
and 
\begin{equation}\label{martingale}
	\widetilde  v_t(x) = v_t(x-\sqrt{2\beta} W_t) - \sqrt{2\beta} Du_t(x - \sqrt{2\beta} W_t), \quad \text{and} \quad
	M_t = \int_0^t \widetilde  v_s dW_s,
\end{equation}
leads, at least formally, to the system
\begin{equation}\label{e.MFGstochBis}
\left\{
\begin{split}
 	&d \widetilde  u_{t} = \left[ -\tr(\widetilde  a_t(x, \widetilde  m_t) D^2 \widetilde  u_t) + \widetilde  H_t(x,D\widetilde  u_{t},\widetilde  m_{t})  \right] dt + dM_{t} \quad \text{in } [0,T)\times \R^d, \\
	&d \widetilde  m_{t} = \div \left[ \div(\widetilde  a_{t}(x,\widetilde  m_t) \widetilde  m_{t})+ \widetilde  m_{t} D_{p} \widetilde  H_t(x,D\widetilde  u_{t} , \widetilde  m_{t})
\right] dt \quad \text{in }  (0,T]\times \R^d, \\
	&\widetilde  u_T(x)= \widetilde  G(x,\widetilde  m_T), \ \ \   \widetilde  m_{0}=\bar m_0 \ \  \text{in} \ \  \R^d.
\end{split}
\right.
\end{equation}
Observe that the new functions $\widetilde  a$, $\widetilde  H$, and $\widetilde  G$ are random and adapted to the common noise $W$. Note also that, since they arise from $a$, $H$, and $G$ through translations along $W$, they retain whatever regularity and structural conditions were satisfied by the deterministic data uniformly over $m \in \mcl P$ (see Remark \ref{R:transform} below).

Crucially, the new forward FP equation in \eqref{e.MFGstochBis} is no longer stochastic, and, therefore, can be interpreted in the usual sense of distributions. On the other hand, the backward equation for $\widetilde  u$ still involves a stochastic correction term, namely, the martingale $M$.

\subsection{Informal discussion of the main results}

The first main task of the paper is to develop a useful notion of weak solution for backward stochastic HJB equations of the form
\begin{equation}\label{E:hjb}
	\left\{
	\begin{split}
	&d \widetilde  u_t = \left[ - \tr(\widetilde  a_t(x) D^2 \widetilde  u_t) + \widetilde  H_t(x,D\widetilde  u_t) \right]dt + dM_t  \quad \text{in } [0,T) \times \RR^d \quad \text{and}\\
	&\widetilde u_T = \widetilde  G \quad \text{in } \RR^d,
	\end{split}
	\right.
\end{equation}
where the unknowns $(\widetilde  u,M)$ and the data $\widetilde  a$, $\widetilde  H$, and $\widetilde  G$ are adapted to a common filtration and $M$ is a martingale taking values in a space of generalized functions (distributions).

Equations like \eqref{E:hjb} should typically be satisfied by the value function of optimal control problems with random coefficients adapted to a given filtration, as studied by Peng \cite{Pe92}, who proved existence and uniqueness of solutions for smooth data and uniformly elliptic $\widetilde  a$. They can also be studied with the theory of path-dependent viscosity solutions, a notion which involves taking derivatives on the path space; see for instance Qiu \cite{QIU18} and Qiu and Wei \cite{QiuWei}. This theory allows for the treatment of more general equations than \eqref{E:hjb}, for example, with non-constant volatility $\beta$, and, other than standard growth and regularity assumptions, few conditions are required for the Hamiltonian. However, in our context, in order to study the full system \eqref{e.MFGstochBis}, in particular, the forward FP equation, an understanding of the gradient $D\widetilde  u_t$ is required, which is not accessible in the path-dependent theory.

The strategy of \cite{CSmfg} for the case $\widetilde  a\equiv 0$ is to adapt the classical approach of Kru\v{z}hkov \cite{Kr60} for forward Hamilton-Jacobi equations (see also Douglis \cite{Do61}, Evans \cite[Section 3.3]{EvBook}, and Fleming \cite{Fl69}), in which solutions are understood in the almost-everywhere sense, and well-posedness is established in the class of functions that are globally semiconcave in the state variable. In the setting of \cite{CSmfg}, this means that the equation is understood as a stochastic evolution satisfied for a.e. $x \in \RR^d$, with the martingale $M$ taking values in $L^\oo(\RR^d)$.

There are immediate obstacles in adapting these methods to the second-order equation \eqref{E:hjb}. As noted above, one cannot expect $\widetilde  u$ to possess any additional regularity than in the first-order setting. In fact, the inherently first-order arguments employed in \cite{Kr60, Do61, EvBook,Fl69, CSmfg} do not carry over to \eqref{E:hjb}, whose solutions no longer have a finite domain of dependence. In addition, \eqref{E:hjb} can no longer be interpreted in an almost-everywhere sense, because the distribution $\tr(\widetilde  a_t D^2 \widetilde  u_t)$ is not a function in general.

Instead, here we interpret \eqref{E:hjb} as a stochastic evolution in the space of distributions, or, more precisely, locally finite measures (see Definition \ref{D:solution} below). Of central importance to the well-posedness, just as in the first-order setting, is a global semiconcavity bound for the solutions. The existence of such a solution (Theorem \ref{T:existence}) relies on the propagation in time of semiconcavity for deterministic HJB equations (that is, \eqref{E:hjb} with $M \equiv 0$). This propagation estimate (Proposition \ref{P:dethjb} below), as far as we know, does not appear in the literature in the present generality.

The uniqueness and stability of solutions (Theorems \ref{T:comparison} and \ref{T:comparison2}) rely on the existence of a smooth, positive supersolution $\psi$ of the forward-in-time, deterministic linearization of \eqref{E:hjb}, namely
\[
	\del_t \psi - \tr[A \psi] + B \cdot D\psi \ge 0 \quad \text{in } \RR^d \times [0,T],
\]
which holds uniformly over bounded $A \in \mbb S^d$ and $B \in \RR^d$. When $A \equiv 0$, it is enough to take the indicator function of a ball with arbitrarily large radius shrinking at a fixed rate. For $A\neq 0$, it is possible to find such a $\psi$ with exponential decay at infinity, but care must be taken when using it as a test function, due to the fact that $D^2u$ is only locally bounded as a measure.

Once the HJB equation \eqref{E:hjb} is well understood, we turn to the study of the system \eqref{e.MFGstochBis}. The existence and uniqueness of a probabilistically strong solution, that is, adapted to the filtration of the common noise $W$, was proved directly in \cite{CSmfg} in the first-order case $\widetilde  a \equiv 0$  under the additional assumption that the Hamiltonian is of  ``separated'' form, that is, 
\[
	H_t(x,p,m) = \hat H_t(x,p) - F_t(x,m),
\]
with the coupling functions $F$ and $G$ strictly monotone. This latter assumption allows for a direct proof of the uniqueness  of the  solution using the well-known ``duality trick'' of Lasry and Lions \cite{LLJapan}.

Our second main contribution is to generalize the existence result by not only allowing for a degenerate diffusion matrix, but also imposing  only mild assumptions on the joint dependence of $a$, $H$, $G$ on $u$ and $m$. In particular, under mild regularity assumptions with respect to the measure variable $m$, we prove in Theorem \ref{T.ExistenceMFG} the existence of global, probabilistically weak solutions of \eqref{e.MFGstochBis}. This is achieved by interpreting \eqref{e.MFGstochBis} as a fixed point problem for $\widetilde  m$ in a subset of $L^\oo(\Omega, C([0,T], \mcl P_2(\RR^d) \cap L^\oo(\RR^d)))$, where $\mcl P_2$ is the space of probability measures with finite second moment. The main issue with this strategy is the failure of compactness, due to the fact that the probability space $\Omega$ may be quite large. Compactness is restored by studying a variant of \eqref{e.MFGstochBis} in which the probability space is discretized in an appropriate way. After invoking Kakutani's fixed point theorem, a weak solution is obtained upon taking weak limits.

The proof of Theorem \ref{T.ExistenceMFG} borrows several ideas from the work of Carmona, Delarue, and Lacker \cite{carmona2016mean}, who study mean field games with common noise from a purely probabilistic perspective. The main difference with our work is that we proceed at the PDE level and not in the optimal control formulation, as we exploit the regularity properties of the value function in the state variable $x$. We therefore bypass concerns about the existence or uniqueness of optimal trajectories, instead developing an understanding of the PDE system itself.

The final main result of the paper (Theorem \ref{T:MFGuniqueness}) is the uniqueness of solutions of \eqref{e.MFGstochBis} under a monotonicity assumption. More precisely, borrowing terminology from the theory of stochastic differential equations, it is shown that pathwise uniqueness holds, and thus, by a Yamada-Watanabe-type  argument, the unique solution is in fact strong in the probabilistic sense, which means not only that $(\widetilde  u, M, \widetilde  m)$ is adapted with respect to the filtration generated by the common noise $W$, but also, can be represented as a measurable functional of the Wiener process $W$.

The uniqueness proof, whose essential method was introduced in \cite{LLJapan}, involves studying the time evolution of the quantity
\[
	t \mapsto \int_{\RR^d} (\widetilde  u^1_t(x) - \widetilde  u^2_t(x)) (\widetilde  m^1_t(x) - \widetilde  m^2_t(x))dx
\]
for two solutions $(\widetilde  u^1,\widetilde  m^1)$ and $(\widetilde  u^2, \widetilde  m^2)$. For the particular system \eqref{e.MFGstochBis}, such an argument cannot proceed as stated, because neither $\widetilde  m$ nor $\widetilde  u$ is regular enough to be used as a test function in each others' equations. The resulting analysis, which involves estimating commutators after mollifying and localizing, relies precisely on the fact that $D^2 \widetilde  u$ is a locally finite measure for a solution $\widetilde  u$ of the backward stochastic HJB equation.

\subsection{Previous work on MFG}
The MFG system, in the absence of common noise, was introduced and studied by Lasry and Lions  in \cite{LLJapan,LL06cr1, LL06cr2}. In the presence of both idiosyncratic and common noise, the  stochastic MFG system was first investigated 
in \cite{CDLL} in which the state space is the torus $\RR^d/\ZZ^d$, and the existence and uniqueness of a strong solution was established for separated Hamiltonians, strictly monotone coupling functions, and nondegenerate diffusion. The result was extended to $\R^d$ by Carmona and Delarue in the monograph \cite{CaDeBook}.

An alternative analytic approach to study MFG equilibria with a common noise is the master equation,  introduced by Lasry and Lions and  presented by Lions in \cite{LiCoursCollege}, which is a deterministic, nonlinear, nonlocal transport equation in the space of probability measures. The existence and uniqueness of classical solutions to the master equation was first established in \cite{CDLL}  (see also \cite{CaDeBook}). A related approach, presented in \cite{LiCoursCollege}, is to lift probability measures to the Hilbert space of $L^2$-integrable random variables via their probability laws.

The master equation has recently attracted much attention, especially for problems with common noise. In \cite{cardaliaguet2020splitting}, Cardaliaguet, Cirant and Porretta build solutions using a splitting method. The setting of a finite state space is investigated by Bertucci, Lasry  and Lions \cite{bertucci2019some, BeLaLi20} and Bayraktar, Cecchin, Cohen and Delarue \cite{bayraktar2019finite}; in the latter work, it is also demonstrated that common noise can restore uniqueness in the absence of monotonicity. Another approach to uniqueness is explored by Gangbo, M\'{e}sz\'{a}ros, Mou and Zhang \cite{gangbo2021mean} via displacement convexity. Various notions of weak solutions to master equations are proposed by Mou and Zhang \cite{MoZh} and Bertucci \cite{Be20, Be21}. Through a combination of the Hilbertian interpretation and Bertucci's notion of monotone solutions, Cardaliaguet and Souganidis  build a unique, global, weak solution to the master equation with common noise but without idiosyncratic noise \cite{CSMaster}; we note that this approach does not seem to adapt to problems with idiosyncratic noise.

The analysis of MFG equilibria with a common noise can also be viewed from a purely probabilistic viewpoint. This idea was first developed 
in \cite{carmona2016mean}, and indeed, our analysis of probabilistically weak solutions of the MFG system borrows several ideas from this paper. The existence results of \cite{carmona2016mean} were recently extended by Djete \cite{djete21} to MFG problems with an interaction through the controls. 

\subsection{Main assumptions}\label{sub:assumptions}

Throughout the paper, we introduce diffusion coefficients, Hamiltonians, and terminal data that, depending on the setting, take on different forms (random or deterministic, coupling within the forward-backward system, monotonicity, etc.). We collect and discuss here the basic assumptions that are used in all sections.

The degenerate diffusion matrices will always take the form
\begin{equation}\label{A:MAINdiffusion}
\mcl A := \Sigma \Sigma^\tau, \  \text{ where, for some } m \le d,  \ \Sigma \in C( [0,T] ;C^{1,1}(\RR^d; \RR^{d \times m})). 
\end{equation}
The terminal data will be a given function $\mcl G$ such that
\begin{equation}\label{A:MAINterminal}
	\mcl G: \RR^d \to \RR \quad \text{is globally bounded, Lipschitz, and semiconcave}.
\end{equation}
Finally, the Hamiltonian $\mcl H$ is assumed to have the following regularity and structural properties:
\begin{equation}\label{A:MAINHamiltonian}
	\left\{
	\begin{split}
	&\text{for all $t \in [0,T]$, } \mcl H_t \in C^2(\RR^d \times \RR^d); \\
	&\text{for all $R > 0$, }\RR^d \times B_R \ni (x,p) \mapsto \mcl H_t(x,p)\\
	&\qquad \text{is uniformly bounded and Lipschitz for $t \in [0,T]$;}\\
	&[0,T] \ni t \mapsto \mcl H_t(x,p) \quad \text{is continuous uniformly over $(x,p) \in \RR^d \times B_R$; and}\\
	&p \mapsto \mcl H_t(x,p) \quad \text{is uniformly convex},
	\end{split}
	\right.
\end{equation}
and
\begin{equation}\label{A:MAINHcondition}
	\left\{
	\begin{split}
	&\text{for some $\lambda_0,C_0 > 0$ and all $t \in [0,T]$, $p,q \in \RR^d$, and $|z| = 1$,}\\
	&|D_x \mcl H_t(x,p)| \le C_0 + \lambda_0(p \cdot D_p \mcl H_t(x,p) - \mcl H_t(x,p)) \text{ and}\\
	&D^2_{pp} \mcl H_t(x,p)q \cdot q + 2D^2_{px}\mcl H_t(x,p) z \cdot q + D^2_{xx} \mcl H_t(x,p)z \cdot z \\
	&\qquad + \lambda_0\left( D_p \mcl H_t(x,p) \cdot p - \mcl H_t(x,p) \right)  \ge -C_0.
	\end{split}
	\right.
\end{equation}

Condition \eqref{A:MAINHcondition}, which appears also in \cite[(1.10)(iii)]{CSmfg}, is an assumption on the nature of the coupling between the $x$ and $p$ variables in $\mcl H$. It also has a natural interpretation from the viewpoint of the underlying stochastic control problem. Indeed, \eqref{A:MAINHcondition} is equivalent to the following condition on the Legendre transform $\mcl H^*_t(x,\alpha) = \sup_p (p \cdot \alpha - \mcl H_t(x,p))$:
\begin{equation}\label{Lagrangian}
	\left\{
	\begin{split}
	&\text{for all $t \in [0,T]$ and $x,\alpha \in \RR^d$,}\\
	&|D_x \mcl H^*_t(x,\alpha)| \le C_0 + \lambda_0 \mcl H^*_t(x,\alpha) \quad \text{and}\\
	&D^2_{xx} \mcl H^*_t(x,\alpha) \le (C_0 + \lambda_0 \mcl H^*_t(x,\alpha))\Id.
	\end{split}
	\right.
\end{equation}
The formula for the value function solving the terminal value problem
\[
	-\del_t U - \tr[ \mcl A_t(x) D^2 U_t(x)] + \mcl H_t(x,DU_t(x)) = 0 \quad \text{in }(0,T) \times \RR^d, \quad U_T = \mcl G
\]
is
\begin{equation}\label{generalcontrol}
	(t,x) \mapsto \inf_{\alpha} \EE \left[ \int_t^T \mcl H^*_s(X^{\alpha,t,x}_s, \alpha_s)ds + \mcl G(X^{\alpha,t,x}_T) \right],
\end{equation}
where, for some $m$-dimensional Brownian motion $B$, $X^{\alpha,t,x}$ solves
\[
	\left\{
	\begin{split}
	&dX^{\alpha,t,x}_s = \alpha_s ds + \Sigma_s(X^{\alpha,t,x}_s) dB_s, \quad s \in [t,T],\\
	&X^{\alpha,t,x}_t = x.
	\end{split}
	\right.
\]
The infimum in \eqref{generalcontrol} is taken over progressively measurable controls $\alpha: [t,T] \to \RR^d$ that satisfy the admissibility criterion
\[
	\EE\left[ \int_t^T \sup_{x \in \RR^d} \mcl H^*_s(x,\alpha_s)ds \right] < \oo.
\]
Formally differentiating \eqref{generalcontrol} twice in the $x$-variable, one arrives at a semiconcavity bound precisely when \eqref{A:MAINdiffusion} and \eqref{A:MAINterminal} hold, and when 
\[
	\EE\left[ \int_t^T \sup_{x \in \RR^d} |D_x \mcl H^*_s(x,\alpha_s) |ds + \int_t^T \sup_{x \in \RR^d, \, |z| = 1} D^2_{xx} \mcl H^*_s(x,\alpha_s)z \cdot z ds \right] < \oo.
\]
Then \eqref{Lagrangian} is exactly what one needs in order for the latter to hold. 

This demonstrates that \eqref{A:MAINHcondition} is the most natural assumption to make on $\mcl H$ in order for semiconcavity to hold for the value function $U$. These arguments can all be made rigorous, but we choose instead to prove Proposition \ref{P:dethjb} below by performing the differentiation on the level of the PDE, keeping with the theme of this paper.

\begin{remark}\label{R:transform}
	Observe that, if the functions $H$, $G$, and $a$ in the original system \eqref{e.MFGstoch} satisfy assumptions \eqref{A:MAINdiffusion} - \eqref{A:MAINHcondition} uniformly over $m \in \mcl P(\RR^d)$, then so do the transformed functions $\widetilde  H$, $\widetilde  G$, and $\widetilde  a$ defined via \eqref{pierre0}, \eqref{takis0}, and \eqref{pierre02}, uniformly over the probability space on which the Wiener process $W$ is defined. This is a feature of the fact that the volatility $\beta$ is constant in the state variable. This is not merely a simplification: the precise propagation of semiconcavity proved in Proposition \ref{P:dethjb} below depends on the various constants in \eqref{A:MAINdiffusion} - \eqref{A:MAINHcondition} above.
\end{remark}	

\subsection{Organization of the paper} In Section \ref{sec:propage}, we prove the propagation of semiconcavity estimate for backward deterministic HJB equations, and also list some properties of semiconcave functions that are used throughout the paper. Section \ref{sec.HJ} introduces the definition of solutions for the backward stochastic HJB equation \eqref{E:hjb}, including a discussion of stochastic processes taking values in the space of locally finite measures, and establishes well-posedness results. Section \ref{sec:mfg} involves the study of weak and strong solutions to the MFG system \eqref{e.MFGstochBis}.

\subsection{Notation} Given an open or closed domain $U$ in Euclidean space, $\mcl B_U$ denotes the Borel $\sigma$-algebra inside $U$ and $C(U)$ is the space of continuous functions on $U$ endowed with the topology of (local) uniform convergence, and $C_c(U)$ is the space of compactly supported continuous functions on $U$. For a function $f$ defined on a measure space $\mcl A$, we set $\norm{f}_{\oo,\mcl A} := \esssup_{x \in \mcl A} |f(x)|$, and, where it does not create confusion, we simply write $\norm{f}_\oo = \norm{f}_{\oo,\mcl A}$. Throughout the paper, (generalized) function spaces $X(\RR^d)$ on all of $\RR^d$ are simply denoted by $X$ when this does not cause confusion.

We define $\mcl D = C_c^\oo$, and the space $\mcl D'$ of distributions consists of real-valued linear functionals $\mcl T:\mcl D \to \RR$ satisfying the following local continuity property:
\begin{align*}
	&\text{if } \left\{ (\phi_n)_{n \in \NN}, \phi \right\} \subset C_c^\oo, \quad \bigcup_{n \in \NN} \supp \phi_n \text{ is bounded, and}\\
	&D^k \phi_n \xrightarrow{n \to \oo} D^k \phi \text{ uniformly for all }k = 0,1,2,\ldots, \text{ then}\\
	&\mcl T(\phi_n) \xrightarrow{n \to \oo} \mcl T(\phi),
\end{align*}
where $D^k \phi$ refers to the tensor of all partial derivatives of $\phi$ of order $k$. We often write $\mcl T(\phi) = \langle \mcl T, \phi \rangle$.

The space of locally finite, signed, Radon measures on $\RR^d$ is denoted by $\mcl M_\loc = \mcl M_\loc(\RR^d)$. For $\mu \in \mcl M_\loc$ and $f \in C_c$, $\langle \mu, f \rangle := \int_{\RR^d} f(x) \mu(dx)$, and, for an open bounded set $U \subset \RR^d$, we write the total variation norm as
\[
	\norm{\mu}_{TV(U)} := \sup_{f \in C_c(U), \; \norm{f}_{\oo} \le 1 } \langle \mu, f \rangle.
\]
We define the quasi-norm
\[
	\norm{\mu}_{\mcl M_\loc} := \sum_{n=1}^\oo \norm{\mu}_{TV(B_n)} \wedge 2^{-n},
\]
where $B_R$ is the open ball centered at $0$ and of radius $R$. We will say $\mcl A \subset \mcl M_\loc$ is bounded if $\{ \norm{\mu}_{\mcl M_\loc} \}_{\mu \in \mcl A}$ is bounded, and, for $p \in [1,\oo]$ and a measure space $(\mcl A, \gamma)$,
\[
	L^p(\mcl A, \mcl M_\loc) := \left\{ \mu: \mcl A \to \mcl M_\loc \mid \norm{\mu}_{L^p(\mcl A, \mcl M_\loc)} := \left( \int_\mcl A  \norm{\mu(a)}_{\mcl M_\loc}^p \gamma(da) \right)^{1/p} < \oo \right\},
\]
with the usual modification if $p = \oo$. The metric $(\mu,\nu) \mapsto \norm{\mu - \nu}_{\mcl M_\loc}$ makes $\mcl M_\loc$ into a complete metric space with the topology of local convergence in total variation, and, similarly, $L^p(\mcl A, \mcl M_\loc)$ is a metric space with the metric $(\mu,\nu) \mapsto \norm{ \mu - \nu}_{L^p(\mcl A, \mcl M_\loc)}$.

The Sobolev space of functions $f \in L^1$ for which the distributional derivative $Df$ belongs to $L^1$ is denoted $W^{1,1}$. Its dual is a space of distributions:
\[
	W^{-1,\oo} := \left\{ \mcl T \in \mcl D' : \norm{\mcl T}_{W^{-1,\oo}} := \sup_{\phi \in \mcl D} \frac{ |\mcl T(\phi)|}{\norm{\phi}_{W^{1,1}} }< \oo \right\}. 
\]
If $f \in L^\oo$, then the distributional derivative $Df$ belongs to $W^{-1,\oo}(\RR^d)$, and, more generally, $W^{-1,\oo}(\RR^d)$ consists of distributions of the form $f_0 + Df_1$ for $f_0,f_1 \in L^\oo$, and, in fact,
\[
	\norm{\mcl T}_{W^{-1,\oo}} := \inf\left\{ \norm{f_0}_\oo + \norm{f_1}_\oo : \mcl T = f_0 + Df_1 , \; f_0,f_1 \in L^\oo\right\}.
\]

The space of probability measures on $\RR^d$ is denoted by $\mcl P(\RR^d)$, and, for $p > 0$, 
\[
	\mcl P_p(\RR^d) := \left\{ \mu \in \mcl P(\RR^d) : \int_{\RR^d} |x|^p \mu(dx) < \oo \right\}.
\]
The $p$-Wasserstein distance $\mbf d_p$ on $\mcl P_p(\RR^d)$ is defined by
\[
	\mbf d_p(\mu,\mu') := \inf_{\gamma} \left(\iint_{\RR^d \times \RR^d} |x - y|^p d\gamma(x,y)\right)^{1/p},
\]
where the infimum is taken over probability measures $\gamma \in \mcl P_p(\RR^d \times \RR^d)$ with first and second marginal respectively $\mu$ and $\mu'$. Given a probability space $(\Omega,\mbb P)$, we write $\EE_\PP := \int_\Omega d \mbb P$.

Throughout, we fix $\rho \in C^2(\RR^d)$ satisfying $\rho \ge 0$, $\supp \rho \subset B_1$, and $\int_{\RR^d} \rho = 1$, and, for $\delta > 0$ and $x \in \RR^d$, we set
\begin{equation}\label{mollifier}
	\rho_\delta(x) := \frac{1}{\delta^d} \rho\left( \frac{x}{\delta}\right).
\end{equation}
The transpose of a matrix $M$ is denoted by $M^\tau$. We define, for $X \in \mbb S^d$, the space of $d\times d$ symmetric matrices, the degenerate elliptic operator $$m_+(X) := \max_{|v| \le 1}  Xv \cdot v. $$ Note that $m_+(X)$ is the maximum of $0$ and the largest eigenvalue of $X$, and therefore $m_+$ is increasing with respect to the partial order on $\mbb S^d$. The Lebesgue measure on $\RR^d$ is denoted by $\mcl L$.

Throughout proofs, we denote by $C$ a positive constant that may change from line to line, depending only on the relevant data in the assumptions of the corresponding theorem/proposition/lemma.

\section{Propagation of semiconcavity for deterministic HJB equations}\label{sec:propage}

The purpose of this section is to prove that semiconcavity bounds propagate in time for solutions of the deterministic HJB equation
\begin{equation}\label{E:dethjb}
	\left\{
	\begin{split}
	&-\del_t u_t - \tr[a_t(x) D^2 u_t] + H_t(x,Du_t(x)) = 0 \quad \text{in } [0,T] \times \RR^d  \quad \text{and} \\
	&u_T = G \quad \text{in } \RR^d,
	\end{split}
	\right.
\end{equation}
where $a$, $H$, and $G$ are fixed and deterministic, and satisfy the assumptions from subsection \ref{sub:assumptions}.

Recall that a function $u:\RR^d \to \RR$ is said to be semiconcave with constant $C > 0$ if
\[
	u(x+h) + u(x-h) - 2u(x) \le C|h|^2 \quad \text{for all } x,h \in \RR^d.
\]
The result below gives a precise estimate of the positive part of the maximal eigenvalue of the density of the measure $D^2 u$, which, in turn, controls the semiconcavity constant (see Lemma \ref{L:SCprops} and Remark \ref{R:densitybound} below).

\begin{proposition}\label{P:dethjb}
	Assume $a$, $G$, and $H$ satisfy \eqref{A:MAINdiffusion}, \eqref{A:MAINterminal}, \eqref{A:MAINHamiltonian}, and \eqref{A:MAINHcondition}, and let $u$ be the unique viscosity solution of \eqref{E:dethjb}. Then, for all $t\in [0,T]$,  
	\begin{equation}\label{detubound}
		\norm{u_t}_{\oo} \le \norm{G}_{\oo} + \norm{H_\cdot(\cdot,0)}_\oo (T-t). 
	\end{equation}
	Moreover, $u$ is semiconcave and Lipschitz in the $x$-variable, uniformly on $[0,T] \times \RR^d$, and there exists a constant $C > 0$ depending only on $T$ and the various data in those assumptions such that, for all $t \in [0,T]$,
	\begin{equation}\label{detD2ubound}
		\begin{split}
		&\esssup_{x \in \RR^d} \left\{ \bigg( m_+\left( \frac{dD^2u_t(x)}{d \mcl L} \right)^2  + |Du_t(x)|^2 \bigg)^{1/2} - \sqrt{2}\lambda_0 u_t(x) \right\}_+\\
		&\le C(T-t) \\
		&\quad + e^{C(T-t)}\esssup_{x \in \RR^d} \left\{ m_+\left( \frac{d D^2g(x)}{d \mcl L} \right)^2 + |Dg(x)|^2 \bigg)^{1/2} - \sqrt{2} \lambda_0 g(x)\right\}_+.
		\end{split}
	\end{equation}
	 Finally, $- \tr[aD^2 u] \in L^\oo([0,T], \mcl M_\loc(\RR^d))$, and \eqref{E:dethjb} is satisfied in the distributional sense.
\end{proposition}

The fact that solutions of HJB equations are semiconcave is well known, and is a consequence of the minimization problem that defines the value function $u$; see for instance \cite{CaSiBook, ishiilionssemiconc,lionsHJB}. Note that these references do not provide bounds as in Proposition \ref{P:dethjb} in the generality assumed here, since either the derivatives of $H$ are assumed uniformly bounded, or an a priori bound on $\norm{Du}_\oo$ is assumed. The main feature of Proposition \ref{P:dethjb} is the propagation of a bound for the quantity on the left-hand side of \eqref{detD2ubound}, which ultimately leads to the semiconcavity and Lipschitz bounds for solutions of the stochastic HJB equation \eqref{E:hjb}.

We summarize some properties and results concerning semiconcave functions used in the proof of Proposition \ref{P:dethjb} and throughout the paper. For more details and explanations, see for instance the book of Evans and Gariepy \cite{EG}.

\begin{lemma}\label{L:SCprops}
Assume $u : \RR^d \to \RR$ is bounded and semiconcave, with semiconcavity constant $C \ge 0$. Then the following hold:

\begin{enumerate}[(a)]
\item\label{L:semiconcave} The distribution $D^2 u$ is a measure of the form $D^2 u = f(x) d \mcl L - d\mu$, where $\mu$ is nonnegative and singular with respect to the Lebesgue measure $\mcl L$ and $f \in L^1_\loc(\RR^d; \mbb S^d)$ satisfies $f \le C \Id$ Lebesgue-a.e.

\item\label{L:interpolation2} The function $u$ is Lipschitz, and $\norm{Du}_{\oo} \le 2 \sqrt{ C \norm{u}_{\oo} }$. In particular, $D^2 u \in W^{-1,\oo}$.

\item\label{L:TVbound} There exists a dimensional constant $c_d > 0$ such that $\norm{D^2u}_{\mcl M_\loc(\RR^d)} \le c_d ( \norm{Du}_\oo + C)$. 

\item\label{L:Duconverge} If $(u^n)_{n \in \NN}: \RR^d \to \RR$ is a sequence of globally Lipschitz, semiconcave functions, uniformly in $n \in \NN$, and, as $n \to \oo$, $u^n$ converges locally uniformly to $u$, then, as $n \to \oo$, $Du^n$ converges almost everywhere to $Du$.

\end{enumerate}
\end{lemma}

\begin{remark}\label{R:densitybound}
	The statement in part \eqref{L:semiconcave} is actually equivalent to semiconcavity. The semiconcavity constant for a globally semiconcave function $u$ is equal to
	\[
		\esssup_{x \in \RR^d} m_+\left( \frac{d D^2 u}{d \mcl L}(x) \right),
	\]
	and, hence, the bound in Proposition \ref{P:dethjb} indeed propagates the semiconcavity estimate for $u$. Throughout the rest of the paper, where it does not create confusion, we abuse notation and denote by $\esssup m_+(D^2 u)$ the semiconcavity constant for a function $u:\RR^d \to \RR$.
\end{remark}

Proposition \ref{P:dethjb}, and in particular the estimate \eqref{detD2ubound}, is proved by differentiating the equation \eqref{E:dethjb}, and, as such, it is necessary to perform a regularization to make the arguments rigorous. For $0 < \eps < 1$, set $a^\eps = \sigma \sigma^\tau + \eps \Id$ and $G^\eps = G * \rho_\eps$ for $\rho_\eps$ as in \eqref{mollifier}, and consider the classical solution $u^\eps$ of
	\begin{equation}\label{E:eps}
	\left\{
	\begin{split}
		&-\del_t u^\eps_t - \tr[a^\eps_t(x) D^2 u^\eps_t(x)] + H_t(x,Du^\eps_t(x)) = 0 \quad \text{in } [0,T] \times \RR^d \quad \text{and}\\
		&u^\eps_T = G^\eps \quad \text{in } \RR^d.
	\end{split}
	\right.
	\end{equation}

\begin{lemma}\label{L:epshjb}
	The solution $u^\eps$ of \eqref{E:eps} satisfies \eqref{detubound} and \eqref{detD2ubound}, where the constant $C > 0$ is as in Proposition \ref{P:dethjb} and independent of $\eps > 0$.
\end{lemma}

We first prove Proposition \ref{P:dethjb} with Lemma \ref{L:epshjb} in hand.

\begin{proof}[Proof of Proposition \ref{P:dethjb}]
	In view of Lemma \ref{L:epshjb}, $(u^\eps)_{0 < \eps < 1}$ is uniformly bounded, Lipschitz, and semiconcave in the space variable. The stability properties of viscosity solutions imply that, as $\eps \to 0$, $u^\eps$ converges locally uniformly in $[0,T] \times \RR^d$ to $u$, and this then yields the Lipschitz regularity and semiconcavity of $u$.
		
	As $\eps \to 0$, $D^2u^\eps$ converges to $D^2 u$ in the sense of distributions, and, in view of Lemma \ref{L:SCprops}\eqref{L:Duconverge}, as $\eps \to 0$, $Du^\eps$ converges almost everywhere to $Du$. We conclude that, as $\eps \to 0$, the left-hand side of \eqref{E:eps} converges to the left-hand side of \eqref{E:dethjb} in the sense of distributions. The claim about $-\tr[a_t(x) D^2 u_t(x)]$ follows from Lemma \ref{L:SCprops}\eqref{L:TVbound}. 

	Using Jensen's inequality we see that, for all $x \in \RR^d$, 
	\begin{align*}
		&\bigg(m_+\left( D^2 G^\eps(x) \right)^2 + |DG^\eps(x)|^2 \bigg)^{1/2} - \sqrt{2} \lambda_0 G^\eps(x)\\
		&\le \rho_\eps * \left\{ \bigg( m_+ \left( \frac{d D^2 G}{d \mcl L} \right)^2 + |DG|^2 \bigg)^{1/2} - \sqrt{2} \lambda_0 G \right\}(x) \\
		&\le \esssup_{x \in \RR^d} \left\{ m_+\left( \frac{d D^2g(x)}{d \mcl L} \right)^2 + |Dg(x)|^2 \bigg)^{1/2} - \sqrt{2} \lambda_0 g(x)\right\}_+ =: M.
	\end{align*}
	Rearranging terms in the bound \eqref{detD2ubound} satisfied by $u^\eps$ gives, for all $t \in [0,T]$,
	\[
		D^2 u^\eps_t \le \left(\left[ (C(T-t) + e^{C(T-t)}M + \sqrt{2} u^\eps_t)^2 - |Du^\eps_t|^2 \right]_+ \right)^{1/2} \Id.
	\]
	As $\eps \to 0$, the left- and right-hand sides above converge in the sense of distributions, yielding, for all $t \in [0,T]$,
	\[
		D^2 u_t \le \left(\left[ (C(T-t) + e^{C(T-t)}M + \sqrt{2} u_t)^2 - |Du_t|^2 \right]_+ \right)^{1/2} \Id.
	\]
	In particular, the same bound holds for the density $\frac{d D^2 u}{d \mcl L}$ a.e. in view of Lemma \ref{L:SCprops}\eqref{L:semiconcave}. Rearranging terms yields the desired bound \eqref{detD2ubound}.
\end{proof}

We now present the

\begin{proof}[Proof of Lemma \ref{L:epshjb}]
	Throughout the proof, we suppress the dependence on $\eps$, but note that the solution $u$ is smooth.

	For $\xi,\eta \in \RR^d$ and $f: \RR^d \to \RR$, we use the shorthand $f_\xi = \xi \cdot Df$ and $f_{\eta \eta} = D^2 f\eta \cdot \eta$. We also suppress dependence on $(t,x)$ below where it does not create confusion, and note that the arguments of $H$ and its derivatives will always be $(x,Du)$.

	Differentiating \eqref{E:eps} twice leads to 
	\begin{equation}\label{onexi}
	-\del_t u_\xi - \tr[a D^2u_\xi] + D_p H \cdot Du_\xi
	 - \tr[a_\xi D^2 u] + \xi \cdot D_x H_t = 0
\end{equation}
and
\begin{equation}\label{twoeta}
	\begin{split}
	-\del_t u_{\eta \eta} &- \tr[a D^2u_{\eta\eta}] + D_p H \cdot Du_{\eta \eta}
	 - 2 \tr[a_\eta D^2u_\eta] - \tr[a_{\eta\eta} D^2 u] \\
	 &+ D^2_{xx}H\eta \cdot \eta + 2D^2_{xp} H\eta \cdot Du_\eta + D^2_{pp} H Du_\eta \cdot Du_\eta = 0.
	 \end{split}
\end{equation}
We now set 
\[
	U = U(t,x,\xi,\eta) := u_\xi + u_{\eta\eta} - \lambda_0 u\left(1 + \frac{|\xi|^2}{2} + \frac{ |\eta|^4}{2}\right)
\]
and compute 
\begin{align*}
	D^2_{(x,\xi,\eta)} U &= 
	\begin{pmatrix}
		D^2_{xx} (u_\xi + u_{\eta\eta}) & D^2 u & 2D_x(D^2u\eta)^\tau \\
		D^2 u & 0 & 0 \\
		2D_x(D^2 u\eta) & 0 & 2D^2 u
	\end{pmatrix}\\
	&-\lambda_0
	\begin{pmatrix}
		D^2 u\left( 1 + \frac{|\xi|^2}{2} + \frac{|\eta|^4}{2} \right) &  Du \otimes \xi & 2|\eta|^2 Du \otimes \eta \\
		\xi \otimes Du & u \Id & 0\\
		2|\eta|^2 \eta \otimes Du & 0 & 2u(|\eta|^2\Id + 2 \eta \otimes  \eta)
	\end{pmatrix}.
\end{align*}
We then set
\[
	A = A(t,x,\xi,\eta) := 
	\begin{pmatrix}
		\sigma \\
		\sigma_\xi + \sigma_{\eta\eta} \\
		\sigma_\eta
	\end{pmatrix}
	\begin{pmatrix}
		\sigma \\
		\sigma_\xi + \sigma_{\eta\eta} \\
		\sigma_\eta
	\end{pmatrix}^\tau
	+
	\eps
	\begin{pmatrix}
		\Id & 0 & 0\\
		0 & 0 & 0\\
		0 & 0 & 0
	\end{pmatrix}
	\ge 0,
\]
and add \eqref{onexi} and \eqref{twoeta} and subtract the product of \eqref{E:dethjb} with $\lambda_0(1 +|\xi|^2/2 + |\eta|^4/2)$. Using the identities
\[
	a_\xi = \sigma_\xi \sigma^\tau + \sigma \sigma^\tau_\xi \quad \text{and} \quad
	a_{\eta\eta} = \sigma_{\eta \eta}\sigma^\tau + 2 \sigma_\eta \sigma_\eta^\tau + \sigma \sigma_{\eta\eta}^\tau,
\]
we obtain, for a constant $C > 0$ depending on $\norm{\sigma}_{C^{1,1}}$ and the bound for $u$,
\begin{align*}
	-\del_t U &- \tr[A D^2_{(x,\xi,\eta)}U] + D_p H \cdot D_x U\\
	&+ D_x H \cdot \xi + D^2_{xx} H\eta \cdot \eta + 2D^2_{px} H\eta \cdot Du_{\eta} + D^2_{pp}H Du_\eta \cdot Du_\eta\\
	&+ \lambda_0 \left(1 + \frac{|\xi|^2}{2} + \frac{|\eta|^4}{2} \right)(Du \cdot D_p H - H)\\
	&\le C(\norm{Du_t}_\oo + 1)(|\xi|^2 + |\eta|^4).
\end{align*}
Using \eqref{A:MAINHcondition} and the inequalities
\[
	1 + \frac{|\xi|^2}{2} + \frac{|\eta|^4}{2} \ge |\xi| + |\eta|^2 \quad \text{and} \quad p \cdot D_p H(x,p) - H(x,p) \ge H(x,0),
\]
this further becomes, for a different constant,
\begin{equation}\label{sub}
	-\del_t U - \tr[A D^2_{(x,\xi,\eta)}U] + D_p H \cdot D_x U
	\le C(\norm{Du_t}_\oo + 1)(1 + |\xi|^2 + |\eta|^4).
\end{equation}
Now, for $\mu,\delta > 0$ and some $M(t)$ to be determined, assume for the sake of contradiction that
\[
	U(t,x,\xi,\eta) - M(t) (1 + |\xi|^2 +  |\eta|^4) - \delta|x|^2 - \mu(T-t)
\]
attains its maximum at some $(t_0,x_0,\xi_0,\eta_0)$ with $t_0 < T$. By standard arguments, we have $|x_0| \le C \delta^{-1/2}$ for some constant $C > 0$. In view of \eqref{A:MAINHamiltonian}, there exists $\bar C > 0$, depending on $\norm{Du}_\oo$ (and therefore, a priori, on $\eps$) such that
\[
	|D_p H_{t_0}(x_0,Du_{t_0}(x_0))| \le \oline{C}.
\]
Plugging into \eqref{sub} and using $A \ge 0$ yields
\[
	\mu - C\delta - \oline{C}\delta^{1/2} + [- \dot M(t_0) - CM - C(\norm{Du_t}_\oo+1)]\left( 1 + \frac{|\xi|^2}{2} + \frac{|\eta|^4}{2} \right) \le 0.
\]
A contradiction is reached upon taking $\delta$ sufficiently small, depending on $\mu$ and $\norm{Du}_{\oo}$, and setting, for arbitrary $\alpha > 0$,
\[
	M(t) := \alpha e^{C(T-t)} + C\int_t^T e^{C(s-t)} ( \norm{Du_s}_\oo + 1)ds.
\]
Taking first $\delta \to 0$ and then $\mu \to 0$, we deduce that, for all $t \in [0,T]$ and $\xi,\eta \in \RR^d$,
\begin{equation}\label{maxconc}
\begin{split}
	&u_{t,\eta\eta} + u_{t,\xi} - \left[C\int_t^T e^{C(s-t)} ( \norm{Du_s}_\oo + 1)ds + \lambda_0 u_t + \alpha e^{C(T-t)} \right]\\
	&\qquad \qquad \qquad \cdot \left(1 + \frac{|\xi|^2}{2} + \frac{|\eta|^4}{2} \right)\\
	&\le \sup_{x,\xi,\eta \in \RR^d} \left\{ G_{\eta\eta}(x) + G_\xi(x) - \left[ \lambda_0 G(x) + \alpha \right]\left( 1 + \frac{|\xi|^2}{2} + \frac{|\eta|^4}{2} \right) \right\}.
\end{split}
\end{equation}
Maximizing the right-hand side of \eqref{maxconc} in $\xi$ and $\eta$ yields
\[
	\sup_{x\in \RR^d} \left\{ \frac{ m_+(D^2 G(x))^2 + |DG(x)|^2}{2(\alpha + \lambda_0 G(x))} - (\alpha + \lambda_0 G(x)) \right\},
\]
which is nonpositive if we take
\[
	\alpha =  \sup_{x \in \RR^d} \left\{ \frac{1}{\sqrt{2}}\left( m_+(D^2 G(x))^2 + |DG(x)|^2\right)^{1/2} - \lambda_0 G(x) \right\}_+
\]
(note that in that case $\alpha + \lambda_0 G(x) \ge 0$ for all $x \in \RR^d$). Thus, \eqref{maxconc} becomes, for all $t \in [0,T]$ and $\xi,\eta \in \RR^d$,
\[
	u_{t,\eta\eta}+ u_{t,\xi} - \left[C\int_t^T e^{C(s-t)}(\norm{Du_s}_\oo + 1)ds + \lambda_0 u_t+ \alpha e^{C(T-t)} \right]\left(1 + \frac{|\xi|^2}{2} + \frac{|\eta|^4}{2} \right) \le 0.
\]
Maximizing in $\xi$ and $\eta$ gives, for all $t \in [0,T]$,
\begin{align*}
	&\frac{ m_+(D^2 u_t)^2 + |Du_t|^2}{2\left[ C\int_t^T e^{C(s-t)}(\norm{Du_s}_\oo + 1)ds + \lambda_0 u_t + \alpha e^{C(T-t)} \right]} \\
	&- C\int_t^T e^{C(s-t)}(\norm{Du_s}_\oo + 1)ds - \lambda_0 u_t- \alpha e^{C(T-t)} \le 0.
\end{align*}
Rearranging terms and using the definition of $\alpha$ yields that the quantity
\[
	\gamma_t := \sup_{x \in \RR^d} \left\{ \bigg( m_+(D^2 u_t(x))^2  + |Du_t(x)|^2 \bigg)^{1/2} - \sqrt{2}\lambda_0 u_t(x) \right\}_+
\]
satisfies
\[
		\gamma_t \le C\int_t^T e^{C(s-t)}(\norm{Du_s}_\oo + 1)ds + e^{C(T-t)}\gamma_T.
\]
From \eqref{detubound} and Cauchy-Schwartz inequality, we deduce  $\norm{Du_s}_\oo \le \gamma_s + C$, and so Gr\"onwall's inequality gives the desired bound
\[
	\gamma_t \le C(T-t)+ e^{C(T-t)} \gamma_T.
\]
\end{proof}

\section{The backward stochastic HJB equation}\label{sec.HJ}

We define a notion of a weak solution $(\widetilde  u,M)$ (Definition \ref{D:solution}) for the terminal value problem
\begin{equation}\label{E:hjbbody}
	\left\{
	\begin{split}
	&d \widetilde  u_t = \left[ - \tr(\widetilde  a_t(x) D^2 \widetilde  u_t) + \widetilde  H_t(x,D\widetilde  u_t) \right]dt + dM_t = 0 \quad \text{in } (0,T) \times \RR^d \quad \text{and} \\
	&\widetilde  u_T = \widetilde  G \quad \text{on } \RR^d.
	\end{split}
	\right.
\end{equation}
More precisely, the solutions are weak in the PDE sense, and strong in the probabilistic sense, since the unknown $(\widetilde  u,M)$ is required to be adapted with respect to the given underlying filtration.

The concept of solutions is based on a semiconcavity bound in space, which will imply in particular that $D^2 \widetilde  u$ is a locally finite measure. We then prove the existence, uniqueness, and stability of such solutions (Theorems \ref{T:existence}, \ref{T:comparison}, and \ref{T:comparison2}).

We fix a filtered probability space $(\Omega, \mbb F, \PP)$ such that
\begin{equation}\label{A:filtration}
	\mbb F = (\mbb F_t)_{t \in [0,T]} \quad \text{is a filtration satisfying the usual conditions,}
\end{equation}
and we assume that
\begin{equation}\label{A:stochHJB}
\left\{
	\begin{split}
		&\widetilde  a: \Omega \times [0,T] \times \RR^d \to \mbb S^d \text{ and } \widetilde  H: \Omega \times [0,T] \times \RR^d \times \RR^d \to \RR \text{ are progressively}\\
		&\qquad \text{measurable with respect to $\mbb F$ and the Borel topology,}\\
		&\widetilde  G: \Omega \times \RR^d \to \RR \text{ is } \mbb F_T \otimes \mcl B_{\RR^d} \text{-measurable, and}\\
		&\widetilde  a, \widetilde  G,\widetilde  H \text{ satisfy \eqref{A:MAINdiffusion}, \eqref{A:MAINterminal}, \eqref{A:MAINHamiltonian}, and \eqref{A:MAINHcondition} uniformly over $\Omega$}.
	\end{split}
\right.
\end{equation}
Recall that equations such as \eqref{E:hjbbody} arise from transforming the HJB equation in the MFG system \eqref{e.MFGstoch} through the relations \eqref{changevariable} - \eqref{martingale}, and therefore \eqref{A:stochHJB} is completely natural (see Remark \ref{R:transform}). Note that, throughout this section, the explicit dependence of the data on the Wiener process $W$ is not used.

%

\subsection{The definition of solutions} The notion of solution, and, in particular, the condition on $D^2\widetilde  u$, is motivated by Proposition \ref{P:dethjb}, which indicates that $D^2 \widetilde u$ can be expected to belong to $\mcl M_\loc \cap W^{-1,\oo}$. 

In what follows, we make use of the quasinorm
\[	\norm{ \mcl T}_{\mcl M_\loc \cap W^{-1,\oo}}
	= \norm{\mcl T}_{\mcl M_\loc(\RR^d)} \vee \norm{\mcl T}_{W^{-1,\oo}}.
\]
Just as for the quasinorm $\norm{\cdot}_{\mcl M_\loc}$, this induces a metric on $\mcl M_\loc \cap W^{-1,\oo}$. 
We will say a subset of $\mcl M_\loc \cap W^{-1,\oo}$ is bounded if the quasinorm is uniformly bounded on this subset. 

We will call a process $X: \Omega \times [0,T] \to \mcl M_\loc \cap W^{-1,\oo}$ progressively measurable if $\Omega \times [0,T] \ni (\omega,t) \mapsto \langle X_t(\omega) ,f \rangle$ is progressively measurable for all $f \in C_c \cup W^{1,1}$, and we will say $X$ is a martingale if $\langle X,f \rangle$ is a martingale for all $f \in C_c \cup W^{1,1}$.

\begin{remark}
	Much of the analysis to follow involves evaluating the relevant distribution-valued stochastic processes at test functions. It is for this reason that we work with the above notions of measurability/integrability, rather than dealing directly with the (strong) Borel topology of the metric space $\mcl M_\loc \cap W^{-1,\oo}$.
\end{remark}

We use the following infinite-dimensional generalization of the well-known fact that martingales adapted to right-continuous filtrations have c\`adl\`ag versions.

\begin{lemma}\label{L:cadlag}
	Assume $M$ is a progressively measurable martingale with respect to $\mbb F$ taking values in a bounded subset of $\mcl M_\loc \cap W^{-1,\oo}$. Then there exists a martingale $M': \Omega \times [0,T] \to \mcl M_\loc \cap W^{-1,\oo}$ such that
	\begin{enumerate}[(a)]
	\item\label{weakstarcadlag} For $\PP$-a.e. $\omega \in \Omega$,
	\[
		t \mapsto \langle M'_t(\omega), f \rangle \quad \text{is c\`adl\`ag for all }f \in C_c \cup W^{1,1}.
	\]
	\item\label{version} For all $t \in [0,T]$, $M_t(\omega) = M'_t(\omega)$ for $\PP$-a.e. $\omega \in \Omega$.
	\end{enumerate}
\end{lemma}

\begin{proof}
	Let $F \subset C_c \cup W^{1,1}$ be countable and dense in both $C_c$ and $W^{1,1}$. Then, for every $f \in F$, $\langle M, f \rangle$ is a real-valued martingale. Therefore there exist events $\Omega', (\Omega_t)_{t \in [0,T]}$ with $\PP(\Omega') = 1$ and $\PP(\Omega_t) = 1$ for all $t \in [0,T]$, and, for all $f \in F$, a martingale $M^f$ such that
	\[
		M^f_t(\omega) = \langle M_t(\omega), f \rangle \quad \text{for all } t \in [0,T], \, \omega \in \Omega_t ,\text{ and } f \in F,
	\]
	and such that $t \mapsto M^f_t(\omega)$ is c\`adl\`ag for all $\omega \in \Omega'$.
	
	By the boundedness of $M$, there exists $C > 0$ such that, for all $f,g \in F$, $t \in [0,T]$, and $\omega \in \Omega_t$,
	\[
		|M^f_t(\omega) - M^g_t(\omega)| = |\langle M_t(\omega), f-g \rangle | \le C \norm{f-g}_{C(\RR^d)} \wedge \norm{f-g}_{W^{1,1}},
	\]
	and therefore $f \mapsto M^f_t(\omega)$ extends continuously to an element of $\mcl M_\loc \cap W^{-1,\oo}$, which we denote by $M'_t(\omega)$.
	
	For $\omega \in \Omega'$, let $I(\omega) := \{ t \in [0,T] : \omega \in \Omega_t \}$. Then, by Fubini's theorem, there exists $\Omega'' \subset \Omega'$ of full probability such that, for all $\omega \in \Omega''$, $|[0,T] \backslash I(\omega) | = 0$, and so, for all $\omega \in \Omega''$ all $t \in I(\omega)$ and $f \in C_c \cup W^{1,1}$,
	\[
		M^f_t(\omega) = \langle M'_t(\omega), f \rangle.
	\]
	Because $\omega \in \Omega'$, $t \mapsto M^f_t(\omega)$ is c\`adl\`ag for all $f \in F$, and therefore for all $f \in C_c \cup W^{1,1}$ by the boundedness of $M$ and the density of $F$. Then, because $I(\omega)$ has full measure, it follows that, for any $t \in [0,T]$ and $\omega \in \Omega''$, $f \mapsto M^f_t(\omega)$ is a continuous linear map $M'_t(\omega)$ on $W^{1,1} \cup C_c$. By definition, $M'$ is a progressively measurable martingale taking values in $\mcl M_\loc \cap W^{1,1}$, with the c\`adl\`ag property of \eqref{weakstarcadlag} satisfied, and, for all $t \in [0,T]$ and $f \in W^{1,1} \cup C_c$,
	\[
		\langle M'_t(\omega), f \rangle = \langle M_t(\omega), f \rangle \quad \text{for all } \omega \in \Omega_t \cap \Omega''.
	\]
	Since $\PP(\Omega_t \cap \Omega'') = 1$ for all $t \in [0,T]$, \eqref{version} is satisfied.
\end{proof}

\begin{definition}\label{D:solution}
	The pair $(\widetilde  u,M)$ is a solution of \eqref{E:hjbbody} if
	\begin{enumerate}[(a)]
	\item\label{solutionadapted} $[0,T] \times \Omega \ni (t,\omega) \to \widetilde  u_t \in C(\RR^d)$ is a c\`{a}dl\`{a}g process progressively measurable with respect to the filtration $\FF$, with $\widetilde  u_{T} = \widetilde  G$,
	\item\label{Mmartingale1} $M: [0,T] \times \Omega \to \mcl M_\loc(\RR^d) \cap W^{-1,\oo}(\RR^d)$ is a c\`{a}dl\`{a}g martingale with respect to the filtration $\FF$,
	\item\label{solutionbounds} there exists a deterministic constant $C > 0$ such that, with probability one and for a.e. $t \in [0,T]$,
	\[
		\norm{D\widetilde  u_t}_{\oo} + \esssup m_+(D^2\widetilde  u_t) + \norm{M_t}_{\mcl M_\loc \cap W^{-1,\oo}}  \le C,
	\]
	and
	\item\label{equationsatisfied} the following is satisfied in the distributional sense on $\RR^d$ and  for all $0 \le t \le T$:
	\begin{equation}\label{distribeq}
		\widetilde  u_t = \widetilde  G + \int_t^T \left[ \tr\left(\widetilde  a_r D^2 \widetilde  u_r \right) - \widetilde  H_r(\cdot, D\widetilde  u_r) \right] dr + M_t - M_T.
	\end{equation}
	\end{enumerate}
\end{definition}

\begin{remark}\label{R:testfunctions}
	In view of the assumptions on the martingale $M$ and Lemma \ref{L:SCprops}\eqref{L:TVbound}, the distributional equality \eqref{distribeq} admits test functions $f \in C_c(\RR^d) \cup W^{1,1}(\RR^d)$.
\end{remark}

%
%

The following lemma is a useful tool in the proofs of existence and uniqueness of solutions.

\begin{lemma}\label{L:composition}
	Assume that $[0,T] \times \Omega \ni (t,\omega) \to u_t \in C(\RR^d) \cap L^\oo(\RR^d)$ and $M: [0,T] \times \Omega \to \mcl M_\loc(\RR^d) \cap W^{-1,\oo}(\RR^d)$ are $\FF$-progressively measurable c\`adl\`ag processes, $M$ is a martingale, $\mu \in L^\oo(\Omega \times [0,T],  \mcl M_\loc(\RR^d) \cap W^{-1,\oo}(\RR^d) )$ is $\FF$-progressively measurable, and, for all $0 \le s \le t \le T$, in the sense of distributions on $\RR^d$,
\begin{equation}\label{measureeq}
	u_t - u_s = \int_s^t \mu_r dr + M_t - M_s.
\end{equation}
Let $\phi \in C^2(\RR) \cap W^{2,\oo}(\RR)$ be convex, and set $v_t(x) := \phi(u_t(x))$ for $(t,x) \in [0,T] \times \RR^d$. Then, for all nonnegative $f$ such that $f, \frac{\del f}{\del t} \in C_c([0,T] \times \RR^d)$ or $f \in W^{1,1}([0,T] \times \RR^d)$,
	\begin{align*}
		\EE&\int_{\RR^d} v_t(x)f_t(x)dx - \EE\int_{\RR^d} v_s(x)f_s(x)dx\\
		&\ge \EE \int_s^t \left[ \left\langle \mu_r, \phi'(u_r)f_r \right\rangle + \int_{\RR^d} v_r(x) \frac{\del f_r(x)}{\del r} dx \right]dr.
	\end{align*}
	The same is true if the equality in \eqref{measureeq} is replaced with $\le$ (resp. $\ge$) and $\phi$ is non-increasing (resp. non-decreasing).
\end{lemma}


\begin{proof}
	We prove only the claim when \eqref{measureeq} holds with equality, because the other two claims are argued similarly. Fix a partition $P = \{s = \tau_0 < \tau_1 < \cdots < \tau_N = t\}$ of $[s,t]$. Then, by the convexity of $\phi$ and the nonnegativity of $f$,
	\begin{align*}
		\int_{\RR^d} &v_t(x)f_t(x)dx - \int_{\RR^d} v_s(x)f_s(x)dx\\
		&= \sum_{n=1}^N \left(\int_{\RR^d} (v_{\tau_n}(x) - v_{\tau_{n-1}}(x) ) f_{\tau_{n-1}}(x)dx + \int_{\RR^d} v_{\tau_n}(x) (f_{\tau_n}(x) - f_{\tau_{n-1}}(x)) dx \right)\\
		&\ge \sum_{n=1}^N \left( \int_{\RR^d} (u_{\tau_n}(x) - u_{\tau_{n-1}}(x)) \phi'(u_{\tau_{n-1}}(x)) f_{\tau_{n-1}}(x) dx \right. \\
		& \qquad + \left. \int_{\RR^d} v_{\tau_n}(x) (f_{\tau_n}(x) - f_{\tau_{n-1}}(x)) dx \right).
	\end{align*}
	For every $n= 1,2,\ldots, N$, $\phi'(u_{\tau_{n-1}})$ is $\mbb F_{\tau_{n-1}}$-measurable, and so, by the martingale property for $M$, for $\PP$-a.e. $\omega \in \Omega$,
	\begin{align*}
		\EE \left[ \langle M_{\tau_n} - M_{\tau_{n-1}}, \phi'(u_{\tau_{n-1}}) f_{\tau_{n-1}} \rangle \mid \mbb F_{\tau_{n-1}} \right](\omega)
		= 0.
	\end{align*}
	The nested property of conditional expectations and \eqref{measureeq} yield
	\[
		\EE  \int_{\RR^d} (u_{\tau_n}(x) - u_{\tau_{n-1}}(x)) \phi'(u_{\tau_{n-1}}(x)) f_{\tau_{n-1}}(x) dx  = \EE \int_{\tau_{n-1}}^{\tau_n}\left\langle \mu_r, \phi'(u_{\tau_{n-1}}) f_{\tau_{n-1}}\right\rangle dr,
	\]
	and so
	\begin{align*}
		&\EE \int_{\RR^d} v_t(x)f_t(x)dx - \EE \int_{\RR^d} v_s(x)f_s(x)dx\\
		&\ge \EE \int_s^t \sum_{n=1}^N \ind_{[\tau_{n-1}, \tau_n)}(r) \left( \left\langle \mu_r,\phi'(u_{\tau_{n-1}}) f_{\tau_{n-1}} \right \rangle + \int_{\RR^d} v_{\tau_n}(x) \frac{\del f_r(x)}{\del r} dx \right)dr.
	\end{align*}
	Denote by $\theta^P$ the process
	\[
		\theta^P_r :=  \sum_{n=1}^N \ind_{[\tau_{n-1}, \tau_n)}(r) \left( \left\langle \mu_r,\phi'(u_{\tau_{n-1}}) f_{\tau_{n-1}} \right \rangle + \int_{\RR^d} v_{\tau_n}(x) \frac{\del f_r(x)}{\del r} dx \right), \quad r \in [s,t].
	\]
	Then $\esssup_{(r,\omega) \in [s,t] \times \Omega} |\theta^P_r|$ is bounded independently of $P$, and, as $|P| \to 0$, with probability one and for almost every $r \in [s,t]$, $\theta^P_r$ converges to
	\[
		\theta_r := \left\langle \mu_r, \phi'(u_r) f \right \rangle + \int_{\RR^d} v_r(x) \frac{\del f_r(x)}{\del r} dx.
	\]
	We conclude upon sending $|P| \to 0$ and appealing to the dominated convergence theorem.
\end{proof}

We will apply the preceding lemma with a particular class of test functions. For $\lambda > 0$ and $\mcl K > 0$, define
\begin{equation}
	\psi_{\lambda,\mcl K}(t,x) := \exp\left( - \frac{ [(|x| - \mcl Kt)_+]^2}{4\lambda(t+1)} \right), \quad (t,x) \in [0,T] \times \RR^d. 
\end{equation}

\begin{lemma}\label{L:dualproblem}
	For any $\lambda > 0$ and $\mcl K > 0$, $\psi_{\lambda,K}$ belongs to $C^2([0,T] \times \RR^d) \cap W^{2,1}([0,T] \times \RR^d)$ and satisfies
	\[
		\frac{\partial \psi_{\lambda,\mcl K}}{\partial t} - \lambda m_+(D^2 \psi_{\lambda,\mcl K}) - \mcl K|D\psi_{\lambda,\mcl K}| \ge 0 \quad \text{in } [0,T] \times \RR^d.
	\]
\end{lemma}

\begin{proof}
	It is clear that the derivatives of $\psi_{\lambda,\mcl K}$ are continuous and belong to $L^1([0,T] \times \RR^d)$. To check the differential inequality, we compute
	\[
		\frac{\partial \psi_{\lambda,\mcl K}}{\partial t} = \exp\left( - \frac{ [ (|x| - \mcl K t)_+]^2}{4\lambda(t+1)} \right)\left[ \frac{[(|x|-\mcl Kt)_+]^2}{4\lambda(t+1)^2}+ \frac{\mcl K (|x| - \mcl Kt)_+}{2\lambda(t+1)} \right],
	\]
	\[
		D\psi_{\lambda,\mcl K}(t,x) = -\exp\left( - \frac{[ (|x| - \mcl Kt)_+]^2}{4\lambda(t+1)} \right) \frac{(|x| - \mcl Kt)_+}{2\lambda(t+1)} \hat x,
	\]
	and
	\begin{align*}
		D^2\psi_{\lambda,\mcl K}(t,x) &= \sgn_+(|x| - \mcl Kt) \exp\left( - \frac{[ (|x| - \mcl Kt)_+]^2}{4\lambda(t+1)} \right) \\
		&\cdot \left[ - \frac{|x| - \mcl Kt}{2|x| \lambda(t+1)} \Id + \left( \frac{ (|x| - \mcl Kt)^2}{4\lambda^2(t+1)^2} + \frac{|x| - \mcl Kt}{2|x|\lambda(t+1)} - \frac{1}{2\lambda(t+1)} \right) \hat x \otimes \hat x \right].
	\end{align*}
	The identity $m_+\left( A \hat x \otimes \hat x + B \Id \right) = (B + A_+)_+$ for all $A,B \in \RR$ and $x \in \RR^d$ yields
	\begin{align*}
		m_+(D^2\psi_{\lambda,\mcl K}(t,x))
		&=  \sgn_+(|x| - \mcl Kt) \exp\left( - \frac{[ (|x| - \mcl Kt)_+]^2}{4\lambda(t+1)} \right)\\
		&\cdot \left[ - \frac{|x| - \mcl Kt}{2|x| \lambda(t+1)} + \left( \frac{ (|x| - \mcl Kt)^2}{4\lambda^2(t+1)^2} + \frac{|x| - \mcl Kt}{2|x|\lambda(t+1)} - \frac{1}{2\lambda(t+1)} \right)_+ \right]_+\\
		&= \sgn_+(|x| - \mcl Kt) \exp\left( - \frac{ [(|x| - \mcl Kt)_+]^2}{4\lambda(t+1)} \right) \left( \frac{ (|x| - \mcl Kt)^2}{4\lambda^2(t+1)^2} - \frac{1}{2\lambda(t+1)} \right)_+.
	\end{align*}
	We conclude that
	\begin{align*}
		&\frac{\partial \psi_{\lambda,\mcl K}}{\partial t} - \lambda m_+(D^2 \psi_{\lambda,\mcl K}) - \mcl K|D\psi_{\lambda,\mcl K}| \\
		&\ge \sgn_+(|x| - \mcl Kt) \exp\left( - \frac{ [(|x| -\mcl Kt)_+]^2}{4\lambda(t+1)} \right)
		\Bigg[  \frac{(|x|-\mcl Kt)^2}{4\lambda(t+1)^2} -  \left( \frac{ (|x| - \mcl Kt)^2}{4\lambda(t+1)^2} - \frac{1}{2(t+1)} \right)_+ \\
		&+ \frac{\mcl K (|x| - \mcl Kt)}{2\lambda(t+1)} - \frac{\mcl K (|x| - \mcl Kt)}{2\lambda(t+1)} \Bigg] \ge 0.
	\end{align*}
\end{proof}

%


\subsection{Existence} Using a discretization procedure, we prove here the existence of a solution of \eqref{E:hjbbody}.

\begin{theorem}\label{T:existence}
	Assume \eqref{A:filtration} and \eqref{A:stochHJB}. Then there exists a solution $(\widetilde  u,M)$ of \eqref{E:hjbbody} in the sense of Definition \ref{D:solution}.
\end{theorem}

Fix $N \in \NN$ and set 
\begin{equation}\label{partition}
	t_n^N := \frac{nT}{2^N} \quad \text{for } n =0,1,2,\ldots, 2^N.
\end{equation}
Define the time-discretized diffusion and Hamiltonian by
\begin{equation}\label{discretea}
	\widetilde  a^N_t := \widetilde  a_{t_{n}^N} \quad \text{for } t \in \left[ t_{n-1}^N, t_{n}^N \right), \quad n = 1,2,\ldots, 2^N
\end{equation}
and
\begin{equation}\label{discreteH}
	\widetilde  H^N_t := \widetilde  H_{t_{n}^N} \quad \text{for } t \in \left[ t_{n-1}^N, t_{n}^N \right), \quad n = 1,2,\ldots, 2^N,
\end{equation}
as well as the right-continuous, piecewise-constant-in-time filtration $\FF^N = (\FF^N_t)_{t \in [0,T]}$ by
\begin{equation}\label{discreteF}
	\mbb F^N_t := \mbb F_{t_{n}^N} \quad \text{for } t \in \left[ t_{n-1}^N, t_{n}^N \right), \, n = 1,2,\ldots, 2^N, \quad \mbb F^N_T = \mbb F_T.
\end{equation}
Observe that, in view of the right-continuity of $\FF$ implied by \eqref{A:filtration},
\begin{equation}\label{alldiscreteF}
	\FF^{N+1} \subset \FF^N \quad \text{for all } N \text{ and } \FF_t = \bigcap_{N \in \NN} \FF^N_t \quad \text{for all } t \in [0,T].
\end{equation}

We now introduce  a c\`adl\`ag-in-time function $\widetilde  u^N: [0,T] \times \RR^d \times \Omega \to \RR$, with jumps occurring at the partition points \eqref{partition}, as follows. We first set $\widetilde  u^N_T := \widetilde  G$ in $\RR^d$. Then, for $n = 1,2,\ldots, 2^N$, given $\widetilde  u^N_{t_{n}^N}$, define $\widetilde  u^N$ in $\left[ t_{n-1}^N, t_{n}^N \right)$ as the solution of the terminal value problem
\begin{equation}\label{E:discreteTVP}
	\left\{
	\begin{split}
	&- \frac{\partial \widetilde  u^N_t}{\partial t} - \tr[ \widetilde  a^N_t D^2 \widetilde  u^N_t] + \widetilde  H^N_t(x,D\widetilde  u^N_t) = 0 \quad \text{in } \left[ t_{n-1}^N, t_{n}^N \right) \times \RR^d \quad \text{and}\\
	&\widetilde  u^N_{t_{n}^N,-} = \EE\left[ \widetilde  u^N_{t_{n}^N} \; \Big| \; \mbb F_{t_{n}^N} \right] \quad \text{in } \RR^d.
	\end{split}
	\right.
\end{equation}
In other words, the jump discontinuity of $\widetilde  u^N$ at $t_{n}^N$ is of size
\[
	\Delta M^N_{t_{n}^N} := \widetilde  u^{N}_{t_{n}^N} - \EE\left[ \widetilde  u^N_{t_{n}^N} \; \Big| \; \mbb F^N_{t_{n-1}^N} \right], \quad n = 1, 2, \ldots, 2^{N}-1.
\]
Finally, define
\begin{equation}\label{discreteM}
	M^N_t := \sum_{n : t_n^N \le t} \Delta M^N_{t_n^N} \quad \text{for } t \in [0,T].
\end{equation}

\begin{lemma}\label{L:discretesolution} The processes constructed above satisfy the following:
	\begin{enumerate}[(a)]
	\item\label{L:adapted} The process $(\widetilde  u^N_t)_{t \in [0,T]}$ is $\FF^N$-adapted, and the process $(M^N_t)_{t \in [0,T]}$ is a $\FF^N$-martingale.
	\item\label{L:bounds} There exists a constant $C > 0$ depending only on $T$ and the data (through the assumptions listed in the statement of the theorem) such that
	\[
		\sup_{N \in \NN} \left( \norm{\widetilde  u^N}_{\oo,[0,T] \times \Omega \times \RR^d} + \norm{D\widetilde  u^N}_{\oo, [0,T] \times \Omega \times \RR^d} + \esssup_{[0,T] \times \Omega \times \RR^d} m_+(D^2 \widetilde  u^N) \right)  \le  C.
	\]
	\item\label{L:equation} For all $t \in [0,T]$ and  in the sense of distributions in $\RR^d$,
	\[
		\widetilde  u^N_t = \widetilde  G + \int_t^T \left(\tr[\widetilde  a^N_s D^2\widetilde  u^N_s] - \widetilde  H^N_s(\cdot, D\widetilde  u^N_s) \right)ds - M^N_T + M^N_t.
	\]
	\end{enumerate}
\end{lemma}

\begin{proof}
	We prove only \eqref{L:bounds}, and then properties \eqref{L:adapted} and \eqref{L:equation} are easily checked. 

	Fix $n= 0,1,\ldots, 2^N-1$. Combining the bound $\norm{\widetilde  u^N_{t_{n+1}^N,-}}_{\oo} \le \norm{\widetilde  u^N_{t_{n+1}^N}}_{\oo}$ with \eqref{detubound} from Proposition \ref{P:dethjb} yields, for some $C > 0$ as in the statement of the lemma,
	\[
		\norm{\widetilde  u^N_t}_{\oo} \le \norm{\widetilde  u^N_{t_{n+1}^N}}_{\oo} + C(t_{n+1} - t) \quad \text{for all } t \in [t_n^N,t_{n+1}^N).
	\]
	We then see inductively that, for all  $t \in [0,T]$,
	\[
		\norm{\widetilde  u^N_t}_{\oo} \le \norm{G}_{\oo} + C(T-t). 
	\]
	A similar argument, appealing to \eqref{detD2ubound} from Proposition \ref{P:dethjb}, gives, for all $t \in [0,T]$ and almost everywhere in $\RR^d \times \Omega$,
	\begin{align*}
		&\left(m_+\left( \frac{dD^2 \widetilde  u^N_t}{d \mcl L} \right)^2 + |D\widetilde  u^N_t|^2 \right)^{1/2} - \sqrt{2}\lambda_0 \widetilde  u^N_t\\
		&\qquad\le \frac{T}{2^N}\sum_{k=0}^{2^N-1} e^{Ck/2^N} + e^{CT} \esssup_{\RR^d \times \Omega} \left\{   \left(m_+\left( \frac{dD^2 G}{d \mcl L} \right)^2 + |DG|^2 \right)^{1/2} -\sqrt{2} \lambda_0 G \right\}_+.
	\end{align*}
	Combining the above  with Lemma \ref{L:SCprops}\eqref{L:interpolation2}, the bound for $\norm{\widetilde  u}_\oo$, and
	\[
		\frac{T}{2^N}\sum_{k=0}^{2^N-1} e^{CkT/2^N} = \frac{T(e^{CT} - 1)}{2^N(e^{CT/2^N} - 1)} \le \frac{e^{CT} - 1}{C}
	\]
	shows that
	\[
		\sup_{N \in \NN} \esssup_{[0,T] \times \RR^d \times \Omega} m_+\left( \frac{d D^2 \widetilde  u^N_t}{d \mcl L} \right) < \oo,
	\]
	so that, by Lemma \ref{L:SCprops}\eqref{L:semiconcave} and \eqref{L:interpolation2}, $\widetilde  u^N$ is uniformly semiconcave and Lipschitz.
\end{proof}

The proof of Theorem \ref{T:existence}, as well as for the stability estimates to follow, involve bounding the difference between uniformly Lipschitz solutions in certain Lebesgue spaces. The following elementary lemma allows to translate the last bound to an estimate in the local-uniform topology.

\begin{lemma}\label{L:interpolation}
	Let $1 \le p < \oo$. Then there exists $C = C_{d,p} > 0$ such that, for all $w \in W^{1,\oo}(\RR^d) \cap L^p(\RR^d)$,
	\[
		\norm{w}_{\oo} \le C \norm{Dw}_{\oo}^{\frac{d}{d+p}} \norm{w}_{p}^{\frac{p}{d+p}}.
	\]
\end{lemma}

\begin{proof}
	Fix $x \in \RR^d$. The function $\mcl C_x(y) := \left( |w(x)| - \norm{Dw}_{\oo} |y-x| \right)_+$ lies below the graph of $|w|$, which implies that $\norm{\mcl C_x}_{p} \le \norm{w}_{p}$. For some $c = c_{d,p} > 0$, we compute $\norm{\mcl C_x}_{p} = c |w(x)|^{d+p} \norm{Dw}_{\oo}^{-d}$, and the result follows upon rearranging terms and taking the supremum over $x \in \RR^d$.
\end{proof}

\begin{proof}[Proof of Theorem \ref{T:existence}]
	Fix $K,N \in \NN$ with $N < K$, let $\phi: \RR \to \RR$ be smooth, nonnegative, nonincreasing, and convex, and set
	\[
		w^{N,K}_t(x) := \phi\left( \widetilde  u^N_t(x) - \widetilde  u^K_t(x)\right) \quad \text{for } (t,x) \in [0,T] \times \RR^d.
	\]
	Let $f: [0,T] \times \RR^d \to \RR$ be smooth and nonnegative, with sufficient decay at infinity (to be determined below). For $0 \le t < t+h \le T$,
	\begin{equation}\label{dwNK}
	\begin{split}
		\int_{\RR^d}& \left( w^{N,K}_{t+h}(x) f_{t+h}(x) - w^{N,K}_t(x) f_t(x)\right)  dx\\
		&= \int_t^{t+h} \Bigg[ - \langle \tr[\widetilde  a^N_s D^2\widetilde  u^N_s - \widetilde  a^K_s D^2 \widetilde  u^K_s ], f_s \phi'(\widetilde  u^N_s - \widetilde  u^K_s) \rangle \\
		&\qquad + \int_{\RR^d} \Bigg( [\widetilde  H^N_s(x,D\widetilde  u^N_s(x)) - \widetilde  H^K_s(x,D\widetilde  u^K_s(x)) ] f_s(x) \phi'(\widetilde  u^N_s(x) - \widetilde  u^K_s(x)) \\
		&\qquad \qquad+ w^{N,K}_s(x) \frac{\del f_s(x)}{\del s}\Bigg)dx\Bigg] ds \\
		&+ \sum_{t_n^N \in (t,t+h]} \int_{\RR^d} \Delta w^{N,K}_{t_n^N}(x)f_{t_n^N}(x)dx + \sum_{t_n^K \in (t,t+h]} \int_{\RR^d} \Delta w^{N,K}_{t_n^K}(x) f_{t_n^K}(x)dx,
	\end{split}
	\end{equation}
	where
	\[
		\Delta w^{N,K}_{t_n^N} := \phi\left(\widetilde  u^N_{t_n^N} - \widetilde  u^K_{t_n^N}\right) - \phi\left(\widetilde  u^N_{t_n^N,-} - \widetilde  u^K_{t_n^N,-}\right)
	\]
	and
	\[
		\Delta w^{N,K}_{t_n^K} := \phi\left(\widetilde  u^N_{t_n^K} - \widetilde  u^K_{t_n^K}\right) - \phi\left(\widetilde  u^N_{t_n^K,-} - \widetilde  u^K_{t_n^K,-}\right).
	\]
	The convexity of $\phi$ implies that
	\[
		\Delta w^{N,K}_{t_n^N} \ge \phi'\left( \widetilde  u^N_{t_n^N,-} - \widetilde  u^K_{t_n^N,-} \right) \left( \Delta M^N_{t^N_n} - \Delta M^K_{t^N_n} \right)
	\]
	and
	\[
		\Delta w^{N,K}_{t_n^K} \ge \phi'\left( \widetilde  u^N_{t_n^K,-} - \widetilde  u^K_{t_n^K,-}\right) \left( \Delta M^N_{t^K_n} - \Delta M^K_{t^K_n} \right).
	\]
	By definition, $\phi'( \widetilde  u^N_{t_n^N,-} - \widetilde  u^K_{t_n^N,-})$ is $\mbb F_{t_n^N}$-measurable, and so
	\[
		\EE \left[ \Delta w^{N,K}_{t_n^N} \;\Big|\; \mbb F_{t_n^N} \right] \ge 0.
	\]
	If $t^K_n$ is not equal to $t^N_m$ for some $m = 1,2,\ldots, 2^N - 1$, then $\Delta M^N_{t^K_n} = M^N_{t^K_n} - M^N_{t^K_n,-} = 0$. Otherwise, $\phi'( \widetilde  u^N_{t_n^K,-} - \widetilde  u^K_{t_n^K,-})$ is $\mbb F_{t_n^K}$-measurable. In any case, we have
	\[
		\EE \left[ \Delta w^{N,K}_{t_n^K} \; \Big| \;  \mbb F_{t_n^K} \right] \ge 0.
	\]
	Upon taking the expectation of \eqref{dwNK}, the nested property of conditional expectation, the convexity of $\phi$, and the nonnegativity of $\widetilde  a^N$ imply that
	\begin{align*}
		\EE \int_{\RR^d}& \left( w^{N,K}_{t+h}(x)f_{t+h}(x) - w^{N,K}_t(x)f_t(x) \right) dx\\
		&\ge \int_t^{t+h} \EE\Bigg[ - \left\langle \tr[\widetilde  a^N_s D^2\widetilde  u^N_s - \widetilde  a^K_s D^2 \widetilde  u^K_s ], f_s \phi'(\widetilde  u^N_s - \widetilde  u^K_s) \right\rangle \\
		&\qquad + \int_{\RR^d} \Bigg( [\widetilde  H^N_s(x,D\widetilde  u^N_s(x)) - \widetilde  H^K_s(x,D\widetilde  u^K_s(x)) ] f_s(x) \phi'(\widetilde  u^N_s(x) - \widetilde  u^K_s(x))\\
		&\qquad \qquad + w^{N,K}_s(x) \frac{\del f_s(x)}{\del s} \Bigg)dx\Bigg] ds \\
		&\ge \int_t^{t+h} \EE \Bigg[  \left\langle - \tr[\widetilde  a^N_s D^2 \widetilde  w^{N,K}_s] + \zeta^{N,K}_s, f_s \right\rangle \\
		&+ \int_{\RR^d} \left(b^{N,K}_s(x) \cdot Dw^{N,K}_s(x) f_s(x) + w^{N,K}_s(x) \frac{\del f_s(x)}{\del s} \right)dx \Bigg]ds,
	\end{align*}
	where
	\[
		b^{N,K}_s := \int_0^1 D_p \widetilde  H^N_s(\cdot, \tau D\widetilde  u^N_s + (1-\tau) D\widetilde  u^K_s) d\tau
	\]
	and
	\[
		\zeta^{N,K}_s :=  \phi'(\widetilde  u^N_s - \widetilde  u^K_s) \left[ -\tr(\widetilde  a^N_s - \widetilde  a^K_s)D^2 \widetilde  u^K_s + \widetilde  H^N_s(\cdot, D\widetilde  u^K_s) - \widetilde  H^K_s(\cdot, D\widetilde  u^K_s)\right].
	\]
	For $\delta > 0$, let $\rho_\delta$ be as in \eqref{mollifier} and define
	\[
		b^{N,K,\delta}_s := \int_0^1 D_p \widetilde  H_s^N \left( \cdot, \tau D(\rho_\delta * \widetilde  u^N_s) + (1-\tau) D(\rho_\delta * \widetilde  u^K_s) \right)d\tau
	\]
	and
	\[
		\zeta^{N,K,\delta}_s := \zeta^{N,K}_s + \left( b^{N,K}_s - b^{N,K,\delta}_s \right) \cdot Dw^{N,K}_s.
	\]
	Then
	\begin{align*}
		&\sup_{t \in [0,T]} \nor{b_t^{N,K,\delta}}{\oo} \le \nor{b}{\oo} \quad \text{and}\\
		&\lim_{\delta \to 0} b^{N,K,\delta} = b^{N,K} \quad \text{almost everywhere in $[0,T] \times \RR^d \times \Omega$},
	\end{align*}
	and, in view of the convexity of $\widetilde  H$, the uniform semiconcavity bounds for $\widetilde  u^N$ and $\widetilde  u^K$ implied by Lemma \ref{L:discretesolution}\eqref{L:bounds}, and \eqref{A:MAINdiffusion},
	\[
		\sup_{t \in [0,T]} \sup_{x \in \RR^d} \sup_{\delta > 0} \sup_{N,K} \left(  \div b^{N,K,\delta}_t(x) + \div\div \widetilde  a^{N}_t(x) \right) < \oo.
	\]
	We next write
	\begin{align*}
		&\EE \int_{\RR^d} \left( w^{N,K}_{t+h}(x)f_{t+h}(x) - w^{N,K}_t(x)f_t(x) \right) dx\\
		&\ge \int_t^{t+h} \EE \Bigg[  \left\langle - \tr[\widetilde  a^{N}_s D^2 w^{N,K}_s] + \zeta^{N,K,\delta}_s, f_s \right\rangle \\
		&+ \int_{\RR^d} \Bigg(b^{N,K,\delta}_s(x) \cdot Dw^{N,K}_s(x) f_s(x) + w^{N,K}_s(x) \frac{\del f_s(x)}{\del s} \Bigg)dx \Bigg]ds\\
		&=\int_t^{t+h} \EE \Bigg[ \int_{\RR^d} w_s^{N,K}(x) \Bigg( \frac{\del f_s(x)}{\del s} -\tr[ D^2(\widetilde  a^{N}_s(x) f_s(x))] \\
		&- b^{N,K,\delta}_s(x) \cdot Df_s(x) - (\div b^{N,K,\delta}_s(x)) f_s(x) \Bigg)dx + \left \langle \zeta^{N,K,\delta}_s, f_s \right \rangle \Bigg] ds\\
		&= \int_t^{t+h} \EE \Bigg[ \int_{\RR^d} w_s^{N,K}(x) \Bigg( \frac{\del f_s(x)}{\del s} - \tr[ \widetilde  a^{N}_s(x) D^2 f_s(x) ]\\
		& - \left( b^{N,K,\delta}_s(x) + 2 \div \widetilde  a^{N}_s(x) \right) \cdot Df_s(x) - \left( \div b^{N,K,\delta}(x) + \div \div \widetilde  a^{N}_s(x) \right) f_s(x) \Bigg)dx \\
		&\qquad \qquad+ \left \langle \zeta^{N,K,\delta}_s, f_s \right \rangle \Bigg] ds.
	\end{align*}
	For suitable $\lambda$ and $\mcl K$ depending only on $\nor{\widetilde  a}{C^{1,1}}$, the constant $C$ from Lemma \ref{L:discretesolution}, and the local bound for $D_p \widetilde  H$ and $D_x \widetilde  H$ from \eqref{A:MAINHamiltonian}, we take $f = \psi_{\lambda,\mcl K}$ from Lemma \ref{L:dualproblem} and set
	\[
		\gamma^{N,K}_t := \EE \int_{\RR^d} w^{N,K}_s(x) \psi_{\lambda,\mcl K,s}(x) dx.
	\]
	Then
	\[
		\dot \gamma^{N,K}_t \ge -C \gamma^{N,K}_t + \eps^{N,K,\delta}_t,
	\quad \text{where} \quad
		\eps^{N,K,\delta}_t := \EE \left \langle \zeta^{N,K,\delta}_t, \psi_{\lambda,\mcl K,t} \right \rangle.
	\]
	The uniform Lipschitz bound for $w^{N,K}$, the fact that $\psi_{\lambda,\mcl K} \in W^{2,1}$, and the dominated convergence theorem yield, upon sending $\delta \to 0$,
	\[
		\dot \gamma^{N,K}_t \ge -C \gamma^{N,K}_t + \eps^{N,K}_t,
	\quad \text{where} \quad
		\eps^{N,K}_t := \EE \left \langle \zeta^{N,K}_t, \psi_{\lambda,\mcl K,t} \right \rangle.
	\]
	We have $\gamma^{N,K}_T = 0$, and so Gr\"onwall's inequality and the continuity of $\widetilde  H$ and $\widetilde  a$ in time give, for all $t \in [0,T]$ and some sequence $\omega_N$ satisfying $\lim_{N \to \oo} \omega_N  =0$,
	\[
		\gamma^{N,K}_t \le \widetilde  C \int_t^T \eps^{N,K}_s ds \le \omega_N.
	\]
	We then take $\phi(r) := r_-$ (which can be justified with a smooth approximation) and, upon switching the roles of $\widetilde  u^N$ and $\widetilde  u^K$, conclude that
	\[
		\sup_{t \in [0,T]}\EE  \int_{\RR^d} | \widetilde  u^N_t(x) - \widetilde  u^K_t(x)| \psi_{\lambda,\mcl K,t}(x) dx 
		\le \omega_N.
	\]
	Lemma \ref{L:interpolation} with $p = 1$ and the uniform Lipschitz bounds for $\widetilde  u^N$, $\widetilde  u^K$, and $\psi_{\lambda,\mcl K}$ then give a constant, independent of $N$ and $K$, such that
	\begin{equation}\label{uniformintbound}
		\sup_{t \in [0,T]} \EE \sup_{x \in \RR^d} \psi_{\lambda,\mcl K,t}(x)^{d+1} |\widetilde  u^N_t(x) - \widetilde  u^K_t(x)|^{d+1} \le C\omega_N.
	\end{equation}
	It follows that, as $N \to \oo$, $\widetilde u^N$ converges, in $L^\oo([0,T], L^{d+1}(\Omega, C(B_R)))$ for all $R > 0$, to some limit $\widetilde u$. Upon extracting a sub-sequence, we deduce that, as $N \to \oo$, $\PP$-almost surely and for a.e. $t \in [0,T]$, $\widetilde  u^N_t$ converges locally uniformly to $\widetilde  u_t$, and, moreover, because of the uniform semiconcavity bound and Lemma \ref{L:SCprops}\eqref{L:Duconverge}, $D\widetilde  u^N_t$ converges almost everywhere to $D\widetilde  u_t$. In view of \eqref{alldiscreteF}, the process $\widetilde u$ is progressively measurable with respect to $\mbb F$.

	For any $t \in [0,T]$, in the sense of distributions,
	\[
		M^N_t = \widetilde  u^N_t - \widetilde  u^N_0 - \int_0^t \left[ \tr[ \widetilde  a^N_s D^2 \widetilde  u^N_s ] - \widetilde  H^N_s(\cdot, D\widetilde  u^N_s) \right] ds.
	\]
	As $N \to \oo$, for a.e. $t \in [0,T]$ and $\PP$-a.s., the right-hand side converges in the distributional sense to
	\begin{equation}\label{MGidentity}
		M_t := \widetilde  u_t - \widetilde  u_0 - \int_0^t \left[ \tr[ \widetilde  a_s D^2 \widetilde  u_s ] - \widetilde  H_s(\cdot, D\widetilde  u_s) \right] ds \in \mcl M_\loc(\RR^d) \cap W^{-1,\oo}(\RR^d).
	\end{equation}
	The process $M$ is progressively measurable with respect to $\FF$ because $\widetilde u$ is.
	
	For fixed $f \in C_c(\RR^d) \cup W^{1,1}(\RR^d)$, $(\langle M^N_{t}, f \rangle)_{t \in [0,T]}$ is a martingale with respect to the filtration $\mbb F^N$. Fix $0 \le s < t \le T$ and $\mcl A \in \mbb F_s$. For any $N \in\NN$, $\mbb F_s \subset \mbb F^N_s$, and so
	\[
		\EE\left[ \langle M^N_t - M^N_s,f \rangle \ind_{\mcl A} \right] = 0.
	\]
	Sending $N \to \oo$ and appealing to Lemma \ref{L:discretesolution}\eqref{L:bounds} and the dominated convergence theorem yields $\EE\left[ \langle M_t - M_s, f \rangle \ind_{\mcl A} \right] = 0$. We conclude that $M$ is a progressively measurable martingale, and then, by Lemma \ref{L:cadlag}, $M$ has a version such that, with probability one, $\langle M', f \rangle$ is c\`adl\`ag for all $f \in C_c \cup W^{1,1}$.
		
	For $t \in [0,T]$, define 
	\[
		\widetilde{u}'_t := \widetilde{G} + M'_t - M'_T - \int_t^T \left[ \tr[ \widetilde  a_s D^2 \widetilde  u_s ] - \widetilde  H_s(\cdot, D\widetilde  u_s) \right] ds.
	\]
	Then, for all $t \in [0,T]$, there exists $\Omega_t \in \mbb F_t$ with $\PP(\Omega_t) = 1$ such that $\widetilde{u}'_t(\omega) = \widetilde{u}_t(\omega)$ for all $\omega \in \Omega_t$, and so the same boundedness, Lipschitz regularity, and semiconcavity properties are satisfied by $\widetilde u'$. It follows from Fubini's theorem, that for $\PP$-a.e. $\omega \in \Omega$, $M'$ is right-continuous and
	\[
		|\{ t \in [0,T] : \omega \notin \Omega_t \} | = 0.
	\]
	Therefore, in the integration over $s \in [t,T]$ above, we may replace $D\widetilde u_s$ and $D^2 \widetilde u_s$ with respectively $D\widetilde u_s'$ and $D^2\widetilde u_s'$ without affecting the value of the integral. For the rest of the proof, we thus work with the versions of $M$ and $\widetilde u$ for which \eqref{MGidentity} holds and $\langle M, f \rangle$ c\`adl\`ag for all $f \in C_c \cup W^{1,1}$.
	
	We next note that, by the semiconcavity and Lipschitz bounds for $\widetilde u$,
	\[
		t \mapsto \int_0^t \left[ \tr[ \widetilde  a_s D^2 \widetilde  u_s ] - \widetilde  H_s(\cdot, D\widetilde  u_s) \right] ds
	\]
	is Lipschitz continuous with respect to the strong topology of $\mcl M_\loc \cap W^{-1,\oo}$, and therefore, for all $f \in W^{1,1} \cup C_c$, $t \mapsto \langle \widetilde{u}'_t, f \rangle$ is c\`adl\`ag. Exploiting the uniform Lipschitz continuity of $\widetilde u$ in space, it follows that, in fact, $t \mapsto \widetilde{u}_t$ is c\`adl\`ag with respect to local-uniform convergence in $C(\RR^d)$. We conclude that $t \mapsto M_t$ is c\`adl\`ag with respect to the strong topology of $\mcl M_\loc \cap W^{-1,\oo}$ induced by the quasinorm $\norm{\cdot}_{\mcl M_\loc \cap W^{-1,\oo}}$.
\end{proof}

\subsection{Sub/supersolutions and a comparison principle}

We relax the notion of solutions of \eqref{E:hjbbody} with the following definition.

\begin{definition}\label{D:subsupersolution}
	The pair $(\widetilde  u,M)$ is a sub- (resp. super-) solution of \eqref{E:hjbbody} if \eqref{solutionadapted}, \eqref{Mmartingale1}, and \eqref{solutionbounds} from Definition \ref{D:solution} are satisfied, and if, with probability one and for $0 \le t_1 \le t_2 \le T$, in the sense of distributions on $\RR^d$,
	\begin{equation}\label{subsolution}
		\widetilde  u_{t_1} \le \widetilde  u_{t_2} + \int_{t_1}^{t_2} \left[  \tr(\widetilde  a_s D^2 \widetilde  u_s) - \widetilde  H_s(\cdot, D\widetilde  u_s) \right]ds - M_{t_2} + M_{t_1}
	\end{equation}
	(resp.
	\begin{equation}\label{supersolution}
		\left. \widetilde  u_{t_1} \ge \widetilde  u_{t_2} + \int_{t_1}^{t_2} \left[  \tr(\widetilde  a_s D^2 \widetilde  u_s) - \widetilde  H_s(\cdot, D\widetilde  u_s) \right]ds - M_{t_2} + M_{t_1} \right).
	\end{equation}
\end{definition}

Uniqueness of solutions for \eqref{E:hjbbody} then follows from the following theorem and the fact that solutions are both sub- and super-solutions.

\begin{theorem}\label{T:comparison}
	Assume \eqref{A:filtration}, $(\widetilde  a,\widetilde  H^1,\widetilde  G^1)$ and $(\widetilde  a,\widetilde  H^2,\widetilde  G^2)$ satisfy \eqref{A:stochHJB}, $\widetilde  H^1 \le \widetilde  H^2$, and $\widetilde  G^1 \ge \widetilde  G^2$, and let $(\widetilde  u^1,M^1)$ and $(\widetilde  u^2,M^2)$ be respectively a super- and sub-solution of \eqref{E:hjbbody} corresponding to $(\widetilde  H^1,\widetilde  G^1)$ and $(\widetilde  H^2,\widetilde  G^2)$. Then $\widetilde  u^1 \ge \widetilde  u^2$.
\end{theorem}

\begin{proof}
Let $v = \widetilde  u^1 - \widetilde  u^2$ and $M = M^1 - M^2$, and, for $(t,x) \in [0,T] \times \RR^d$, define
\[
	b_t(x) := \int_0^1 D\widetilde  H^1_t\left( x,\tau D\widetilde  u^1_t(x) + (1-\tau) D\widetilde  u^2_t(x) \right)d\tau.
\]
Then
\[
	-dv_t \ge \left[ \tr(\widetilde  a_t D^2 v_t) - b_t \cdot Dv_t \right] ds - dM_t \quad \text{in }[0,T].
\]
Let $\phi: \RR \to \RR$ be $C^2$, non-negative, non-increasing, and convex, and set $w = \phi(v)$. Then, by Lemma \ref{L:composition}, for any nonnegative $f \in W^{2,1}([0,T] \times \RR^d) \cap C^2([0,T] \times \RR^d)$ and $0 \le s < t \le T$,
\begin{equation}\label{dgamma1}
\begin{split}
	\EE&\int_{\RR^d} w_t(x) f_t(x)dx - \EE\int_{\RR^d} w_s(x) f_s(x)dx\\
	&\ge \int_s^t \EE \Bigg[ -\left\langle \tr[\widetilde  a_r D^2v_r ], \phi'(v_r)f_r \right \rangle + \int_{\RR^d} \Bigg( \phi'(v_r(x)) b_r(x) \cdot Dv_r(x)f_r(x) \\
	&\qquad\qquad + w_r(x) \frac{\del f_r(x)}{\del r} \Bigg) dx \Bigg] dr\\
	&\ge \int_s^t \EE \Bigg[ -\left\langle \tr[\widetilde  a_r D^2 w_r ],f_r \right \rangle + \int_{\RR^d} \Bigg( b_r(x) \cdot Dw_r(x)f_r(x) \\
	&\qquad \qquad+ w_r(x) \frac{\del f_r(x)}{\del r} \Bigg) dx \Bigg] dr.
\end{split}
\end{equation}
Fix $\delta > 0$, let $\rho_\delta$ be as in \eqref{mollifier}, and, for $t \in [0,T]$, set
\[
	b^\delta_t := \int_0^1 D\widetilde  H^1_t\left( \cdot, \tau D(\rho_\delta *\widetilde  u^1_t) + (1-\tau) D(\rho_\delta * \widetilde  u^2_t)\right)d\tau.
\]
Then, as $\delta \to 0$, $b^\delta \to b$ almost everywhere in $[0,T] \times \RR^d \times \Omega$. Moreover, because $\widetilde  H^1$ is convex and $u^1$ and $u^2$ are globally Lipschitz and semiconcave, we have, for some $C > 0$ independent of $\delta$, 
\[
	\nor{b^\delta}{\oo} \le C \quad \text{and} \quad \div b^\delta \le C.
\]
For $r \in [0,T]$, set $\zeta^\delta_r :=  - (b_r - b^\delta_r) \cdot Dw_r$. Continuing \eqref{dgamma1}, we have
\begin{align*}
	&\EE\int_{\RR^d} w_t(x) f_t(x)dx - \EE\int_{\RR^d} w_s(x) f_s(x)dx\\
	&\ge \int_s^t \EE \Bigg[ -\left\langle \tr[\widetilde  a_r D^2 w_r ] + \zeta^\delta_r,f_r \right \rangle + \int_{\RR^d} \Bigg( \div (b^\delta_r(x) w_r(x) )f_r(x) \\
	&\qquad \qquad + w_r(x) \left[\frac{\del f_r(x)}{\del r} - \div b^\delta_r(x) f_r(x) \right] \Bigg) dx \Bigg] dr\\
	&= \int_s^t \EE \Bigg[ \int_{\RR^d} w_r(x) \bigg( \frac{\del f_r(x)}{\del r} -  \div \div(\widetilde  a_r(x) f_r(x) ) - b_r^\delta(x) \cdot Df_r(x)\\
	& \qquad \qquad - \div b^\delta_r(x) f_r(x) \bigg) dx - \left \langle \zeta^\delta_r, f_r \right \rangle \Bigg]dr\\
	&= \int_s^t \EE \Bigg[ \int_{\RR^d} w_r(x) \Bigg( \frac{\del f_r(x)}{\del r} - \tr[ \widetilde  a_r(x) D^2 f_r(x) ] - (b_r^\delta(x) + 2 \div \widetilde  a_r(x)) \cdot Df_r(x) \\
	& \qquad \qquad - ( \div b^\delta_r(x) + \div \div \widetilde  a_r(x) )f_r(x) \Bigg) dx - \left \langle \zeta^\delta_r, f_r \right \rangle \Bigg]dr.
\end{align*}
Then, if we choose $\lambda$, $\mcl K$, and $C > 0$ sufficiently large, depending only on $\norm{\widetilde  a}_{C^{1,1}}$, the local bounds for $D_p \widetilde  H^1$ and $D^2_{pp} \widetilde  H^1$, and the Lipschitz and semiconcavity bounds for $\widetilde  u^1$ and $\widetilde  u^2$, and if we choose $f = \psi_{\lambda,\mcl K}$ from Lemma \ref{L:dualproblem} and set
\[
	\gamma_t := \EE \int_{\RR^d} w_t(x)\psi_{\lambda,\mcl K}(t,x)dx,
\]
we conclude that, for $t \in [0,T]$,
\[
	\dot \gamma_t \ge - C \gamma_t - \left \langle \zeta^\delta_t ,\psi_{\lambda,\mcl K, t} \right \rangle.
\]
By the dominated convergence theorem, using the fact that $\psi_{\lambda,\mcl K}(t,\cdot) \in W^{1,1}$, upon sending $\delta \to 0$, we have $\dot \gamma_t \ge -C \gamma_t$. Gr\"onwall's inequality gives $\sup_{t \in [0,T]} \gamma_t \le e^{C(T-t)} \gamma_T$, or, in other words,
\[
	\sup_{t \in [0,T]} \EE \int_{\RR^d} \phi(\widetilde  u^1_t(x) - \widetilde  u^2_t(x)) \psi_{\lambda,\mcl K}(t,x)dx
	\le e^{C(T-t)} \EE \int_{\RR^d} \phi(\widetilde  G^1(x) - \widetilde  G^2(x)) \psi_{\lambda,\mcl K}(t,x)dx.
\]
We then take $\phi$ to be such that $\phi(r) > 0$ if and only if $r < 0$. Using the fact that $\widetilde  G^1 \ge \widetilde  G^2$ and $\psi_{\lambda,K} > 0$ in $[0,T] \times \RR^d$, we conclude that, with probability one, $\widetilde  u^1 \ge \widetilde  u^2$ in $[0,T] \times \RR^d$, as desired.

\end{proof}

\subsection{Stability} We now discuss stability for \eqref{E:hjbbody}. This result is used later to prove the existence of weak solutions to the mean field games system. It also immediately gives a vanishing viscosity type result relating solutions of \eqref{E:hjbbody} and solutions of the first order problem considered in \cite{CSmfg}.

\begin{theorem}\label{T:comparison2}
	Assume $(\widetilde  a^1, \widetilde  G^1, \widetilde  H^1)$ and $(\widetilde  a^2, \widetilde  G^2, \widetilde  H^2)$ satisfy \eqref{A:filtration} and \eqref{A:stochHJB}, and let $(\widetilde  u^1,M^1)$ and $(\widetilde  u^2,M^2)$ be the corresponding solutions of \eqref{E:hjbbody} in the sense of Definition \ref{D:solution}. Then, for all $R > 0$, there exists $C > 0$ depending on $R$, $d$, the Lipschitz and semiconcavity constants for $\widetilde  u^1$ and $\widetilde  u^2$, and constants in the given assumptions, such that
	\begin{align*}
		\sup_{t \in [0,T]} &\EE \norm{\widetilde  u^1_t - \widetilde  u^2_t}_{\oo, B_R}^{d+2}\\
	&\le C\EE \left( \sup_{t \in [0,T]}\norm{\widetilde  a^1_t - \widetilde  a^2_t}_{C^{0,1}} + \norm{\widetilde  H^1 - \widetilde  H^2}_{\oo, [0,T] \times \RR^d \times B_C} + \norm{\widetilde  G^1 - \widetilde  G^2}_\oo \right).
	\end{align*}
\end{theorem}

\begin{proof}
We argue as in the proof of Theorem \ref{T:comparison}, and let $v$, $M$, $b$, $b^\delta$, $\phi$, $w$, and $\zeta^\delta$ be defined in the same way. Defining also, for $(t,x) \in [0,T] \times \RR^d$,
\[
	\kappa_t(x) := \widetilde  H^2_t(x,D\widetilde  u^2_t(x)) - \widetilde  H^1_t(x,D\widetilde  u^2_t(x)),
\]
 we deduce, by Lemma \ref{L:composition}, for any $f \in W^{1,1}([0,T] \times \RR^d)$ and $0 \le s < t \le T$,
\begin{align*}
	&\EE\int_{\RR^d} w_t(x) f_t(x)dx - \EE\int_{\RR^d} w_s(x) f_s(x)dx\\
	&\ge \int_s^t \EE \bigg[ -\left\langle \tr[(\widetilde  a^1_r - \widetilde  a^2_r) D^2 u^1_r]\phi'(v_r) + \tr[\widetilde  a^2_r D^2 w_r ] + \zeta^\delta_r,f_r \right \rangle\\
	&\quad + \int_{\RR^d} \bigg( \bigg[- \kappa_r(x) +  \div (b^\delta_r(x) w_r(x) ) \bigg] f_r(x) \\
	&\qquad + w_r(x) \bigg[\frac{\del f_r(x)}{\del r} - \div b^\delta_r(x) f_r(x) \bigg] \bigg) dx \bigg] dr\\
	&= \int_s^t \EE \bigg[ \int_{\RR^d} w_r(x) \Bigg( \frac{\del f_r(x)}{\del r} - \tr[ D^2(\widetilde  a^2_r(x) f_r(x) )]\\
	&\quad - b_r^\delta(x) \cdot Df_r(x) - \div b^\delta_r(x) f_r(x) \Bigg) dx - \left \langle \tr[(\widetilde  a^1_r - \widetilde  a^2_r) D^2 \widetilde  u^1_r]\phi'(v_r) + \zeta^\delta_r + \kappa_r, f_r \right \rangle \bigg]dr\\
	&= \int_s^t \EE \Bigg[ \int_{\RR^d} w_r(x) \Bigg( \frac{\del f_r(x)}{\del r} - \tr[ \widetilde  a^2_r(x) D^2 f_r(x) ] - (b_r^\delta(x) + 2 \div\widetilde  a^2_r(x)) \cdot Df_r(x) \\
	& \qquad \qquad - ( \div b^\delta_r(x) + \div \div \widetilde  a^2_r(x) )f_r(x) \Bigg) dx\\
	&\qquad \qquad  - \left \langle  \tr[(\widetilde  a^1_r - \widetilde  a^2_r) D^2 u^1_r]\phi'(v_r) +\zeta^\delta_r + \kappa_r, f_r \right \rangle \Bigg]dr.
\end{align*}
We now take $f = \psi_{\lambda,\mcl K}$ as in Lemma \ref{L:composition} for $\lambda,\mcl K$ sufficiently large, depending only on the data specified in the statement of the theorem, and set
\[
	\gamma_t := \EE \int_{\RR^d} w_t(x)\psi_{\lambda,\mcl K}(t,x)dx.
\]
Then, upon sending $\delta \to 0$ and using the dominated convergence theorem, we have, for $t \in [0,T]$,
\[
	\dot \gamma_t \ge - C \gamma_t - \EE \left \langle \tr[(\widetilde  a^1_t -  \widetilde  a^2_t) D^2 \widetilde  u^1_t]\phi'(v_t) + \kappa_t,\psi_{\lambda,\mcl K, t} \right \rangle.
\]
We further estimate, for some $C > 0$ as in the statement of the theorem,
\[
	\left| \left\langle \kappa_t, \psi_{\lambda,\mcl K, t} \right\rangle \right| \le C \sup_{(x,p) \in \RR^d \times B_C} |\widetilde  H^1_t(x,p) - \widetilde  H^2_t(x,p)|,
\]
and
\begin{align*}
	\Big \langle &\tr[(\widetilde  a^1_t - \widetilde  a^2_t) D^2 \widetilde  u^1_t]\phi'(v_t),\psi_{\lambda,\mcl K, t} \Big \rangle\\
	&= - \int_{\RR^d} D\widetilde  u^1_t(x) \cdot \Big[ \div(\widetilde  a^1_t(x) - \widetilde  a^2_t(x)) \phi'(v_t(x)) \psi_{\lambda,\mcl K}(t,x)\\
	& \quad + (\widetilde  a^1_t(x) - \widetilde  a^2_t(x))Dv_t(x) \phi''(v_t(x))\psi_{\lambda,\mcl K}(t,x)\\
	&\quad + (\widetilde  a^1_t(x) - \widetilde  a^2_t(x)) D\psi_{\lambda,\mcl K}(t,x) \phi'(v_t(x)) \Big]dx.
\end{align*}
We then take $\phi(r) = (r_-)^2$. An application of Young's inequality yields, for a different $C > 0$,
\[
	\dot \gamma_t \ge -C \gamma_t - C\EE\left ( \norm{\widetilde  a^1_t - \widetilde  a^2_t}_{C^{0,1}} + \norm{\widetilde  H^1 - \widetilde  H^2}_{\oo, [0,T] \times \RR^d \times B_C} \right).
\]
Upon applying Gr\"onwall's inequality and repeating the argument after the switching the roles of $\widetilde  u^1$ and $\widetilde  u^2$, we obtain
\begin{align*}
	\sup_{t \in [0,T]} & \EE \int_{\RR^d} |\widetilde  u^1_t(x) - \widetilde  u^2_t(x)|^2 \psi_{\lambda,\mcl K}(t,x) dx\\
	& \le C\EE\left(\norm{\widetilde  G_1 - \widetilde  G_2}_\oo^2 + \sup_{t \in [0,T]}\norm{\widetilde  a^1_t - \widetilde  a^2_t}_{C^{0,1}} + \norm{\widetilde  H^1 - \widetilde  H^2}_{\oo, [0,T] \times \RR^d \times B_C} \right).
\end{align*}
The desired estimate now follows from Lemma \ref{L:interpolation} with $p = 2$, and the boundedness of $\widetilde  G^1$ and $\widetilde  G^2$.
\end{proof}

\subsection{Optimal control representation} We finish this section by stating an optimal control representation of the solution. The result, and proof, are similar to the first-order case in \cite[Proposition 2.7]{CSmfg}. 

We assume without loss of generality that the probability space $(\Omega, \mbb P)$ supports an $m$-dimensional Brownian motion $(B_t)_{t \in [0,T]}$ independent of $\mbb F$. We then define the set of admissible controls
\begin{align*}
	\mcl C &:= \{ \alpha \in L^\oo( [0,T] \times \Omega; \RR^d) : \alpha \text{ is progressively measurable} \\
	& \qquad \text{with respect to the filtration generated by $B$}\}.
\end{align*}
Given $\alpha \in \mcl C$, we denote by $X^{\alpha,t,x} :[t,T] \times \Omega^* \to \RR^d$ the unique solution of the stochastic differential equation
\begin{equation}\label{controlsde}
	\begin{cases}
	dX^{\alpha,t,x}_s = \alpha_s ds + \widetilde  \sigma_s(X^{\alpha,t,x}_s) dB_s, & t \le s \le T,\\
	X^{\alpha,t,x}_t = x.
	\end{cases}
\end{equation}
Observe that $X^{\alpha,t,x}$ is adapted to the filtration generated by $\mbb F$ and $B$.

\begin{proposition}
	The unique solution $\widetilde  u$ of \eqref{E:hjbbody} has the representation
	\[
		\widetilde  u_t(x) = \essinf_{\alpha \in \mcl C} \EE \left[ \int_t^T \widetilde  H^*_s(X^{t,x,\alpha}_s, \alpha_s)ds + \widetilde  G(X^{t,x,\alpha}_T) \, \Big| \, \mbb F_t \right].
	\]
\end{proposition}

\begin{proof}
	Fix $N \in \NN$ and let $\widetilde u^N$ be as in \eqref{E:discreteTVP}. Then, for any $n = 1,2,\ldots, 2^N$ and $t \in [t_{n-1}^N, t_n^N)$, the standard stochastic control formula (see \cite{lionsHJB2}) gives
	\[
		\widetilde u^N_t(x) = \essinf_{\alpha \in \mcl C} \EE \left[ \int_t^{t^N_n} \widetilde H^{N,*}_s(X^{t,x,\alpha,N}_s, \alpha_s)ds + \widetilde u^N_{t^N_n,-}(X^{t,x,\alpha,N}_{t^N_n}) \right],
	\]
	where $X^{t,x,\alpha,N}$ solves
	\[
		\begin{cases}
			dX^{\alpha,t,x,N}_s = \alpha_s ds + \widetilde  \sigma^N_s(X^{\alpha,t,x,N}_s) dB_s, & t \le s \le T,\\
			X^{\alpha,t,x,N}_t = x.
		\end{cases}
	\]
	The $\mbb F_{t^N_n}$-measurability of $H^N$ and $\sigma^N$ (and therefore $X^{\alpha,t,x,N}$) on $[t, t_n^N)$ yields, in view of the definition of $u^N_{t^N_n,-}$,
	\[
		\widetilde u^N_t(x) = \essinf_{\alpha \in \mcl C} \EE \left[ \int_t^{t_n^N} \widetilde H^{N,*}_s(X^{t,x,\alpha,N}_s, \alpha_s)ds + \widetilde u^N_{t^N_n}(X^{t,x,\alpha,N}_{t^N_n}) \, \Big| \, \mbb F^N_t \right],
	\]
	and an inductive argument leads to 
	\begin{equation}\label{Ncontrol}
		\widetilde u^N_t(x) = \essinf_{\alpha \in \mcl C} \EE \left[ \int_t^T \widetilde H^{N,*}_s(X^{t,x,\alpha,N}_s, \alpha_s)ds + \widetilde G(X^{t,x,\alpha,N}_T) \, \Big| \, \mbb F^N_t \right].
	\end{equation}
	There exists $C > 0$ depending only on $T$ and the Lipschitz constant of $\widetilde\sigma$ in the $\RR^d$-variable such that, uniformly in $(t,x) \in [0,T] \times \RR^d$, $\alpha \in \mcl C$, $N \in \NN$, and $s \in [t,T]$,
	\[
		\EE |X^{t,x,\alpha,N}_s - X^{t,x,\alpha}_s|^2 \le C \EE \int_0^T \norm{\widetilde\sigma^N_s - \widetilde \sigma_s}_\oo^2 ds \xrightarrow{N \to \oo} 0.
	\]
	By Lemma \ref{L:discretesolution}, $\norm{D \widetilde u^N}_{\oo, \Omega \times [0,T] \times \RR^d}$ is bounded uniformly in $N$. It follows that there exists a deterministic $M > 0$ independent of $N$ such that the essential infimum in \eqref{Ncontrol} can be restricted to $\alpha \in \mcl C$ satisfying $\norm{\alpha}_\oo \le M$. Using the uniform continuity of $\widetilde \sigma$ on $[0,T] \times \RR^d$ and $\widetilde H^*$ on $[0,T] \times \RR^d \times B_M$, we may take the limit as $N \to \oo$ in \eqref{Ncontrol}, which yields the result.
\end{proof}

\section{The MFG system}\label{sec:mfg}

We now consider the MFG system \eqref{e.MFGstoch}. As discussed in the introduction, this is done by studying the transformed system
\begin{equation}\label{e.MFGstochBisbody}
\left\{
	\begin{split} 
	&d\widetilde  u_{t} = \bigl[ -\tr(\widetilde  a_t(x, \widetilde  m_t) D^2 \widetilde  u_t)+ \widetilde  H_t(x,D\widetilde  u_{t},\widetilde  m_{t})  \bigr] dt + dM_{t}
\ \  \text{in} \ \ (0,T)\times \R^d, \\
	&d \widetilde  m_{t} = \div \bigl[  \div(\widetilde  a_{t}(x,\widetilde  m_t) \widetilde  m_{t})+ \widetilde  m_{t} D_{p} \widetilde  H_t(x,D\widetilde  u_{t} , \widetilde  m_{t})
 \bigr] dt \ \ \text{in} \ \   (0,T)\times \R^d, \\
	&\widetilde  u_T(x)= \widetilde  G(x,\widetilde  m_T), \ \ \   \widetilde  m_{0}=\bar m_0 \ \  \text{in} \ \  \R^d.
	\end{split}
	\right.
\end{equation}
We begin by proving a general existence result, Theorem \ref{T.ExistenceMFG}, under the assumption that $\widetilde  a$, $\widetilde  G$, and $\widetilde  H$ satisfy the same conditions as in Section \ref{sec.HJ}, uniformly in the $m$-variable and Lipschitz continuous with respect to the Wasserstein distance. Uniqueness of probabilistically strong solutions is then established under an additional monotonicity condition in Theorem \ref{T:MFGuniqueness}.

We remind the reader that \eqref{e.MFGstochBisbody} is formally derived from the original system \eqref{e.MFGstoch} through the relations given, for $(t,x,p,m) \in [0,T] \times \RR^d \times \RR^d \times \mcl P(\RR^d)$, by
\begin{equation}\label{changevariablebody}
	\widetilde  u_t(x)= u_t(x+\sqrt{2\beta}W_{t}), \qquad \widetilde  m_t= (\Id- \sqrt{2\beta}W_{t})_\sharp  m_t,
\end{equation}
\begin{equation}\label{pierre0body}
	\widetilde  a_t(x,m)=  a_t(x+\sqrt{2\beta} W_t,(\Id+\sqrt{2\beta}W_t)_\sharp m), 
\end{equation}
\begin{equation}\label{takis0body} 
	\widetilde  H_t(x,p,m)=  H_t( x+\sqrt{2\beta} W_t,p,(\Id+\sqrt{2\beta}W_t)_\sharp m) ,  
 \end{equation}
\begin{equation}\label{pierre02body} 
	\widetilde  G(x,m)= G(x+\sqrt{2\beta} W_T, (\Id+\sqrt{2\beta}W_T)_\sharp m),
\end{equation}
and
\begin{equation}\label{martingalebody}
	\widetilde  v_t(x) = v_t(x-\sqrt{2\beta} W_t) - \sqrt{2\beta} Du_t(x - \sqrt{2\beta} W_t), \quad \text{and} \quad
	M_t = \int_0^t \widetilde  v_s dW_s.
\end{equation}

Motivated by the theory of weak and strong solutions to stochastic ordinary differential equations, we define a notion of probabilistically weak solutions of \eqref{e.MFGstochBisbody}. More precisely, a weak solution is characterized not only by the unknowns $\widetilde  u$, $M$, and $\widetilde  m$, but also the filtered probability space and Wiener process $W$.

\begin{definition}\label{def.weaksolMGF} We say that $(\Omega, \FF, \PP, W, \widetilde  u,M, \widetilde  m)$ is a weak solution of \eqref{e.MFGstochBisbody} if 
\begin{enumerate}[(a)]
\item\label{def.Pspace} $(\Omega, \FF, \PP)$ is a filtered probability space, with a complete right-continuous filtration $\FF$, 
\item\label{def.mw} The random variable $\widetilde  m$ takes values in $C([0,T],\mathcal P_2(\R^d))$ with $\mbb P$-probability one, $\widetilde m$ is adapted to the filtration $\FF$, and $W$ is a $d$-dimensional $\FF$-Brownian motion,
\item\label{def.hjb} $(\widetilde  u,M)$ satisfies, in the sense of Definition \ref{D:solution}, the Hamilton-Jacobi equation 
\begin{align*}
&d \widetilde  u_{t} = \bigl[ -\tr(\widetilde  a_t(x, \widetilde  m_t) D^2 \widetilde  u_t)+ \widetilde  H_t(x,D \widetilde  u_{t},\widetilde  m_{t})  \bigr] dt
+ dM_{t}
\ \  {\rm in } \ \ (0,T) \times \R^d,
\end{align*}
\item\label{def.FP} $\PP$-a.s. and in the sense of distributions, $\widetilde  m$ solves the Fokker-Planck equation 
\begin{align*}
& d \widetilde  m_{t} =  \div\bigl[ \div(\widetilde  a_{t}(x,\widetilde  m_t) \widetilde  m_{t})+   \widetilde  m_{t} D_{p} \widetilde  H_t(x,D\widetilde u_{t},\widetilde  m_t) 
 \bigr] dt 
\ \  {\rm in } \ \   (0,T) \times \R^d,
\end{align*}
and
\item\label{def.boundary} $\widetilde  m_0=\bar m_0$ and $\widetilde  u_T= \widetilde  G(\cdot,\widetilde  m_T)$ $\PP$-a.s. 
\end{enumerate}
We say that the solution is strong if $(\widetilde  u,M,\widetilde  m)$ is adapted to the complete filtration generated by $W$. 
\end{definition}

\begin{remark}\label{R:originalsolution?}
In this paper, we give meaning to the original system \eqref{e.MFGstoch} by solving \eqref{e.MFGstochBisbody}. We do not explore whether the weak solution of the transformed system leads to some notion of solution $(u,v,m)$, and indeed, it is not clear whether a process $\widetilde v$ as in \eqref{martingalebody} can be extracted from a probabilistically weak solution $(\widetilde u, M, \widetilde m)$. On the other hand, in the theory of mean field games, one is typically interested in the value function $u$ and the density of agents $m$, which can be recovered as follows:
\[
	u_t(x) = \widetilde u_t(x + \sqrt{2\beta} W_t), \quad m_t = (\Id - \sqrt{2\beta} W_t)_\# \widetilde m_t.
\]
\end{remark}

\begin{remark}\label{R:martginaledist}
In view of \eqref{martingalebody} and Definition \ref{def.weaksolMGF}, the term $\div_x v$ appearing in \eqref{e.MFGstoch} would be a distribution which is the derivative of a measure in $\mcl M_\loc$. The study of the transformed system \eqref{e.MFGstochBisbody} then has another advantage: not only is the forward Fokker-Planck equation nonstochastic, but also, all terms in the Hamilton-Jacobi-Bellman are, at worst, distributions belonging to $\mcl M_\loc \cap W^{-1,\oo}$.
\end{remark}

\begin{remark}
	It does not appear to be the case that the transformations \eqref{changevariablebody} - \eqref{martingalebody} have a regularizing or averaging effect, a phenomenon which has been exploited for ODEs with rough or stochastic forcing \cite{D_07,FGP_10,CG_16}. This is because $u$ and $m$ are themselves stochastic processes adapted to the same filtration as $W$. 
\end{remark}

\subsection{Existence of a weak solution} We are able to prove the existence of a weak solution, that is, the tuple $(\Omega, \mbb F, \mbb P, W, \widetilde  u, M, \widetilde  m)$ satisfying Definition \ref{def.weaksolMGF}, under
general assumptions on the deterministic functions $a$, $H$, and $G$, recalling that $\widetilde  a$, $\widetilde  H$, and $\widetilde  G$ are then defined with the constructed Wiener process $W$ through \eqref{pierre0body}, \eqref{takis0body}, and \eqref{pierre02body}. 

In the case where $a \equiv 0$, treated in \cite{CSmfg}, or in the nondegenerate setting of \cite{CDLL}, existence and uniqueness of a strong solution is established by assuming monotonicity and a separated structure from the beginning. Our existence result is more general than these even when $a$ is either zero or nondegenerate, since we generalize to nonseparated dependence, and assume nothing more about the dependence on the measure variable $m$ except mild regularity.

More precisely, we will assume that
\begin{equation}\label{A:stochMFG}
	\left\{
	\begin{split}
		&a_\cdot(\cdot,m): [0,T] \times \RR^d \to \mbb S^d, \; G(\cdot,m) : \RR^d \to \RR, \\
		&\qquad \text{and } H_\cdot(\cdot,\cdot,m): [0,T] \times \RR^d \times \RR^d \to \RR\\
		&\text{satisfy \eqref{A:MAINdiffusion}, \eqref{A:MAINterminal}, \eqref{A:MAINHamiltonian}, and \eqref{A:MAINHcondition} uniformly over $m \in \mcl P_2(\RR^d)$,}\\
		&\text{and, for all $R > 0$, uniformly over $(t,x,p) \in [0,T] \times \RR^d \times B_R$},\\
		&m \mapsto (\sigma_t(x,m), G(x,m), H_t(x,p,m)) \text{ is Lipschitz in the $\mcl P_2(\RR^d)$ distance.}
	\end{split}
	\right.
\end{equation}
%
%
In order to exploit compactness in $\mcl P_2(\RR^d)$, we also assume that the initial distribution satisfies
\begin{equation}\label{A:barm0MFG}
\bar m_0 \in \mathcal P_p(\R^d) \text{ for some $p>2$ and has a bounded density.}
\end{equation}

\begin{theorem}\label{T.ExistenceMFG} 
Assume \eqref{A:stochMFG} and \eqref{A:barm0MFG}. Then there exists a weak solution of \eqref{e.MFGstochBisbody}. 
\end{theorem}

\subsubsection{The discretized system} As a first step, we construct a (strong) solution of a discretized version of \eqref{e.MFGstochBisbody}. More precisely, rather than discretizing the time and state spaces, we directly discretize the canonical probability space of paths. This leads to a fixed point problem on a metric space which, due to the finiteness of the probability space, is now compact, for which Kakutani's fixed point theorem is well-suited.

Let $\Omega^0=C([0,T], \R^d)$, $\FF^0 = (\FF^0_t)_{t \in [0,T]}$ be the filtration generated by the canonical process $W \in \Omega^0$, and $\PP^0$ be the Wiener measure on $\Omega^0$, so that $W$ is a Brownian motion under $\PP^0$. 

The discretization of $\Omega^0$ is done through a certain projection map onto piecewise affine paths, whose slopes lie in a finite but possibly large set. The properties of this construction are outlined in the next lemma, whose proof we omit. A similar procedure is discussed in detail in the work of Carmona, Delarue, and Lacker \cite[2.4.2]{carmona2016mean}.

\begin{lemma}\label{L:discproj}
For $n \in \NN$, there exists a map $\pi^n:\Omega^0\to \Omega^0$ such that the following hold:
\begin{enumerate}[(i)]
\item $\pi^n$ is nonanticipative and Borel measurable,
\item for any $x\in \Omega^0$ and $i = 0,1,2,\ldots, 2^n - 1$, $\pi^n(x)$ is affine on
\[
	[t^n_i, t^n_{i+1}] := \left[\frac{Ti}{2^n}, \frac{T(i+1)}{2^{n}} \right],
\] 
\item the increments $(\pi^n(x)(t^n_{i+1})-\pi^n(x)(t^n_{i}))_{i = 0,1,2,\ldots, 2^n - 1}$ belong to a finite set depending only on $n$, independent of $x$, and,
\item if $\FF^n = (\FF^n_t)_{t \in [0,T]}$ denotes the piecewise-constant, right-continuous filtration generated by $\pi^n(W)$, then 
\begin{equation}\label{piecewiseWn}
	\begin{split}
	&\lim_{n \to \oo} \E_{\PP^0}\left[\sup_{t \in [0,T]} |\pi^n(W)_t-W_t|\right]=0 \quad \text{and}\\
	&\bigcup_{n \in \NN} \FF^n_t = \FF^0_t \quad \text{for all } t \in [0,T].
	\end{split}
\end{equation}
\end{enumerate}
\end{lemma}
%

We now construct the discretized variant of \eqref{e.MFGstochBisbody} by defining $\widetilde  a^n$, $\widetilde  H^n$, and $\widetilde  G^n$ from $a$, $H$ and $G$ just as in \eqref{pierre0body}, \eqref{takis0body} and \eqref{pierre02body}, replacing $W$ in each with the discrete projection $\pi^n(W)$. We then seek a solution $(\widetilde  u^n, M^n, \widetilde  m^n)$ of 
\begin{equation}\label{eq.discMFG}
\left\{
	\begin{split}
	&d \widetilde  u^n_{t} = \bigl[ -\tr(\widetilde  a^n_t(x,m^n_t) D^2 \widetilde  u^n_t)+ \widetilde  H^n_t(x,D\widetilde  u^n_{t},\widetilde  m^n_t)  \bigr] dt + dM^n_{t}
\ \ \text{in} \ \ (0,T) \times \R^d,\\
	&d \widetilde  m^n_{t} = \div\bigl[  \div(\widetilde  a^n_{t}(x,\widetilde  m^n_t) \widetilde  m^n_{t})+ \widetilde  m^n_{t} D_{p} \widetilde  H^n_t(x,D\widetilde  u^n_{t},\widetilde  m^n_t) 
\bigr] dt  \ \  \text{in} \ \   (0,T) \times \R^d, \\
	&\widetilde   m^n_0=\bar m_0,\ \ \widetilde  u^n_T(x)= \widetilde  G^n(x,\widetilde  m^n_T) \ \  \text{in} \ \   \R^d.
	\end{split}
	\right.
\end{equation}

\begin{proposition}\label{P:discreteMFG}
There exists a solution $(\widetilde  u^n,M^n,\widetilde  m^n)$ of \eqref{eq.discMFG} adapted to the discrete filtration $\FF^n$.
\end{proposition}

Proposition \ref{P:discreteMFG} is proved by setting up an appropriate fixed point problem. We define, for $R > 0$ (recall that $p > 2$ is as in \eqref{A:barm0MFG}),
\begin{align*}
	\mcl M^n_R := 
	\Bigg\{ &m \in L^\oo\big(\Omega^0; C([0,T], \mcl P_2(\RR^d)) \cap L^\oo([0,T] \times \RR^d) \big), \; m \text{ is $\FF^n$-adapted, } m_0 = \bar m_0,\\
	&\text{and } \norm{m}_{\oo,\Omega^0 \times [0,T] \times \RR^d}  +\esssup_{\omega \in \Omega^0} \sup_{t \in [0,T]} \int_{\RR^d} |x|^p m_t(x)dx\\
	&\qquad +\esssup_{\omega \in \Omega^0} \sup_{s,t \in [0,T], \, s \ne t} \frac{{\bf d}_2(m_s,m_t)}{|s-t|^{1/2}} \le R \Bigg\},
\end{align*}
along with the metric
\[
	{\bf d}_{\mcl M^n_R}(m,m') =  \E\left[\sup_{t\in [0,T]} {\bf d}_2(m_t,m'_t)\right].
\]
The space $\mcl M^n_R$ contains the constant process $\bar m_0$, and is thus nonempty, as long as
\[
	R \ge \norm{\bar m_0}_\oo + \int_{\RR^d} |x|^p \bar m_0(x) dx.
\]
Moreover, $\left( \mcl M^n_R, {\bf d}_{\mcl M^n}  \right)$ is a convex and compact metric space. Indeed, because the processes are $\FF^n$-adapted, they may be identified with random variables defined over a finite probability space, and they take values in a compact subset of $C([0,T], \mathcal P_2(\RR^d))$.


For a fixed $\widetilde  m \in \mcl M^n_R$, we consider the backward stochastic Hamilton-Jacobi equation  
\begin{equation}\label{E:discreteHJB}
\left\{
	\begin{split}
	&d \widetilde  u^{(\widetilde  m)}_{t} = \bigl[  -\tr(\widetilde  a^n_t(x,\widetilde  m_t) D^2 \widetilde  u^{(\widetilde  m)}_t)+ \widetilde  H^n_t(x, D\widetilde  u^{(\widetilde  m)}_{t},\widetilde  m_t)  \bigr] dt\\
	&\hspace{2.5cm} + dM^{(\widetilde  m)}_{t}
\ \  \text{in} \ \ (0,T) \times \R^d,\\
	&\widetilde  u^{(\widetilde  m)}_T(x)= \widetilde  G^n(x,\widetilde  m_T) \ \  \text{in} \ \   \R^d.
	\end{split}
\right.
\end{equation}


\begin{lemma}\label{L:umMm} For any $\widetilde  m \in \mcl M^n_R$, there exists a unique solution $(\widetilde  u^{(\widetilde  m)}, M^{(\widetilde  m)})$ of \eqref{E:discreteHJB}, and there exists a constant $C > 0$ that is independent of $R$ and $\widetilde  m \in \mcl M^n_R$ such that, with probability one,
	\[
		\nor{\widetilde  u^{(\widetilde  m)}}{\oo} + \nor{D\widetilde  u^{(\widetilde  m)}}{\oo} + \esssup m_+(D^2 \widetilde  u^{(\widetilde  m)} ) \le C.
	\]
	Moreover, the map
	\[
		\left( \mcl M^n_R, {\bf d}_{\mcl M^n} \right) \ni \widetilde  m \mapsto D\widetilde  u^{(\widetilde  m)} \in L^1(\Omega^0, L^1_\loc([0,T] \times\RR^d))
	\]
	is continuous.
\end{lemma}

\begin{proof}
	In view of \eqref{A:stochMFG}, the random variables
	\[
		(t,x) \mapsto \left( \widetilde  a^n_t(x, \widetilde  m_t), \widetilde  H^n_t(x,\cdot,\widetilde  m_t), \widetilde  G^n_t(x,\widetilde  m_t) \right)
	\]
	satisfy \eqref{A:stochHJB} for the filtered probability space $(\Omega^0, \FF^n, \PP^0)$, with constants that are independent of $R$ and $\widetilde  m$. The existence and uniqueness of a solution satisfying the desired bounds is then a consequence of Theorems \ref{T:existence} and \ref{T:comparison}, and the continuity statement follows from the uniform Lipschitz continuity of $(\widetilde  a^n, \widetilde  H^n, \widetilde  G^n)$ in the $\mcl P_2$-variable, Theorem \ref{T:comparison2}, Lemma \ref{L:SCprops}\eqref{L:Duconverge}, and the bounded convergence theorem.
\end{proof}

\begin{lemma}\label{L:mumu2} 
Fix $\widetilde  m\in \mcl M^n_R$, and let $(\widetilde  u^{(\widetilde  m)}, M^{(\widetilde  m)})$ be as in Lemma \ref{L:umMm}. Then there exists a constant $\widetilde  R > 0 $, independent of $R$ and $\widetilde  m$, and an $\mbb F^n$-adapted process $\mu \in C([0,T], \mathcal P_2(\R^d)) \cap L^\oo([0,T] \times \RR^d)$ such that, with probability one,
\[
	\norm{\mu}_{\oo, [0,T] \times \RR^d} + \sup_{t \in [0,T]} \int_{\RR^d} |x|^p \mu_t(x) dx + \sup_{s,t \in [0,T], \, s \ne t} \frac{ {\bf d}_2(\mu_s, \mu_t)}{|s-t|^{1/2} } \le \widetilde  R,
\]
and, in the sense of distributions,
\begin{equation}\label{eq.mumu2}
\left\{
\begin{split}
	&d \mu_{t} = \div \bigl[  \div(\widetilde  a^n_{t}(x,\widetilde  m_t) \mu_{t})+  \mu_{t} D_{p} \widetilde  H^n_t(x, D\widetilde  u^{(\widetilde  m)}_{t},\widetilde  m_t) 
 \bigr] dt  \ \  \text{in} \ \   (0,T) \times \R^d, \\
	&\mu_0=\bar m_0\ \  \text{in} \ \   \R^d.
\end{split}
\right.
\end{equation}
\end{lemma}

\begin{proof} For $\eps >0$, let $\rho_\eps$ be a standard mollifier as in \eqref{mollifier}, and set
\[
	\widetilde  u^{(\widetilde  m),\eps} = \rho_\eps * \widetilde  u^{(\widetilde  m)} \quad \text{and} \quad
	b^\eps_t(x) := -D_p \widetilde  H^n_t(x, D \widetilde  u^{(\widetilde  m),\eps}_t(x), \widetilde  m_t).
\]
We then consider the regularized equation 
\begin{equation}\label{eq.mumu2eps}
\left\{
	\begin{split}
	&\del_t \mu^\eps_{t} = \div \bigl[\eps D \mu_t^\eps+  \div(\widetilde  a^n_{t}(x,\widetilde  m_t) \mu^\eps_{t}) -  \mu^\eps_{t} b^\eps_t
 \bigr]  \ \  \text{in} \ \  (0,T) \times  \R^d, \\
	&\mu^\eps_0=\bar m_0\ \  \text{in} \ \   \R^d.
	\end{split}
	\right.
\end{equation}
With probability one, \eqref{eq.mumu2eps} is a uniformly elliptic equation with coefficients which are smooth and bounded in space and measurable in time. Therefore, there exists a unique solution in the sense of distributions, which is adapted to $\FF^n$, since $\widetilde  a^n$, $\widetilde  H^n$ and $\widetilde  u^{(\widetilde  m)}$ are. Standard arguments yield $\widetilde  R$ independent of $R$, $\widetilde  m$, and $\eps$ such that
\[
	\sup_{t \in [0,T]} \int_{\RR^d} |x|^p \mu^\eps_t(x) dx + \sup_{s,t \in [0,T], \, s \ne t} \frac{ {\bf d}_2(\mu^\eps_s, \mu^\eps_t)}{|s-t|^{1/2} } \le \widetilde  R.
\]
By Lemma \ref{L:umMm}, $\widetilde  u^{\widetilde  m,\eps}$ is uniformly Lipschitz continuous and semiconcave in space. This, along with the uniform convexity of $\widetilde  H^n$ in the gradient variable, implies that there exists $C > 0$ independent of $R$, $\widetilde  m$, and $\eps$ such that $\div_x b^\eps \ge -C$. The maximum principle then implies that $\mu^\eps$ is bounded in $L^\oo(\Omega^0 \times [0,T] \times \RR^d)$ independent of $\eps$, $R$, and $\widetilde  m$, and so, enlarging $\widetilde  R$ if necessary,
\[
	\norm{\mu^\eps}_{\oo,[0,T] \times \RR^d} + \sup_{t \in [0,T]} \int_{\RR^d} |x|^p \mu^\eps_t(x) dx + \sup_{s,t \in [0,T], \, s \ne t} \frac{ {\bf d}_2(\mu^\eps_s, \mu^\eps_t)}{|s-t|^{1/2} } \le \widetilde  R.
\]
Exploiting the finiteness of the probability space $(\Omega^0, \mbb F^n, \mbb P^0)$, we then extract a subsequence of $(\mu^\eps)$ that converges weakly-$*$ in $L^\infty(\Omega^0\times [0,T]\times \R^d)$ to some $\mu$, which, with probability one, is a solution of \eqref{eq.mumu2} with the desired properties.
\end{proof}

We now combine Lemmas \ref{L:umMm} and \ref{L:mumu2} to set up the fixed point argument.

\begin{proof}[Proof of Proposition \ref{P:discreteMFG}]
Set $R = \widetilde  R$, where $\widetilde  R$ is as in Lemma \ref{L:mumu2} (this is allowed because $\widetilde  R$ is independent of $R$). Given $\widetilde  m \in \mcl  M^n_R$, we let $\mathcal E(\widetilde  m)$ be the set of processes $\mu$ that solve \eqref{eq.mumu2} and satisfy the properties laid out in Lemma \ref{L:mumu2}. This defines a map taking values in convex subsets of $\mcl M^n_R$:
\[
	\mcl M^n_R \ni \widetilde  m \mapsto \mathcal E(\widetilde  m) \in 2^{\mcl M^n_R},
\]
and, in view of Lemma \ref{L:umMm} and the uniform Lipschitz continuity of the data in the $\mcl P_2$-variable, the graph of $\mcl E$ is closed. By the Kakutani fixed point theorem, there exists $\widetilde  m^n\in \mathcal M^n_R$ such that $\widetilde  m^n\in \mathcal E(\widetilde  m^n)$. Setting $\widetilde  u^n = \widetilde  u^{(\widetilde  m^n)}$ and $M^n = M^{(\widetilde  m^n)}$ gives the desired solution.
\end{proof}

\subsubsection{The passage to the limit} We begin by embedding the discrete problem above into a larger probability space. For $R = \widetilde  R$ as in the proof of Proposition \ref{P:discreteMFG}, we define
\begin{equation}\label{measurepathspace}
\begin{split}
	\Omega^* := &\Bigg\{ \omega = (m,w) \in C([0,T], \mcl P_2(\RR^d) \times \RR^d)\\
	& \cap L^\oo([0,T] \times \RR^d, \RR \times \RR^d): w_0 = 0, \, m_0 = \bar m_0, \quad \text{and}\\
	& \norm{m}_{\oo, [0,T] \times \RR^d} + \sup_{t \in [0,T]} \int_{\RR^d} |x|^p m_t(x)dx + \sup_{s,t \in [0,T], s \le t} \frac{{\bf d}_2(m_s,m_t)}{|s-t|^{1/2}} \le R \Bigg\}.
\end{split}
\end{equation}
The set $\Omega^*$ is a Polish space with the metric defined, for $\omega = (m,w)$ and $\omega' = (m',w')$, by
$$
	{\bf d}_{\Omega^*}(\omega, \omega')= \sup_{t\in [0,T]} \left( {\bf d}_2(m_t,m_t') + |w_t-w_t'| \right).
$$
We denote by $\FF = (\FF_t)_{t \in [0,T]}$ the filtration on $\Omega^*$ generated by the canonical process $\omega=(m,w)$.

Let $\mathcal P(\Omega^*)$ be the set of Borel probability measures on $\Omega^*$ with finite first order moment, endowed with the distance defined, for $P, P' \in \mcl P(\Omega^*)$,
$$
	{\bf d}_{\mcl P(\Omega^*)} (P, P') = \inf_{\pi} \E_\pi\left[ \bf{d}_{\bar\Omega}(\omega, \omega)\right]
$$
where the infimum is taken over the Borel probability measures $\pi$ on $\Omega^* \times \Omega^*$ with first marginal $P$ and second marginal $P'$. 

By Proposition \ref{P:discreteMFG}, for all $n \in \NN$, there exists a solution $(\widetilde  u^n, M^n, \widetilde  m^n)$ of the discrete MFG system \eqref{eq.discMFG}, which defines a probability measure $P^n$ on $\Omega^*$:
\begin{equation}\label{Pn}
	P^n := (\widetilde  m^n, W)_\sharp \mbb P^0 \in  \mathcal P(\bar\Omega).
\end{equation}
We can likewise express the various data as random variables on $\Omega^*$. Abusing notation slightly, we define, for $n \in \NN$, $(t,x,p) \in [0,T] \times \RR^d \times \RR^d$, and $\omega = (m,w) \in \Omega^*$,
\begin{equation}\label{tildean}
	\widetilde  a^n_t(x,\omega) := a_t(x+\sqrt{2\beta}\pi^n(w)_t, (\Id+\sqrt{2\beta}\pi^n(w)_t)_\sharp m_t),
\end{equation}
\begin{equation}\label{tildeHn}
	\widetilde  H^n_t(x,p,\omega) := \widetilde  H_t(x+\sqrt{2\beta} \pi^n(w)_t,p,(\Id+\sqrt{2\beta}\pi^n(w)_t)_\sharp m_t),
\end{equation}
\begin{equation}\label{tildeGn}
	\widetilde  G^n(x,\omega)= G(x+\sqrt{2\beta}\pi^n(w)_T, (\Id+\sqrt{2\beta}\pi^n(w)_T)_\sharp m_T),
\end{equation}
\begin{equation}\label{tildeun}
	\widetilde  u^n_t(x,\omega) = \widetilde  u^n_t(x, \pi^n(w)),
\end{equation}
and
\begin{equation}\label{Martingalen}
	M^n_t(\omega)= M^n_t(\pi^n(w)).
\end{equation}
The filtration $\FF^n$ on $\Omega^0$ can be seen as a filtration on $\Omega^*$ generated by $(m,w) \mapsto \pi^n(w)$. With this notation, $(\widetilde  u^n, M^n)$, defined on the filtered probability space $(\Omega^*, \FF^n, P^n)$, solves, in the sense of Definition \ref{D:solution},
\begin{equation}\label{E:tildeun}
	\left\{
	\begin{split}
	&d\widetilde  u^n_{t} = \bigl[ -\tr(\widetilde  a^n_t(x,\omega) D^2 \widetilde  u^n_t)+ \widetilde  H^n_t(x,D\widetilde  u^n_{t},\omega)  \bigr] dt + d M^n_{t}
	\quad  \text{in }  (0,T) \times \R^d,\\
 	&\widetilde  u^n_T= \widetilde  G^n(\cdot,\omega) \quad \text{in }  \R^d.
	\end{split}
	\right.
\end{equation}
The weak solution of \eqref{e.MFGstochBisbody} is obtained by taking a weak limit, as $n \to \oo$, of $P^n$.

\begin{lemma}\label{L:Pnlimit}
	Up to a subsequence, as $n \to \oo$, $P^n$ converges to some $P \in \mcl P(\Omega^*)$ with respect to the distance ${\bf d}_{\mcl P(\Omega^*)}$. Moreover, if $\FF^P$ denotes the natural right-continuous augmentation of the filtration $\FF$ by $P$-null sets, then the second component $w$ of the canonical process $(m,w) \in\Omega^*$ is a $\FF^P$-Brownian motion under $P$.
\end{lemma}

\begin{proof}
	The first marginal of $P^n$ is supported on a compact set of $C([0,T], \mcl P_2(\RR^d))$ that is independent of $n$, and the second marginal is the Wiener measure on $C([0,T], \RR^d)$. It follows that $(P^n)_{n \in \NN}$ is tight, and so the convergence along a subsequence is then a standard consequence of the fact that weak convergence implies convergence in the ${\mbf d}_{\mcl P(\Omega^*)}$ distance.
	
	For all $n$, the random variable $(m,w) \mapsto (w_t)_{t \in [0,T]}$ is a Brownian motion under $P^n$, and therefore the same is true for $P$. To finish the proof, it suffices to show that, for all $0 \le s < t \le T$, $w_t - w_s$ is independent of $\mbb F^P_s$.
	
	 Fix $N \in \NN$, Borel functions $\phi: \RR^d \to \RR$ and $\psi: (\RR^{1+d})^N \to \RR$, and $0 \le s_1 < s_2 < \cdots < s_N \le s < t$. For $n \in \NN$, we have
\begin{equation}\label{independence}
	\begin{split}
	&\int_{\Omega^*} \phi(w_t - w_s) \psi((m,w)_{s_1}, (m,w)_{s_2}, \ldots, (m,w)_{s_N}) dP^n(m,w)\\
	&= 
	\int_{\Omega^0} \phi(w_t - w_s) \psi\left( \widetilde  m^n_{s_1}(\cdot,w), w_{s_1}, \widetilde  m^n_{s_2}(\cdot,w), w_{s_2}, \ldots, \widetilde  m^n_{s_N}(\cdot,w), w_{s_N} \right) d\PP^0(w).
	\end{split}
\end{equation}
For $i = 1,2,\ldots,N$, $\widetilde  m^n_{s_i}$ is $\mbb F^n_{s_i}$-measurable, and therefore $\mbb F^0_{k2^{-n}}$-measurable, where $k$ is the largest integer smaller than $2^n s_i$. It follows that $\widetilde  m^n_{s_i}$ is independent of $w_t - w_s$, and so
\begin{align*}
	&\int_{\Omega^0} \phi(w_t - w_s) \psi\left( \widetilde  m^n_{s_1}(\cdot,w), w_{s_1}, \widetilde  m^n_{s_2}(\cdot,w), w_{s_2}, \ldots, \widetilde  m^n_{s_N}(\cdot,w), w_{s_N} \right) d\PP^0(w)\\
	&= 
	\int_{\Omega^0} \phi(w_t - w_s) d\PP^0(w)\\
	&\quad \cdot \int_{\Omega^0} \psi\left( \widetilde  m^n_{s_1}(\cdot,w'), w'_{s_1}, \widetilde  m^n_{s_2}(\cdot,w'), w'_{s_2}, \ldots, \widetilde  m^n_{s_N}(\cdot,w'), w'_{s_N} \right) d\PP^0(w')\\
	&= \int_{\Omega^*} \phi(w_t - w_s) dP^n(m,w) \int_{\Omega^*} \psi\left( (m', w')_{s_1}, (m',w')_{s_2},\ldots,(m',w')_{s_N} \right) dP^n(m',w').
\end{align*}
Taking the limit along an appropriate subsequence as $n \to \oo$, \eqref{independence} becomes
\begin{align*}
	\int_{\Omega^*} &\phi(w_t - w_{s}) \psi((m,w)_{s_1}, (m,w)_{s_2}, \ldots, (m,w)_{s_N}) dP(m,w)\\
	&= \int_{\Omega^*} \phi(w_t - w_{s}) dP(m,w) \int_{\Omega^*} \psi\left( (m', w')_{s_1}, (m',w')_{s_2},\ldots,(m',w')_{s_N} \right) dP(m',w').
\end{align*}
Because $N$, $(s_1,s_2,\ldots, s_N)$, $\phi$, and $\psi$ were arbitrary, this implies that, under $P$, $w_t - w_s$ is independent of $\FF_s$, and therefore the same is true for $\FF^P_s$.
\end{proof}

The rest of this subsection is devoted to proving that the tuple $(\Omega^*, \FF^P,P,w,\widetilde  u,M,m)$ is a weak solution of \eqref{e.MFGstochBisbody}, where $w$ and $m$ stand for the processes $(m,w) \mapsto w$ and $(m,w) \mapsto m$. Accordingly, we now set, for $(t,x,p) \in [0,T] \times \RR^d \times \RR^d$ and $\omega = (m,w) \in \Omega^*$,
\begin{equation}\label{tildealimit}
	\widetilde  a_t(x,\omega)=  a_t(x+\sqrt{2\beta} w_t,(\Id+\sqrt{2\beta}w_t)_\sharp m_t), 
\end{equation}
\begin{equation}\label{tildeHlimit} 
	\widetilde  H_t(x,p,\omega)=  H_t( x+\sqrt{2\beta} w_t,p,(\Id+\sqrt{2\beta}w_t)_\sharp m_t),  
\end{equation}
\begin{equation}\label{tildeGlimit}
	\widetilde  G(x,\omega)= G(x+\sqrt{2\beta} w_T, (\Id+\sqrt{2\beta}w_T)_\sharp m_T).
\end{equation}
Observe that, in view of \eqref{A:stochMFG}, for all $L > 0$, there exists $C = C_L > 0$ such that, for all $n \in \NN$, $\omega,\omega'\in \Omega^*$, and $t \in [0,T]$,
\begin{equation}\label{dataclose}
\begin{split}
	&\norm{\widetilde  a^n_t(\cdot,\omega) - \widetilde  a_t(\cdot,\omega')}_{C^{0,1}(\RR^d)}
	+ \norm{\widetilde  H^n_t(\cdot,\cdot,\omega) - \widetilde  H_t(\cdot,\cdot,\omega')}_{\oo,\RR^d \times B_L}\\
	&+  \norm{D_p \widetilde  H^n_t(\cdot,\cdot,\omega) - D_p \widetilde  H_t(\cdot,\cdot,\omega')}_{\oo,\RR^d \times B_L}
	+ \norm{\widetilde  G^n(\cdot,\omega) - \widetilde  G(\cdot,\omega') }_{\oo,\RR^d}\\
	&\le C\left(  {\bf d}_{\Omega^*}(\omega,\omega') + \norm{\pi^n(w) - w}_{\oo,[0,T]} \right).
\end{split}
\end{equation}
Now, for the probability space $(\Omega^*, \mbb F^P, P)$, we consider the solution $(\widetilde  u, M)$, in the sense of  Definition \ref{D:solution}, of the backward stochastic Hamilton-Jacobi-Bellman equation
\begin{equation}\label{E:u^P}
\left\{
	\begin{split}
	&d\widetilde  u_{t} = \bigl[ -\tr(\widetilde  a_t(x,\omega) D^2 \widetilde  u_t)+ \widetilde  H_t(x,D\widetilde  u_{t},\omega)  \bigr] dt + dM_{t} \quad \text{in }  (0,T) \times \R^d,\\
	&\widetilde  u_T= \widetilde  G(\cdot,\omega) \quad \text{in }  \R^d.
	\end{split}
	\right.
\end{equation}

\begin{lemma}\label{lem.stabilo} 
For $n \in \NN$, let $\gamma^n$ be an optimal coupling between  $P^n$  and $P$. Then there exists a subsequence $(n_k)_{k \in \NN}$ such that, for all $R > 0$,
\begin{equation}\label{doubleulimit}
	\lim_{k\to+\infty} \sup_{t\in [0,T]} \int_{\Omega^* \times \Omega^*} \left[\sup_{x\in B_R} |\hat u^{n_k}_t(x,\omega)-\widetilde  u_t(x,\omega')|^{d+2} \right] d\gamma^{n_k}(\omega,\omega') =0
\end{equation}
and
\begin{equation}\label{doubleDulimit}
	\lim_{k \to \oo} \int_{\Omega^* \times \Omega^*} \int_0^T \int_{B_R} |D\widetilde  u_t^{n_k}(x,\omega) - D \widetilde  u_t(x, \omega')|dx dt d\gamma^{n_k}(\omega,\omega') = 0.
\end{equation}
\end{lemma}

\begin{proof} 
For each $n \in \NN$, the filtered probability space $(\Omega^* \times \Omega^*, \mbb F^n \otimes \mbb F, \gamma^n)$ and the processes
\[
	\Omega^* \times \Omega^* \ni (\omega,\omega') \mapsto \left( \widetilde  a^n_\cdot(\cdot,\omega), \widetilde  H^n_\cdot(\cdot,\cdot,\omega), \widetilde  G^n(\cdot,\omega), \widetilde  a_\cdot(\cdot,\omega'), \widetilde  H_\cdot(\cdot,\cdot,\omega'), \widetilde  G(\cdot,\omega') \right)
\]
satisfy \eqref{A:filtration} and \eqref{A:stochHJB}, and
\[
	(\omega,\omega') \mapsto (\widetilde  u^n(\cdot,\omega),  M^n(\cdot,\omega) )
	\quad \text{and} \quad
	(\omega, \omega') \mapsto ( \widetilde  u(\cdot,\omega'), M(\cdot,\omega'))
\]
solve respectively \eqref{E:tildeun} and \eqref{E:u^P} in the sense of Definition \ref{D:solution}. Theorem \ref{T:comparison2} then gives, for any $R > 0$ and some constants $C,L > 0$ depending on $R$, $d$, and the constants the assumptions,
\begin{align*}
		&\sup_{t \in [0,T]} \int_{\Omega^* \times \Omega^*} \norm{\widetilde  u^n_t(\cdot,\omega) - \widetilde  u_t(\cdot,\omega')}_{\oo, B_R}^{d+2} d\gamma^n(\omega,\omega')\\
	&\le C\int_{\Omega^* \times \Omega^*} \Bigg( \sup_{t \in [0,T]}\norm{\widetilde  a^n_t(\cdot,\omega) - \widetilde  a_t(\cdot,\omega')}_{C^{0,1}} + \norm{\widetilde  H^n(\cdot,\cdot,\omega) - \widetilde  H(\cdot,\cdot,\omega')}_{\oo, [0,T] \times \RR^d \times B_L} \\
	& \qquad + \norm{\widetilde  G^n(\cdot,\omega) - \widetilde  G(\cdot,\omega')}_\oo \Bigg) d\gamma^n(\omega,\omega').
\end{align*}
The convergence statement \eqref{doubleulimit} now follows from \eqref{piecewiseWn}, \eqref{dataclose}, and Lemma \ref{L:Pnlimit}.

We now define ${\mbf \Omega} := (\Omega^*)^{\NN}$, and write $\bar \omega = (\omega^0,\omega^1,\omega^2,\ldots) \in {\mbf \Omega}$. We define the cylindrical filtration $\mbf F$ on $\mbf \Omega$ by
\[
	\mbf F_t := \sigma\left( (\omega^0_s)_{s \in [0,t]}, (\omega^1_s)_{s \in [0,t]}, \ldots, (\omega^n_s)_{s \in [0,t]}, \, n = 1,2,\ldots \right).
\]
By the Kolmogorov extension theorem, there exists a unique probability measure ${\mbf P}$ on ${\mbf \Omega}$ such that, for any $n \in \NN$ and Borel measurable $\Phi: (\Omega^*)^{1+n} \to \RR$,
\begin{align*}
	&\int_{{\mbf \Omega}} \Phi(\omega^0, \omega^1, \ldots, \omega^n) d{\mbf P}( \bar\omega)\\
	&:= \int_{(\Omega^*)^{1+n}} \Phi(\omega^0,\omega^1,\ldots, \omega^n) d\gamma^1(\omega^1,\omega^0) d\gamma^2(\omega^2,\omega^0) \cdots d\gamma^n(\omega^n,\omega^0).
\end{align*}
For $(t,x,\bar\omega) \in [0,T] \times \RR^d \times {\mbf \Omega}$ and $n \in \NN$, define
\[
	\widetilde  U^n_t(x,\bar \omega) := \widetilde  u^n_t(x,\omega^n) \quad \text{and} \quad \widetilde  U_t(x,\bar \omega) := \widetilde  u_t(x,\omega^0).
\]
Then \eqref{doubleulimit} can be rewritten as
\[
	\lim_{k \to \oo} \sup_{t \in[0,T]} \int_{{\mbf \Omega}} \sup_{x \in B_R} |\widetilde  U^{n_k}_t(x,\bar \omega) - \widetilde  U_t(x,\bar \omega) |^{d+2} d {\mbf P}(\bar \omega) = 0,
\]
which implies that, for some further subsequence, labeled again as $(n_k)_{k \in \NN}$, for all $t \in [0,T]$, $R > 0$, and ${\mbf P}$-a.s. in $\bar \omega$,
\[
	\lim_{k \to \oo} \sup_{x \in B_R} |\widetilde  U^{n_k}_t(x,\bar \omega) - \widetilde  U_t(x,\bar \omega)| = 0.
\]
Uniformly over $t \in [0,T]$, ${\mbf P}$-a.e. $\bar \omega$, and $n \in \NN$, $\widetilde  U^n_t(\cdot,\bar \omega)$ and $\widetilde  U_t(\cdot,\bar \omega)$ are globally Lipschitz continuous and semiconcave, and so Lemma \ref{L:SCprops}\eqref{L:Duconverge} and the dominated convergence theorem imply that, for all $R > 0$,
\[
	\lim_{k \to \oo} \int_{{\mbf \Omega}} \int_0^T \int_{B_R} |D\widetilde  U_t^{n_k}(x,\bar\omega) - D \widetilde  U_t(x, \bar\omega)|dx dt d{\mbf P}(\bar\omega) = 0,
\]
which is equivalent to \eqref{doubleDulimit}.

\end{proof}

%

\begin{proof}[Proof of Theorem \ref{T.ExistenceMFG}]
As stated above, we establish the claim of the theorem by proving that $(\Omega^*, \FF^P,P,w,\widetilde  u,M,m)$ is a weak solution of \eqref{e.MFGstochBisbody}. Properties \eqref{def.Pspace}, \eqref{def.hjb}, and \eqref{def.boundary} are true by construction, and \eqref{def.mw} is a consequence of Lemma \ref{L:Pnlimit}. It remains to show that the process $m$ solves, $P$-a.s., the Fokker-Planck equation associated to $\widetilde  u$. To do this, it suffices to show that, for all $\phi \in C^1_c((0,T) \times \RR^d)$ and $A \in \mbb F^P_T$,
\begin{align*}
	I_A:= \int_A \int_0^T \int_{\RR^d} m_t(x) &\bigg[ \del_t \phi_t(x) + \tr \widetilde  a_t(x,\omega) D^2 \phi_t(x) \\
	&- D\phi_t(x) \cdot D_p \widetilde  H_t(x, D\widetilde  u_t(x,\omega),\omega) \bigg] dxdt dP(\omega) = 0.
\end{align*}
The canonical process $(m,w) \mapsto m$ satisfies $P^n$-a.s. the equation
\begin{equation}\label{mWRTPn}
	\left\{
	\begin{split}
	&d m_{t} =  \div \bigl[ \div (\widetilde   a^n_{t}(x,\omega) m_{t})+  m_{t} D_{p} \widetilde   H^n_t(x,D\widetilde  u^{n}_{t},\omega) 
 \bigr] dt \ \ \text{in} \ \ (0,T)\times  \R^d, \\
 	&m_0=\bar m_0 \ \ \text{in} \ \   \R^d.
	\end{split}
	\right.
\end{equation}
Let $A' \in \mbb F_T$ and a $P$-null set $N \subset \Omega$ be such that $A = A' \cup N$. We thus have
\begin{align*}
	I^n_{A'} := \int_{A'} \int_0^T \int_{\RR^d} m_t(x) &\bigg[ \del_t \phi_t(x) + \tr \widetilde  a^n_t(x,\omega) D^2 \phi_t(x)\\
	& - D\phi_t(x) \cdot D_p \widetilde  H^n_t(x, D\widetilde  u^n_t(x,\omega),\omega) \bigg] dxdt dP^n(\omega) = 0.
\end{align*}
Let $R > 0$ be such that $\supp \phi \subset [0,T] \times B_R$. By \eqref{dataclose}, there exists a constant $C > 0$ such that, for all $n \in \NN$,
\begin{align*}
	|I_{A'}| = |I_{A'} - I^n_{A'}| 
	&\le C \norm{\phi}_{C^1_t C^2_x} \Bigg( d_{\mcl P(\Omega^*)}(P^n,P) + \EE^0 \norm{\pi_n(w) - w}_{\oo,[0,T]} \\
	&\quad +\int_{\Omega^* \times \Omega^*} \int_0^T \int_{B_R} |D\widetilde  u^n_t(x,\omega) - D \widetilde  u_t(x,\omega') | d\gamma^n(\omega,\omega') \Bigg),
\end{align*}
where, as in Lemma \ref{lem.stabilo}, $\gamma^n$ is an optimal coupling between $P^n$ and $P$. Sending $n \to \oo$ along the appropriate subsequence, we conclude from \eqref{piecewiseWn}, Lemma \ref{L:Pnlimit}, and Lemma \ref{lem.stabilo} that $I_{A'} = 0$, and therefore $I_{A} = I_{A' \cup N} = 0$, because $P(N) = 0$.

\end{proof}

\subsection{Uniqueness}
A weak solution in the sense of Definition \ref{def.weaksolMGF} can be identified with the law of $(\widetilde  m, W)$ on the probability space $\Omega^*$ defined in \eqref{measurepathspace}. Indeed, $(\widetilde  u,M)$ can then be recovered by solving the backward stochastic HJB equation with the theory from Section \ref{sec.HJ}. This motivates the following definition.

\begin{definition}\label{def:uniqueness}
	Uniqueness in law holds for \eqref{e.MFGstochBisbody} if, for any two weak solutions $(\Omega_i,\mbb F_i, \mbb P_i, W_i, \widetilde  u_i, M_i, \widetilde  m_i)_{i=1,2}$, we have
	\[
		(\widetilde  m_1,W_1)_\sharp \mbb P_1 = (\widetilde  m_2,W_2)_\sharp \mbb P_2 \quad \text{in } \mcl P(\Omega^*).
	\]
	Pathwise uniqueness holds if, for two solutions $(\widetilde  u_i,M_i,\widetilde  m_i)_{i=1,2}$ defined on the same probability space with a given Wiener process $W$, we have $\widetilde  u_1 = \widetilde  u_2$, $M_1 = M_2$, and $\widetilde  m_1 = \widetilde  m_2$ with probability one.
\end{definition}

Under further structural assumptions, namely, the well-known Lasry-Lions monotonicity conditions \cite{LLJapan}, we now establish pathwise uniqueness for the system \eqref{e.MFGstochBisbody}.

In addition to \eqref{A:stochMFG} and \eqref{A:barm0MFG}, we assume that (abusing notation)
\begin{equation}\label{A:diffusionind}
	a_t(x,m) = a_t(x) \text{ is independent of }m \in \mcl P_2(\RR^d),
\end{equation}
\begin{equation}\label{A:separatedH}
	H_t(x,p,m) = H_t(x,p) - F_t(x,m) \quad \text{for } (t,x,p,m) \in [0,T] \times \RR^d \times \RR^d \times \mcl P_2(\RR^d),
\end{equation}
and the coupling functions $F$ and $G$ are strictly monotone, that is,
\begin{equation}\label{A:FGmonotone}
	\left\{
	\begin{split}
	&\text{for all } m,m' \in \mcl P_2(\RR^d) \text{ and } t \in [0,T],\\
	&\int_{\R^d}( F_t(x,m) - F_t(x,m'))(m-m')(dx)\geq 0 \text{ and}\\
	&\int_{\R^d}( G(x,m)- G(x,m'))(m-m')(dx)\geq0, \\
	&\text{with equality in either holding if and only if } m=m'.
	\end{split}
	\right.
\end{equation} 
Writing, for $(t,x,m) \in [0,T] \times \RR^d \times \mcl P_2(\RR^d)$,
\begin{equation}\label{tildeF}
	\widetilde  F_t(x,m) := F_t(x + \sqrt{2\beta} W_t, (\Id + \sqrt{2\beta} W_t)_\sharp m),
\end{equation}
the MFG system \eqref{e.MFGstochBisbody} then becomes 
\begin{equation}\label{e.MFGstochBisMono}
	\left\{
	\begin{split}
 	&d \widetilde  u_{t} = \bigl[ -\tr(\widetilde  a_t(x) D^2 \widetilde  u_t)+ \widetilde  H_t(x, D\widetilde  u_{t})- \widetilde  F_t(x, \widetilde  m_{t})  \bigr] dt + dM_{t} \quad \text{in } (0,T) \times \R^d, \\
	&d \widetilde  m_{t} = \div\bigl[ \div(\widetilde  a_{t}(x) \widetilde  m_{t})+ \widetilde  m_{t} D_{p} \widetilde  H_t(x, D\widetilde  u_{t})
\bigr] dt \quad\text{in } (0,T) \times \R^d,\\
	&\widetilde  u_T(x) = \widetilde  G(x,\widetilde  m_T), \ \ \   \widetilde  m_{0}= \bar m_0 \ \ \text{in} \ \  \R^d.
	\end{split}
	\right.
\end{equation}

\begin{theorem}\label{T:MFGuniqueness}
Assume \eqref{A:stochMFG}, \eqref{A:barm0MFG}, and \eqref{A:diffusionind} - \eqref{A:FGmonotone}.
Then pathwise uniqueness holds for \eqref{e.MFGstochBisMono} in the sense of Definition \ref{def:uniqueness}, and, moreover, every weak solution is a strong solution.
\end{theorem}

The argument is an adaptation of those in \cite{LL06cr2,LLJapan}, similar to \cite{CSmfg}. Informally,
one would like to compute
\[
	d\int_{\RR^d} (\widetilde  u^1_{t}(x) - \widetilde  u^2_{t}(x)) (\widetilde  m^{1}_{t}(x) - \widetilde  m^{2}_{t}(x))dx,
\]
take expectations to cancel the martingale terms, and invoke the convexity of $H$ and the monotonicity of $F$ and $G$ to show that $\widetilde  m^1 = \widetilde  m^2$. The main subtlety, which does not arise in the first order setting of \cite{CSmfg}, is the possible discontinuity of $\widetilde  m^1$ and $\widetilde  m^2$ in the state variable, which prevents them from being used as test functions in the distributional equality in Definition \ref{D:solution}. At the same time, $\widetilde  u^1$ and $\widetilde  u^2$ fail to be $C^{1,1}$ in the state variable in general, and so cannot be used as test functions for the Fokker-Planck equation.

We overcome this issue by introducing a standard mollifier and cutoff, and the resulting commutators vanish in the limit precisely because $D^2\widetilde  u^1$ and $D^2\widetilde  u^2$ are locally finite Radon measures.

\begin{proof}[Proof of Theorem \ref{T:MFGuniqueness}]

It is enough to prove the pathwise uniqueness, since then, arguing as in the classic paper of Yamada and Watanabe \cite{yamadawatanabe1971} (see also \cite[Proposition 6.1]{carmona2016mean}), invoking the existence of a weak solution by Theorem \ref{T.ExistenceMFG}, every solution is strong.

Let $(\Omega,\mbb F, \mbb P, W,\widetilde  u^i,M^i, \widetilde  m^i)_{i=1,2}$ be two weak solutions, defined on the same probability space with the same Wiener process $W$. Let $(\rho_\delta)_{\delta > 0}$ be as in \eqref{mollifier} and let $\eta \in C^\oo_c(\RR^d)$ be a standard cut-off function equal to $1$ in $B_1$ and $0$ outside of $B_2$, set $\eta_R(x) = \eta(x/R)$ for $x \in \RR^d$, $R > 0$, and define
\[
	\widetilde  m^{i,\delta} := \widetilde  m^i * \rho_\delta \quad \text{and} \quad \widetilde  m^{i,\delta,R} := \widetilde  m^i_\delta \eta_R, \quad i = 1,2.
\]
Then, for $i = 1,2$ and $\delta, R > 0$,
\begin{equation}\label{smoothcutoffm}
	\begin{split}
	\del_t \widetilde  m^{i,\delta,R}_t = \div \div \left[ \widetilde  a_t(x)\widetilde  m^{i,\delta,R}_t \right] + \eta_R(x) &\div \left[ \rho_\delta * \left( D_p \widetilde  H_t(\cdot,D\widetilde  u^i_t) \widetilde  m^i_t \right)\right]\\
	&+ \alpha_t^{i,\delta,R}(x) + \beta_t^{i,\delta,R}(x),
	\end{split}
\end{equation}
where
\[
	\alpha_t^{i,\delta,R}(x) := \eta_R(x) \div^2_x \int_{\RR^d} [\widetilde  a_t(y) - \widetilde  a_t(x)] \widetilde  m^i_t(y) \rho_\delta(x-y)dy
\]
and
\[
	\beta_t^{i,\delta,R}(x) := -  \div\left[ \widetilde  a_t(x)\widetilde  m^{i,\delta}_t\right]\cdot D\eta_R(x) - \widetilde  m_t^{i,\delta} \tr\left[ \widetilde  a_t(x) D^2\eta_R(x) \right].
\]
Set $\alpha^{\delta,R} = \alpha^{1,\delta,R} - \alpha^{2,\delta,R}$ and $\beta^{\delta,R} = \beta^{1,\delta,R} - \beta^{2,\delta,R}$. 

We then claim that
\begin{equation}\label{approxLL}
\begin{split}
	&\EE \int_{\RR^d} \left( \widetilde  G(x,\widetilde  m^1_T) - \widetilde  G(x,\widetilde  m^2_T) \right) (\widetilde  m^{1,\delta,R}_T(x) - \widetilde  m^{2,\delta,R}_T(x))dx\\
	&+ \EE \int_0^T \int_{\RR^d} \left( \widetilde  F_t(x,\widetilde  m^1_t) - \widetilde  F_t(x,\widetilde  m^2_t) \right) (\widetilde  m^{1,\delta,R}_t(x) - \widetilde  m^{2,\delta,R}_t(x)) dxdt\\
	& + \EE \int_0^T \int_{\RR^d} \rho_\delta(x) * D( \eta_R(\widetilde  u^1_t - \widetilde  u^2_t) )(x) \bigg[ \widetilde  m^1_t(x) D_p \widetilde  H_t(x,D\widetilde  u^1_t(x))\\
	&\hspace{6cm}- \widetilde  m^2_t(x) D_p \widetilde  H_t(x,D\widetilde  u^2_t(x)) \bigg] dxdt\\
	&- \EE \int_0^T \int_{\RR^d} \left( \widetilde  H_t(x,D\widetilde  u^1_t(x)) - \widetilde  H_t(x,D\widetilde  u^2_t(x)) \right) (\widetilde  m^{1,\delta,R}_t(x) - \widetilde  m^{2,\delta,R}_t(x) ) dxdt\\
	&= \EE\int_0^T \int_{\R^d} (\widetilde  u^1_t(x) - \widetilde  u^2_t(x)) (\alpha^{\delta,R}_t(x) + \beta^{\delta,R}_t(x) ) dxdt.
\end{split}
\end{equation}
We omit the full details of the argument, which is similar to \cite[Theorem 3.3]{CSmfg}. In short, for a fine partition $\{0 = t_0 < t_1 < t_2 < \cdots < t_N = T\}$ of $[0,T]$, the increments
\begin{align*}
	\int_{\RR^d} &(\widetilde  u^1_{t_{i+1}}(x) - \widetilde  u^2_{t_{i+1}}(x)) (\widetilde  m^{1,\delta,R}_{t_{i+1}}(x) - \widetilde  m^{2,\delta,R}_{t_{i+1}}(x))dx\\
	&-
	\int_{\RR^d} (\widetilde  u^1_{t_{i}}(x) - \widetilde  u^2_{t_{i}}(x)) (\widetilde  m^{1,\delta,R}_{t_{i}}(x) - \widetilde  m^{2,\delta,R}_{t_{i}}(x))dx	
\end{align*}
are decomposed, using Definition \ref{D:solution} and \eqref{smoothcutoffm}; note that $\widetilde  m^{1,\delta,R}, \widetilde  m^{2,\delta,R}$ are smooth and compactly supported in the state variable. Taking the expectation, exploiting the martingale property of $M^1$ and $M^2$, and sending the size of the partition to $0$ then yields \eqref{approxLL}.

We have $\widetilde  m^1,\widetilde  m^2 \in L^\oo(\Omega, (L^\oo \cap L^1)([0,T] \times \RR^d))$ and $D\widetilde  u^1,D\widetilde  u^2 \in L^\oo(\Omega \times [0,T] \times \RR^d)$, as well as $\norm{D\eta_R}_\oo \le C R^{-1}$ for some universal constant $C > 0$, and so, sending first $\delta \to 0$ and then $R \to \oo$, the left-hand side of \eqref{approxLL} becomes
\begin{align*}
	&\EE \int_{\RR^d} \left( \widetilde  G(x,\widetilde  m^1_T) - \widetilde  G(x,\widetilde  m^2_T) \right) (\widetilde  m^{1}_T(x) - \widetilde  m^{2}_T(x))dx\\
	&+ \EE \int_0^T \int_{\RR^d} \left( \widetilde  F_t(x,\widetilde  m^1_t) - \widetilde  F_t(x,\widetilde  m^2_t) \right) (\widetilde  m^{1}_t(x) - \widetilde  m^{2}_t(x)) dxdt\\
	& + \EE \int_0^T \int_{\RR^d} (D\widetilde  u^1_t(x) - D\widetilde  u^2_t(x)) \bigg[ \widetilde  m^1_t(x) D_p \widetilde  H_t(x,D\widetilde  u^1_t(x)) \\
	&\hspace{6cm}- \widetilde  m^2_t(x) D_p \widetilde  H_t(x,D\widetilde  u^2_t(x)) \bigg] dxdt\\
	&- \EE \int_0^T \int_{\RR^d} \left(  \widetilde  H_t(x,D\widetilde  u^1_t(x)) - \widetilde  H_t(x,D\widetilde  u^2_t(x)) \right) (\widetilde  m^{1}_t(x) - \widetilde  m^{2}_t(x) ) dxdt\\
	&\ge \EE \int_{\RR^d} \left( \widetilde  G(x,\widetilde  m^1_T) - \widetilde  G(x,\widetilde  m^2_T) \right) (\widetilde  m^{1}_T(x) - \widetilde  m^{2}_T(x))dx\\
	&+ \EE \int_0^T \int_{\RR^d} \left( \widetilde  F_t(x,\widetilde  m^1_t) - \widetilde  F_t(x,\widetilde  m^2_t) \right) (\widetilde  m^{1}_t(x) - \widetilde  m^{2}_t(x)) dxdt,
\end{align*}
where the inequality follows from the convexity of $\widetilde  H$.

Define $\mu := D^2 (\widetilde  u^1_t - \widetilde  u^2_t)$, which belongs to $L^\oo([0,T] \times \Omega, \mcl M_\loc(\RR^d) \cap W^{-1,\oo}(\RR^d))$ in view of Lemma \ref{L:SCprops}\eqref{L:TVbound} and the uniform Lipschitz and semiconcavity constants for $\widetilde  u^1$ and $\widetilde  u^2$. Then
\begin{align*}
	&\int_0^T  \int_{\RR^d} (\widetilde  u^1_t(x) - \widetilde  u^2_t(x)) \alpha^{\delta,R}_t(x) dxdt\\
	&= \int_0^T \Bigg[ \tr \left\langle \mu_t , \eta_R \int_{\RR^d} (\widetilde  a_t(y) - \widetilde  a_t(\cdot))(\widetilde  m^1_t(y) - \widetilde  m^2_t(y))\rho_\delta(\cdot - y)dy \right \rangle \\
	&- \int_{\RR^d} (D\widetilde  u^1_t(x) - D\widetilde  u^2_t(x)) \\
	&\qquad \cdot \int_{\RR^d} (\widetilde  a_t(y) - \widetilde  a_t(x))(\widetilde  m^1_t(y) - \widetilde  m^2_t(y))\rho_\delta(x - y)dyD\eta_R(x) dx \\
	&+ \int_{\RR^d} (\widetilde  u^1_t(x) - \widetilde  u^2_t(x))\\
	&\qquad \cdot \tr\bigg[  D^2 \eta_R(x) \int_{\RR^d}(\widetilde  a_t(y) - \widetilde  a_t(x))(\widetilde  m^1_t(y) - \widetilde  m^2_t(y))\rho_\delta(x - y)dy \bigg] dx \Bigg]dt.
\end{align*}
For fixed $R > 0$, the continuous function
\[
	x \mapsto \eta_R(x) \int_{\RR^d} (\widetilde  a_t(y) - \widetilde  a_t(x))(\widetilde  m^1_t(y) - \widetilde  m^2_t(y))\rho_\delta(x - y)dy
\]
is supported in $B_R$, independently of $\delta$, and, as $\delta \to 0$, it converges uniformly to $0$. It follows that the pairing of $\mu$ with this function vanishes as $\delta \to 0$, and, therefore,
\[
	\lim_{\delta \to 0} \int_0^T  \int_{\RR^d} (\widetilde  u^1_t(x) - \widetilde  u^2_t(x)) \alpha^{\delta,R}_t(x) dx = 0.
\]
We also have
\begin{align*}
	\int_{\RR^d} &(\widetilde  u^1_t(x) -  \widetilde  u^2_t(x)) \beta^{\delta,R}_t(x)dx\\
	&= \int_{\RR^d} \left( \widetilde  m^{1,\delta}_t(x) - \widetilde  m^{2,\delta}_t(x) \right) \widetilde  a_t(x) D\eta_R(x) \cdot (D\widetilde  u^1_t(x) - D\widetilde  u^2_t(x))dx, 
\end{align*}
and so, for some constant $C > 0$ independent of $\delta$ and $R$, with probability one,
\[
	\left| \int_0^T \int_{\RR^d} (\widetilde  u^1_t(x) - \widetilde  u^2_t(x)) \beta^{\delta,R}_t(x)dx dt \right|
	\le \frac{C}{R}.
\]
We conclude, upon sending first $\delta \to 0$ and then $R \to 0$ on both sides of \eqref{approxLL}, that
\begin{align*}
	&\EE \int_{\RR^d} \left( \widetilde  G(x,\widetilde  m^1_T) - \widetilde  G(x,\widetilde  m^2_T) \right) (\widetilde  m^{1}_T(x) - \widetilde  m^{2}_T(x))dx\\
	&+ \EE \int_0^T \int_{\RR^d} \left( \widetilde   F_t(x,\widetilde  m^1_t) - \widetilde  F_t(x,\widetilde  m^2_t) \right) (\widetilde  m^{1}_t(x) - \widetilde  m^{2}_t(x)) dxdt
	\le 0.
\end{align*}
The strict monotonicity stated in \eqref{A:FGmonotone} for $G$ and $F$ implies that $\widetilde  G$ and $\widetilde  F$ are strictly monotone, and so we conclude that $\widetilde  m^1 = \widetilde  m^2$. The proof is finished in view of Theorem \ref{T:comparison}, which implies that $\widetilde  u^1 = \widetilde  u^2$ and $M^1 = M^2$.
\end{proof}

%
%

\begin{funding}
 The first author was partially supported by the third author's Air Force Office for Scientific Research grant FA9550-18-1-0494. 
 
 The second author was partially supported by the National Science Foundation grants DMS-1902658, DMS-1840314, and DMS-2437066.
 
The third author was partially supported  by the National Science Foundation grant DMS-1900599, the Office for Naval Research grant N000141712095 and the Air Force Office for Scientific Research grant FA9550-18-1-0494. 

All three authors were partially supported by the Institute for Mathematical and Statistical Innovations during the Fall of 2021.
\end{funding}



\bibliographystyle{imsart-number} 
\bibliography{cssMFGbib}       

\begin{thebibliography}{34}

\bibitem{bayraktar2019finite}
\begin{barticle}[author]
\bauthor{\bsnm{Bayraktar},~\bfnm{Erhan}\binits{E.}},
  \bauthor{\bsnm{Cecchin},~\bfnm{Alekos}\binits{A.}},
  \bauthor{\bsnm{Cohen},~\bfnm{Asaf}\binits{A.}} \AND
  \bauthor{\bsnm{Delarue},~\bfnm{Fran\c{c}ois}\binits{F.}}
(\byear{2021}).
\btitle{Finite state mean field games with {W}right-{F}isher common noise}.
\bjournal{J. Math. Pures Appl. (9)}
\bvolume{147}
\bpages{98--162}.
\bdoi{10.1016/j.matpur.2021.01.003}
\bmrnumber{4213680}
\end{barticle}
\endbibitem

\bibitem{Be20}
\begin{barticle}[author]
\bauthor{\bsnm{Bertucci},~\bfnm{Charles}\binits{C.}}
(\byear{2021}).
\btitle{Monotone solutions for mean field games master equations: finite state
  space and optimal stopping}.
\bjournal{J. \'{E}c. polytech. Math.}
\bvolume{8}
\bpages{1099--1132}.
\bdoi{10.5802/jep.167}
\bmrnumber{4275225}
\end{barticle}
\endbibitem

\bibitem{Be21}
\begin{barticle}[author]
\bauthor{\bsnm{Bertucci},~\bfnm{Charles}\binits{C.}}
(\byear{2023}).
\btitle{Monotone solutions for mean field games master equations: continuous
  state space and common noise}.
\bjournal{Comm. Partial Differential Equations}
\bvolume{48}
\bpages{1245--1285}.
\bdoi{10.1080/03605302.2023.2276564}
\bmrnumber{4687277}
\end{barticle}
\endbibitem

\bibitem{bertucci2019some}
\begin{barticle}[author]
\bauthor{\bsnm{Bertucci},~\bfnm{Charles}\binits{C.}},
  \bauthor{\bsnm{Lasry},~\bfnm{Jean-Michel}\binits{J.-M.}} \AND
  \bauthor{\bsnm{Lions},~\bfnm{Pierre-Louis}\binits{P.-L.}}
(\byear{2019}).
\btitle{Some remarks on mean field games}.
\bjournal{Comm. Partial Differential Equations}
\bvolume{44}
\bpages{205--227}.
\bdoi{10.1080/03605302.2018.1542438}
\bmrnumber{3941633}
\end{barticle}
\endbibitem

\bibitem{BeLaLi20}
\begin{barticle}[author]
\bauthor{\bsnm{Bertucci},~\bfnm{Charles}\binits{C.}},
  \bauthor{\bsnm{Lasry},~\bfnm{Jean-Michel}\binits{J.-M.}} \AND
  \bauthor{\bsnm{Lions},~\bfnm{Pierre-Louis}\binits{P.-L.}}
(\byear{2021}).
\btitle{Master equation for the finite state space planning problem}.
\bjournal{Arch. Ration. Mech. Anal.}
\bvolume{242}
\bpages{327--342}.
\bdoi{10.1007/s00205-021-01687-8}
\bmrnumber{4302761}
\end{barticle}
\endbibitem

\bibitem{CaSiBook}
\begin{bbook}[author]
\bauthor{\bsnm{Cannarsa},~\bfnm{Piermarco}\binits{P.}} \AND
  \bauthor{\bsnm{Sinestrari},~\bfnm{Carlo}\binits{C.}}
(\byear{2004}).
\btitle{Semiconcave functions, {H}amilton-{J}acobi equations, and optimal
  control}.
\bseries{Progress in Nonlinear Differential Equations and their Applications}
\bvolume{58}.
\bpublisher{Birkh\"{a}user Boston, Inc., Boston, MA}.
\bmrnumber{2041617}
\end{bbook}
\endbibitem

\bibitem{cardaliaguet2020splitting}
\begin{barticle}[author]
\bauthor{\bsnm{Cardaliaguet},~\bfnm{Pierre}\binits{P.}},
  \bauthor{\bsnm{Cirant},~\bfnm{Marco}\binits{M.}} \AND
  \bauthor{\bsnm{Porretta},~\bfnm{Alessio}\binits{A.}}
(\byear{2023}).
\btitle{Splitting methods and short time existence for the master equations in
  mean field games}.
\bjournal{J. Eur. Math. Soc. (JEMS)}
\bvolume{25}
\bpages{1823--1918}.
\bdoi{10.4171/jems/1227}
\bmrnumber{4592861}
\end{barticle}
\endbibitem

\bibitem{CDLL}
\begin{bbook}[author]
\bauthor{\bsnm{Cardaliaguet},~\bfnm{Pierre}\binits{P.}},
  \bauthor{\bsnm{Delarue},~\bfnm{Fran\c{c}ois}\binits{F.}},
  \bauthor{\bsnm{Lasry},~\bfnm{Jean-Michel}\binits{J.-M.}} \AND
  \bauthor{\bsnm{Lions},~\bfnm{Pierre-Louis}\binits{P.-L.}}
(\byear{2019}).
\btitle{The master equation and the convergence problem in mean field games}.
\bseries{Annals of Mathematics Studies}
\bvolume{201}.
\bpublisher{Princeton University Press, Princeton, NJ}.
\bdoi{10.2307/j.ctvckq7qf}
\bmrnumber{3967062}
\end{bbook}
\endbibitem

\bibitem{CSMaster}
\begin{barticle}[author]
\bauthor{\bsnm{Cardaliaguet},~\bfnm{Pierre}\binits{P.}} \AND
  \bauthor{\bsnm{Souganidis},~\bfnm{Panagiotis}\binits{P.}}
(\byear{2022}).
\btitle{Monotone solutions of the master equation for mean field games with
  idiosyncratic noise}.
\bjournal{SIAM J. Math. Anal.}
\bvolume{54}
\bpages{4198--4237}.
\bdoi{10.1137/21M1450008}
\bmrnumber{4451309}
\end{barticle}
\endbibitem

\bibitem{CSmfg}
\begin{barticle}[author]
\bauthor{\bsnm{Cardaliaguet},~\bfnm{Pierre}\binits{P.}} \AND
  \bauthor{\bsnm{Souganidis},~\bfnm{Panagiotis~E.}\binits{P.~E.}}
(\byear{2022}).
\btitle{On first order mean field game systems with a common noise}.
\bjournal{Ann. Appl. Probab.}
\bvolume{32}
\bpages{2289--2326}.
\bdoi{10.1214/21-aap1734}
\bmrnumber{4430014}
\end{barticle}
\endbibitem

\bibitem{CaDeBook}
\begin{bbook}[author]
\bauthor{\bsnm{Carmona},~\bfnm{Ren\'{e}}\binits{R.}} \AND
  \bauthor{\bsnm{Delarue},~\bfnm{Fran\c{c}ois}\binits{F.}}
(\byear{2018}).
\btitle{Probabilistic theory of mean field games with applications. {II}}.
\bseries{Probability Theory and Stochastic Modelling}
\bvolume{84}.
\bpublisher{Springer, Cham}
\bnote{Mean field games with common noise and master equations}.
\bmrnumber{3753660}
\end{bbook}
\endbibitem

\bibitem{carmona2016mean}
\begin{barticle}[author]
\bauthor{\bsnm{Carmona},~\bfnm{Ren\'{e}}\binits{R.}},
  \bauthor{\bsnm{Delarue},~\bfnm{Fran\c{c}ois}\binits{F.}} \AND
  \bauthor{\bsnm{Lacker},~\bfnm{Daniel}\binits{D.}}
(\byear{2016}).
\btitle{Mean field games with common noise}.
\bjournal{Ann. Probab.}
\bvolume{44}
\bpages{3740--3803}.
\bdoi{10.1214/15-AOP1060}
\bmrnumber{3572323}
\end{barticle}
\endbibitem

\bibitem{CG_16}
\begin{barticle}[author]
\bauthor{\bsnm{Catellier},~\bfnm{R.}\binits{R.}} \AND
  \bauthor{\bsnm{Gubinelli},~\bfnm{M.}\binits{M.}}
(\byear{2016}).
\btitle{Averaging along irregular curves and regularisation of {ODE}s}.
\bjournal{Stochastic Process. Appl.}
\bvolume{126}
\bpages{2323--2366}.
\bdoi{10.1016/j.spa.2016.02.002}
\bmrnumber{3505229}
\end{barticle}
\endbibitem

\bibitem{D_07}
\begin{barticle}[author]
\bauthor{\bsnm{Davie},~\bfnm{A.~M.}\binits{A.~M.}}
(\byear{2007}).
\btitle{Uniqueness of solutions of stochastic differential equations}.
\bjournal{Int. Math. Res. Not. IMRN}
\bvolume{24}
\bpages{Art. ID rnm124, 26}.
\bdoi{10.1093/imrn/rnm124}
\bmrnumber{2377011}
\end{barticle}
\endbibitem

\bibitem{djete21}
\begin{barticle}[author]
\bauthor{\bsnm{Djete},~\bfnm{Mao~Fabrice}\binits{M.~F.}}
(\byear{2023}).
\btitle{Large population games with interactions through controls and common
  noise: convergence results and equivalence between open-loop and closed-loop
  controls}.
\bjournal{ESAIM Control Optim. Calc. Var.}
\bvolume{29}
\bpages{Paper No. 39, 42}.
\bdoi{10.1051/cocv/2023005}
\bmrnumber{4598796}
\end{barticle}
\endbibitem

\bibitem{Do61}
\begin{barticle}[author]
\bauthor{\bsnm{Douglis},~\bfnm{Avron}\binits{A.}}
(\byear{1961}).
\btitle{The continuous dependence of generalized solutions of non-linear
  partial differential equations upon initial data}.
\bjournal{Comm. Pure Appl. Math.}
\bvolume{14}
\bpages{267--284}.
\bdoi{10.1002/cpa.3160140307}
\bmrnumber{139848}
\end{barticle}
\endbibitem

\bibitem{EvBook}
\begin{bbook}[author]
\bauthor{\bsnm{Evans},~\bfnm{Lawrence~C.}\binits{L.~C.}}
(\byear{2010}).
\btitle{Partial differential equations},
\bedition{second} ed.
\bseries{Graduate Studies in Mathematics}
\bvolume{19}.
\bpublisher{American Mathematical Society, Providence, RI}.
\bdoi{10.1090/gsm/019}
\bmrnumber{2597943}
\end{bbook}
\endbibitem

\bibitem{EG}
\begin{bbook}[author]
\bauthor{\bsnm{Evans},~\bfnm{Lawrence~C.}\binits{L.~C.}} \AND
  \bauthor{\bsnm{Gariepy},~\bfnm{Ronald~F.}\binits{R.~F.}}
(\byear{2015}).
\btitle{Measure theory and fine properties of functions},
\bedition{revised} ed.
\bseries{Textbooks in Mathematics}.
\bpublisher{CRC Press, Boca Raton, FL}.
\bmrnumber{3409135}
\end{bbook}
\endbibitem

\bibitem{FGP_10}
\begin{barticle}[author]
\bauthor{\bsnm{Flandoli},~\bfnm{F.}\binits{F.}},
  \bauthor{\bsnm{Gubinelli},~\bfnm{M.}\binits{M.}} \AND
  \bauthor{\bsnm{Priola},~\bfnm{E.}\binits{E.}}
(\byear{2010}).
\btitle{Well-posedness of the transport equation by stochastic perturbation}.
\bjournal{Invent. Math.}
\bvolume{180}
\bpages{1--53}.
\bdoi{10.1007/s00222-009-0224-4}
\bmrnumber{2593276}
\end{barticle}
\endbibitem

\bibitem{Fl69}
\begin{barticle}[author]
\bauthor{\bsnm{Fleming},~\bfnm{Wendell~H.}\binits{W.~H.}}
(\byear{1969}).
\btitle{The {C}auchy problem for a nonlinear first order partial differential
  equation}.
\bjournal{J. Differential Equations}
\bvolume{5}
\bpages{515--530}.
\bdoi{10.1016/0022-0396(69)90091-6}
\bmrnumber{235269}
\end{barticle}
\endbibitem

\bibitem{gangbo2021mean}
\begin{bunpublished}[author]
\bauthor{\bsnm{Gangbo},~\bfnm{W.}\binits{W.}},
  \bauthor{\bsnm{M\'{e}sz\'{a}ros},~\bfnm{A.~R.}\binits{A.~R.}},
  \bauthor{\bsnm{Mou},~\bfnm{C.}\binits{C.}} \AND
  \bauthor{\bsnm{Zhang},~\bfnm{J.}\binits{J.}}
\btitle{Mean Field Games Master Equations with Non-separable Hamiltonians and
  Displacement Monotonicity}.
\bnote{Preprint, arXiv:2101.12362 [math.AP]}.
\end{bunpublished}
\endbibitem

\bibitem{ishiilionssemiconc}
\begin{barticle}[author]
\bauthor{\bsnm{Ishii},~\bfnm{H.}\binits{H.}} \AND
  \bauthor{\bsnm{Lions},~\bfnm{P.~L.}\binits{P.~L.}}
(\byear{1990}).
\btitle{Viscosity solutions of fully nonlinear second-order elliptic partial
  differential equations}.
\bjournal{J. Differential Equations}
\bvolume{83}
\bpages{26--78}.
\bdoi{10.1016/0022-0396(90)90068-Z}
\bmrnumber{1031377}
\end{barticle}
\endbibitem

\bibitem{Kr60}
\begin{barticle}[author]
\bauthor{\bsnm{Kru\v{z}kov},~\bfnm{S.~N.}\binits{S.~N.}}
(\byear{1960}).
\btitle{The {C}auchy problem in the large for certain non-linear first order
  differential equations}.
\bjournal{Soviet Math. Dokl.}
\bvolume{1}
\bpages{474--477}.
\bmrnumber{0121575}
\end{barticle}
\endbibitem

\bibitem{LL06cr1}
\begin{barticle}[author]
\bauthor{\bsnm{Lasry},~\bfnm{Jean-Michel}\binits{J.-M.}} \AND
  \bauthor{\bsnm{Lions},~\bfnm{Pierre-Louis}\binits{P.-L.}}
(\byear{2006}).
\btitle{Jeux \`a champ moyen. {I}. {L}e cas stationnaire}.
\bjournal{C. R. Math. Acad. Sci. Paris}
\bvolume{343}
\bpages{619--625}.
\bdoi{10.1016/j.crma.2006.09.019}
\bmrnumber{2269875}
\end{barticle}
\endbibitem

\bibitem{LL06cr2}
\begin{barticle}[author]
\bauthor{\bsnm{Lasry},~\bfnm{Jean-Michel}\binits{J.-M.}} \AND
  \bauthor{\bsnm{Lions},~\bfnm{Pierre-Louis}\binits{P.-L.}}
(\byear{2006}).
\btitle{Jeux \`a champ moyen. {II}. {H}orizon fini et contr\^{o}le optimal}.
\bjournal{C. R. Math. Acad. Sci. Paris}
\bvolume{343}
\bpages{679--684}.
\bdoi{10.1016/j.crma.2006.09.018}
\bmrnumber{2271747}
\end{barticle}
\endbibitem

\bibitem{LLJapan}
\begin{barticle}[author]
\bauthor{\bsnm{Lasry},~\bfnm{Jean-Michel}\binits{J.-M.}} \AND
  \bauthor{\bsnm{Lions},~\bfnm{Pierre-Louis}\binits{P.-L.}}
(\byear{2007}).
\btitle{Mean field games}.
\bjournal{Jpn. J. Math.}
\bvolume{2}
\bpages{229--260}.
\bdoi{10.1007/s11537-007-0657-8}
\bmrnumber{2295621}
\end{barticle}
\endbibitem

\bibitem{LiCoursCollege}
\begin{bunpublished}[author]
\bauthor{\bsnm{Lions},~\bfnm{P.~L.}\binits{P.~L.}}
\btitle{Courses at the {C}oll\`{e}ge de {F}rance}.
\end{bunpublished}
\endbibitem

\bibitem{lionsHJB}
\begin{bincollection}[author]
\bauthor{\bsnm{Lions},~\bfnm{P.~L.}\binits{P.~L.}}
(\byear{1983}).
\btitle{Optimal control of diffusion processes and
  {H}amilton-{J}acobi-{B}ellman equations. {III}. {R}egularity of the optimal
  cost function}.
In \bbooktitle{Nonlinear partial differential equations and their applications.
  {C}oll\`ege de {F}rance seminar, {V}ol. {V} ({P}aris, 1981/1982)}.
\bseries{Res. Notes in Math.}
\bvolume{93}
\bpages{95--205}.
\bpublisher{Pitman, Boston, MA}.
\bmrnumber{725360}
\end{bincollection}
\endbibitem

\bibitem{lionsHJB2}
\begin{barticle}[author]
\bauthor{\bsnm{Lions},~\bfnm{P.~L.}\binits{P.~L.}}
(\byear{1983}).
\btitle{Optimal control of diffusion processes and
  {H}amilton-{J}acobi-{B}ellman equations. {II}. {V}iscosity solutions and
  uniqueness}.
\bjournal{Comm. Partial Differential Equations}
\bvolume{8}
\bpages{1229--1276}.
\bdoi{10.1080/03605308308820301}
\bmrnumber{709162}
\end{barticle}
\endbibitem

\bibitem{MoZh}
\begin{barticle}[author]
\bauthor{\bsnm{Mou},~\bfnm{Chenchen}\binits{C.}} \AND
  \bauthor{\bsnm{Zhang},~\bfnm{Jianfeng}\binits{J.}}
(\byear{2024}).
\btitle{Wellposedness of second order master equations for mean field games
  with nonsmooth data}.
\bjournal{Mem. Amer. Math. Soc.}
\bvolume{302}
\bpages{v+86}.
\bdoi{10.1090/memo/1515}
\bmrnumber{4813035}
\end{barticle}
\endbibitem

\bibitem{Pe92}
\begin{barticle}[author]
\bauthor{\bsnm{Peng},~\bfnm{Shi~Ge}\binits{S.~G.}}
(\byear{1992}).
\btitle{Stochastic {H}amilton-{J}acobi-{B}ellman equations}.
\bjournal{SIAM J. Control Optim.}
\bvolume{30}
\bpages{284--304}.
\bdoi{10.1137/0330018}
\bmrnumber{1149069}
\end{barticle}
\endbibitem

\bibitem{QIU18}
\begin{barticle}[author]
\bauthor{\bsnm{Qiu},~\bfnm{Jinniao}\binits{J.}}
(\byear{2018}).
\btitle{Viscosity solutions of stochastic {H}amilton-{J}acobi-{B}ellman
  equations}.
\bjournal{SIAM J. Control Optim.}
\bvolume{56}
\bpages{3708--3730}.
\bdoi{10.1137/17M1148232}
\bmrnumber{3864678}
\end{barticle}
\endbibitem

\bibitem{QiuWei}
\begin{barticle}[author]
\bauthor{\bsnm{Qiu},~\bfnm{Jinniao}\binits{J.}} \AND
  \bauthor{\bsnm{Wei},~\bfnm{Wenning}\binits{W.}}
(\byear{2019}).
\btitle{Uniqueness of viscosity solutions of stochastic {H}amilton-{J}acobi
  equations}.
\bjournal{Acta Math. Sci. Ser. B (Engl. Ed.)}
\bvolume{39}
\bpages{857--873}.
\bdoi{10.1007/s10473-019-0314-3}
\bmrnumber{4066509}
\end{barticle}
\endbibitem

\bibitem{yamadawatanabe1971}
\begin{barticle}[author]
\bauthor{\bsnm{Yamada},~\bfnm{Toshio}\binits{T.}} \AND
  \bauthor{\bsnm{Watanabe},~\bfnm{Shinzo}\binits{S.}}
(\byear{1971}).
\btitle{On the uniqueness of solutions of stochastic differential equations}.
\bjournal{J. Math. Kyoto Univ.}
\bvolume{11}
\bpages{155--167}.
\bdoi{10.1215/kjm/1250523691}
\bmrnumber{278420}
\end{barticle}
\endbibitem

\end{thebibliography}


\end{document}